\documentclass[10pt]{preprint}

\usepackage[margin=3.5cm]{geometry}
\usepackage[utf8x]{inputenc}
\usepackage[english]{babel}
\usepackage[full]{textcomp}
\usepackage[T1]{fontenc}
\usepackage[osf]{newtxtext}
\usepackage{appendix}
\usepackage{amssymb}
\usepackage{amsmath}
\usepackage{amsthm}

\usepackage{mhequ}
\usepackage{booktabs}
\usepackage{tikz}
\usepackage{mathrsfs}
\usepackage[noadjust]{cite}
\usepackage{microtype}
\usepackage{comment}
\usepackage{slashed}
\usepackage[showonlyrefs]{mathtools}
\usepackage{centernot}
\usepackage{footnote}
\usepackage{enumerate}
\usepackage[shortlabels]{enumitem}
\usepackage{stackrel}
\usepackage{longtable}
\usepackage{cprotect}
\usepackage{xstring}
\usepackage{calc}
\usepackage{bbm}
\usepackage{dsfont}

\usepackage[colorlinks=true, pdfstartview=FitV, linkcolor=colorLink, citecolor=colorCite, urlcolor=colorLink, linktocpage=true]{hyperref}

\usepackage[capitalize]{cleveref}

\usepackage{upgreek}
\usepackage{parskip}
\usepackage{graphicx}
\usepackage{orcidlink}

\makeatletter
\renewcommand{\paragraph}{%
\@startsection{paragraph}{4}%
{\z@}{1.5ex \@plus 1.5ex \@minus .2ex}{-0.7em}%
{\normalfont\normalsize\bfseries}%
}
\makeatother

\makeatletter
\def\thm@space@setup{%
  \thm@preskip=\parskip \thm@postskip=0pt
}
\makeatother

\setlist[itemize]{leftmargin=5mm}

\DeclareSymbolFont{timesoperators}{T1}{ptm}{m}{n}
\SetSymbolFont{timesoperators}{bold}{T1}{ptm}{b}{n}
\DeclareMathAlphabet{\mathbb}{U}{jkpsyb}{m}{n}
\SetMathAlphabet{\mathbb}{bold}{U}{jkpsyb}{bx}{n}

\allowdisplaybreaks
\definecolor{colorLink}{RGB}{0,100,162}
\definecolor{colorCite}{RGB}{8,124,100}

\newcommand{\R}{\mathbb{R}}
\newcommand{\f}[1]{\boldsymbol{#1}}
\newcommand{\sym}{\mathrm{sym}}
\newcommand{\de}{\mathrm{d}}
\newcommand{\tr}{\mathrm{tr}}
\newcommand{\tu}{\tilde{\f u}}
\DeclareMathOperator{\C}{\mathcal{C}}

\newcommand{\expect}[1]{\mathbb E \left[ #1 \right]}
\newcommand{\expectp}[1]{\mathbb E \left[\psi\left( #1\right) \right]}
\newcommand{\E}{\mathcal{E}}
\DeclareMathOperator{\esssup}{ess\,sup}
\usepackage{enumitem}
\newcommand{\D}{\mathcal{D}}
\newcommand{\N}{\mathbb{N}}
\newcommand{\ov}[1]{\overline{#1}}

\newcommand{\fs}[1]{\tilde{\f{#1}}}
\newcommand{\Tr}[1]{\mathrm{Tr}\left[#1 \right]}
\renewcommand{\P}{\mathbb{P}}
\newcommand{\dx}{\de x}
\newcommand{\ds}{\de s}
\newcommand{\DA}{D(\mathcal A)}
\newcommand{\A}{\mathcal A}
\usetikzlibrary{shapes.misc}
\usetikzlibrary{shapes.symbols}
\usetikzlibrary{shapes.geometric}
\usetikzlibrary{decorations}
\usetikzlibrary{decorations.markings}
\usetikzlibrary{calc}
\usetikzlibrary{external}
\usetikzlibrary{arrows}
\usetikzlibrary{patterns}

\theoremstyle{plain}

\newtheorem{theorem}{Theorem}[section]
\crefname{theorem}{Theorem}{Theorems}

\newtheorem{corollary}[theorem]{Corollary}
\crefname{corollary}{Corollary}{Corollaries}

\newtheorem{lemma}[theorem]{Lemma}
\crefname{lemma}{Lemma}{Lemmas}

\crefname{proposition}{Proposition}{Propositions}

\crefname{claim}{Claim}{Claims}

\crefname{theoremA}{Theorem}{Theorems}

\crefname{propositionA}{Proposition}{Propositions}

\theoremstyle{definition}

\newtheorem{definition}[theorem]{Definition}
\crefname{definition}{Definition}{Definitions}

\crefname{notation}{Notation}{Notations}

\crefname{acknowledgements}{Acknowledgements}{Acknowledgements}

\newtheorem{assumption}[theorem]{Assumption}
\crefname{assumption}{Assumption}{Assumptions}

\newtheorem{remark}[theorem]{Remark}
\crefname{remark}{Remark}{Remarks}

\crefname{observation}{Observation}{Observations}

\crefname{manualassumptioninner}{Assumption}{Assumptions}

\crefname{customthm}{Theorem}{Theorems}

\numberwithin{equation}{section}

\usepackage[textsize=scriptsize,linecolor=white,bordercolor=red,backgroundcolor=white]{todonotes}

\allowdisplaybreaks

\title{Probabilistically Strong 
Solutions to Stochastic Euler Equations}

\author{Benjamin Gess$^{1}$, Robert Lasarzik$^{2,3}$}

\institute{
    Institut f\"{u}r Mathematik, Technische Universit\"{a}t Berlin \& Max–Planck–Institute for Mathematics in the Sciences, Leipzig, Email: \href{mailto:benjamin.gess@tu-berlin.de}{\color{black} \texttt{benjamin.gess@tu-berlin.de}} 
    \and
    Friedrich--Alexander Universität Erlangen--Nürnberg, Erlangen \& Weierstrass Institute for Applied Analysis and Stochastics (WIAS), Berlin, Email: \href{mailto:robert.lasarzik@wias-berlin.de}{\color{black} \texttt{robert.lasarzik@wias-berlin.de}} 
   }

\begin{document}

\maketitle

\begin{abstract}
    In this paper, we establish the existence of probabilistically strong, measure-valued
 solutions for the stochastic incompressible Navier--Stokes equations  with additive and transport noise. Considering the vanishing viscosity limit, we prove the convergence of these solutions to probabilistically strong, dissipative weak solutions to the stochastic incompressible Euler equations with additive and transport noise. 
Introducing the concept of energy-variational solutions in the stochastic context we are able to treat the nonlinearities without changing the probability space. This resolves the open problem of constructing, in three dimensions, probabilistically, strong solutions of the stochastic Euler equations as vanishing-viscosity limits on the prescribed stochastic basis. In addition, the constructed solutions satisfy the energy inequality and general $L^2$ initial data is treated.  
\end{abstract}

\section{Introduction}

In this work, we present four main contributions in stochastic fluid dynamics. First, we establish the existence of probabilistically strong, measure-valued solutions to the three-dimensional stochastic incompressible Navier--Stokes equations. Second, we prove the convergence of these solutions to probabilistically strong, dissipative weak solutions of the incompressible 3D Euler equations in the vanishing viscosity limit. In particular, this yields the existence of probabilistically strong solutions to the Euler equations, for any initial condition and obtained as limits of the Navier--Stokes system. Third, we introduce the concept of energy-variational solutions to the field of stochastic fluid dynamics and demonstrate the versatility of this framework in establishing the aforementioned vanishing viscosity limit. Fourth, we demonstrate that the framework developed in this work is robust enough to cover not only additive noise but also transport noise, proving the convergence of Navier--Stokes equations driven by transport noise to probabilistically strong, dissipative weak solutions of the Euler system. We next discuss these contributions in turn. 

Since the foundational work of Flandoli and G\k{a}tarek~\cite{FlandoliGatarek}, there has been a plethora of contributions to the field of stochastic fluid dynamics extending this constructive approach. Indeed, it has become the predominant method for adapting deterministic compactness proofs to the stochastic context (see, e.g.,~\cite{breit2021dissipative, strongsol2} and references therein). This construction typically proceeds by first establishing uniform \textit{a priori} estimates on an approximating system, which yield finite-moment bounds on compactly embedded spaces. Subsequently, the Skorohod--Jakubowski representation theorem~\cite{Jakubowski_2} is invoked. Analogous to the deterministic case, this allows one to deduce almost sure convergence of the approximations to a limit. However, a critical feature of this construction is that the application of Skorohod--Jakubowski requires a change of the underlying probability space to handle the lack of strong compactness in the random variable, and the resulting possibility of oscillations. Consequently, these constructions lead to probabilistically weak solutions (or martingale solutions). This presents a significant restriction in settings where uniqueness is unknown, such as the 3D Navier--Stokes equations, as the Yamada--Watanabe argument cannot be utilized to recover probabilistically strong solutions.\\
The restriction to probabilistically weak solutions has received considerable attention in the literature, marking the existence of probabilistically strong, analytically weak solutions to the 3D stochastic Navier--Stokes equations as an open problem. In recent years, significant progress has been made via convex integration techniques, proving the existence of probabilistically strong solutions to both Navier--Stokes and Euler equations~\cite{hofmanova2023global,lu2025holder,lu2025proof,chen2024sharp,pappalettera2025dissipative}. These results come with a caveat: such solutions are generally ``wild'' and, for Navier--Stokes, cannot be obtained within the Leray--Hopf class, nor, for Euler, are they known to be obtainable 
 as vanishing viscosity limits from Navier--Stokes. \\
The first contribution of this work lies in the demonstration that the standard martingale approach is not the only avenue for existence theory in stochastic fluid dynamics. In fact, the concept of Young measures offers a tool to capture oscillations, thereby rendering the change of probability space unnecessary. While in the context of scalar conservation laws, where uniqueness is available and the Yamada–Watanabe principle can therefore be applied, the related concept of generalized kinetic solutions was used in \cite{debussche2010scalar} to carry out the compactness argument on a fixed probability space, this does not seem to have been exploited in stochastic fluid dynamics. In the context of the 3D Navier--Stokes equations, this unveils a distinction between the deterministic and stochastic theories: In the deterministic case, the concept of an analytically weak solution is strictly stronger than that of a measure-valued solution; thus, once analytically weak solutions are established, the measure-valued concept becomes less relevant. In the probabilistic context, however, there is no such strict hierarchy. Via standard compactness arguments, one faces a trade-off: one is led either to \textit{probabilistically weak, analytically weak} solutions (via Skorohod--Jakubowski) or, as shown in the present work, to \textit{probabilistically strong, measure-valued} solutions. The first result of this work makes this distinction rigorous by proving the existence of a joint probabilistically strong, measure-valued solution that after a change of probability space also is probabilistically weak, analytically weak.

The second main contribution of this work applies this observation to the stochastic incompressible 3D Euler equations. In this setting, the approach becomes particularly useful, as it allows to construct probabilistically strong,dissipative weak solutions for the stochastic 3D Euler equations, for any initial condition, satisfying an energy inequality, and obtained as limits of the stochastic Navier–Stokes equations; a result previously unknown. Specifically, Breit and Moyo~\cite{breit2021dissipative} constructed probabilistically weak (martingale) dissipative measure-valued solutions, leaving as an open problem 
the existence of  probabilistically strong solutions obtained as limits from Navier--Stokes.

The third main contribution is the transference of the concept of energy-variational solutions to the field of stochastic fluid dynamics. Originally developed for deterministic systems~\cite{EiterLas_2, EnvarIncomp}, these solutions are defined not by satisfying the equation in the sense of distributions directly, but via the variation of the energy dissipation. We show that in the stochastic setting, these energy-variational solutions imply dissipative measure-valued solutions based on defect measures (similar to the framework in~\cite{hofmanova2024nonuniqueness}). We further demonstrate the potential of the energy-variational framework, as it is shown to allow for an efficient proof of the convergence of Navier--Stokes to Euler, effectively bypassing many of the technical challenges associated with identifying nonlinear limits.

Finally, as a fourth contribution, we extend our results to the setting of transport noise. We establish the convergence of the stochastic Navier--Stokes equations driven by transport noise to probabilistically strong, dissipative weak solutions of the Euler equations. Transport noise is of fundamental physical significance as it models the advection of the fluid by unresolved, rapidly fluctuating small-scale features. This modeling paradigm finds its roots in the geometric variational approaches of Holm and collaborators (the SALT framework)~\cite{Holm2015}, which preserve the circulation properties of the fluid structure. Complementarily, the works of Flandoli~\cite{Flandoli2011} highlight the role of transport noise in the potential regularization of singularities and the enhancement of dissipation, while the foundational analysis of Mikulevicius and Rozovskii~\cite{MikRoz2004} provides the necessary stochastic analytic framework for these nonlinear SPDEs. Again, this result proves for the first time the existence of probabilistically strong, dissipative weak solutions for the stochastic Euler equations with transport noise for any initial condition, obtained as a vanishing viscosity limit from 3D stochastic Navier--Stokes equations.

\subsection*{Position in the literature}

The literature on stochastic fluid dynamics is extensive; we refer to the comprehensive monographs \cite{BreitFeireislHofmanova,FlandoliBook} for modern accounts of the field and to \cite{ChowBook,DaPratoZabczyk} for the broader analytic and probabilistic framework for infinite-dimensional stochastic equations. We concentrate here solely on the literature directly relevant to the present work. For local strong existence theory we refer to \cite{strongsol2} and references therein, for the two-dimensional setting to  \cite{KuksinShirikyan} and references therein.

\textit{Energy and compactness methods.}
The energy-estimate-based existence theory for stochastic fluid equations originates in the foundational compactness framework of Flandoli and Gatarek \cite{FlandoliGatarek}. Starting from suitable approximations, this approach combines uniform energy estimates, tightness, and a Skorokhod--Jakubowski representation in order to identify an analytically weak limit. Since the representation step changes the probability space and the driving Wiener process, the resulting solutions are, in general, probabilistically weak martingale solutions. In regimes where pathwise uniqueness is unavailable, as for the three-dimensional Navier--Stokes and Euler equations, this change of stochastic basis cannot subsequently be removed by a Yamada--Watanabe argument.

An important step towards a probabilistically strong existence theory was taken by Kim \cite{kim2015measure} who constructed measure-valued solutions to the stochastic Euler equations over a prescribed probability space by means of measurable selection. The measurable-selection argument does not, however, establish progressive measurability of the selected solutions and hence does not produce probabilistically strong solutions. The probabilistically weak part of Kim's work was subsequently extended by Breit and Moyo \cite{breit2021dissipative} who introduced a concept of dissipative martingale solutions to the three-dimensional stochastic Euler equations, described the nonlinear defects through generalized Young measures, and constructed such solutions as vanishing-viscosity limits of finite-energy stochastic Navier--Stokes solutions. Their refined theory retains the energy inequality and yields weak--strong uniqueness, both pathwise and in law. The powerful framework of Hofmanová, Zhu, and Zhu \cite{HZZ_2} further developed the theory of dissipative martingale solutions, including stability, Markov selection, and nonuniqueness in law. Under an additional Kolmogorov-type regularity hypothesis, Chaudhary and Vallet \cite{chaudhary2023short} also obtained a probabilistically weak inviscid limit with stronger convergence properties.

These works established the energetically natural global existence theory and, in particular, showed that dissipative Euler solutions can be selected by vanishing viscosity from Leray-solutions to Navier--Stokes. Their compactness arguments nevertheless lead to probabilistically weak / martingale solutions on a newly constructed stochastic basis. Thus, within the traditional energy-based approach, the existence of a global probabilistically strong solution to the three-dimensional stochastic Euler equations remained open.
\textit{Convex integration and the vanishing-viscosity problem.}
The development of stochastic convex integration fundamentally advanced the program to construct probabilistically strong solutions to both Navier--Stokes and Euler equations. In a key contribution, Hofmanová, Zhu, and Zhu \cite{HZZ_2,hofmanova2024nonuniqueness} produced probabilistically strong, analytically weak solutions to the stochastic three-dimensional Euler and Navier--Stokes equations. Subsequently, this was extended to global-in-time probabilistically strong solutions in several settings: For additive noise, Lü \cite{lu2025holder} obtained global probabilistically strong solutions for prescribed Hölder initial data, while further refinements addressed Onsager-critical behavior and local energy dissipation \cite{lu2025proof,pappalettera2025dissipative}. The setting of Euler equations driven by Stratonovich transport noise, was treated by Hofmanová, Lange, and Pappalettera in \cite{hofmanova2024global}. In \cite{HofmanovaZhuZhuErgodicity}, Hofmanová, Zhu, and Zhu stationary Euler solutions were obtained as vanishing-viscosity limits of stationary Navier--Stokes solutions constructed by convex integration. The viscous approximations in this construction are, however, explicitly non-Leray solutions and need not satisfy the Navier--Stokes energy inequality. 

These important developments resolved the long-standing obstruction to constructing global probabilistically strong, analytically weak solutions to stochastic Euler. What remained systematically outside this theory was their selection by energy-admissible Navier--Stokes dynamics. The oscillatory mechanisms underlying convex integration are not generated by the compactness of Leray solutions, and the corresponding convex-integration solutions of the Navier--Stokes equations generally lie outside the Leray--Hopf class.

As a consequence, the simultaneous realization of probabilistic strong solutions, satisfying, dissipative energy control, and obtained by vanishing viscosity from finite-energy Leray solutions of stochastic Navier--Stokes equations remained an open problem. This is settled in the present work.

\subsection*{Structure of the article}
The article is structured as follows: Firstly, we introduce the relevant equations, the Navier--Stokes and Euler equations with additive as well as transport noise, the relevant notation and some preliminary results in Section~\ref{sec:2}. Secondly, we introduce the different relevant solution concepts for the Navier--Stokes equations and prove the associated results in Section~\ref{sec:3} and finally in Section~\ref{sec:4}, we consider the Euler equations, prove the existence of probabilistically strong solutions, their equivalence
to dissipative weak solutions and the weak-strong uniqueness of energy-variational solutions to both systems, the Navier--Stokes and the Euler equations. 

\section{Preliminaries\label{sec:2}}
 
\subsection{The incompressible Navier--Stokes and Euler equations with stochastic forcing}
 We consider simultaneously   the incompressible Navier--Stokes and Euler equations with additive and transport stochastic forcing
\begin{subequations} \label{eq:stochastic_euler_additive}
\begin{align}
\mathrm{d}  \f u +( (\f u \cdot \nabla) \f u -\frac{\nu}{2} \Delta \f u )\de t + \de \nabla p  &= \sigma^1 \, \de {W} + ( \sigma^2  \cdot \nabla) \f u \circ \de W , && \text{in } (0,T) \times \mathcal{D}, \label{eq:eu1}\\
\nabla \cdot \f u &= 0, && \text{in } (0,T) \times \mathcal{D}, \label{eq:eu2}\\
\nu \left(  (\nabla  \f u)_{\sym} \cdot \f n + \alpha \f u\right)_{\tau} = 0 \,,\quad \f u \cdot \f n &= 0, && \text{on } (0,T) \times \partial\mathcal{D}, \label{eq:eu3}\\
\f u(0,x) &= \f u_0(x), && \text{in } \mathcal{D},\label{eq:eu4}
\end{align}
\end{subequations}
where the unknowns of the system are
 the velocity field  \( \f u(\omega,t,x) \in \mathbb{R}^d \) and the pressure
 \( p(\omega,t,x) \in \mathbb{R} \), which can be viewed as a Lagrange multiplier for the incompressibility constraint~\eqref{eq:eu2}. Given are
     the divergence-free initial velocity field  \( u_0 \in L^p(\Omega;L^2_{\sigma}(\mathcal{D})) \) with $p>4$ and independent from the cylindrical Wiener process   \( {W} \) defined on a separable Hilbert space \( \mathfrak{U} \), as well as 
 \( \sigma^1 \in L_2(\mathfrak{U}, L^2_{\sigma}(\mathcal{D})) \), and $ \sigma^2 \in L_2(\mathfrak{U}, H^2(\D)\cap H^1_{\sigma}(\D))$, two fixed Hilbert–Schmidt operators.
The operator $(\sigma^2\cdot \nabla)\f u $ has to be interpreted as a Hilbert--Schmidt operator $(\sigma^2\cdot \nabla)\f u \in L_2(\mathfrak{U}, L^2(\mathcal{D}))$ for all $\f u \in H^1(\D)$ 
as $H^2(\D)\cap H^1_{\sigma}(\D) \hookrightarrow L^\infty(\D)  $ for $d\leq 3 $.  See the definition in the Notation section below.

The velocity satisfies the zero-outflow boundary condition~\eqref{eq:eu3}$_2$, where \(\f n\) denotes the outward unit normal vector to \(\partial\D\). For \(\nu>0\), we additionally impose the Navier slip-with-friction boundary conditions. Throughout the article, we assume that \(\alpha>0\).
The Navier slip-with-friction conditions can be viewed as interpolating between the Navier slip conditions, corresponding to \(\alpha=0\), and the classical no-slip conditions, formally recovered as \(\alpha\to\infty\). We refrain from imposing the usual no-slip boundary conditions for the Navier--Stokes equations because of the possible formation of boundary layers and the lack of unconditional convergence results for the vanishing-viscosity limit in the no-slip case, even in the deterministic setting; see~\cite[Thm.~4.1]{TitiBoundaryLayer}. On the other hand, we do not consider the Navier slip conditions with \(\alpha=0\), as the well-posedness of the associated Stokes system would require additional restrictions on the domain; see~\cite{StokesRegularity}. The Navier slip-with-friction conditions considered here therefore provide physically reasonable boundary conditions that are compatible with both the rigorous existence theory and the vanishing-viscosity analysis developed in this article.

We note that the pressure $p$ is only a semimartingale thus, in order to avoid cross coupling in the correction terms, we first apply the Helmholtz-projection $ \mathcal{P}:L^2(\D)\to L^2_\sigma(\D)$ to the equation  eliminating the pressure. This leads to 
\begin{equation}
    \mathrm{d} \f u + \mathcal{P}\left [  (\f u \cdot \nabla) \f u - \frac{\nu}{2} \Delta \f u   \right]\de t = 
    \left( \sigma^1 + \mathcal{P} [( \sigma^2 \cdot \nabla) \f u ]\right) \circ \de {W}  \, . 
\end{equation}

We may rewrite the Stratonovich noise term into Itō 
via 
\begin{equation}
    \mathrm{d} \f u +\mathcal{P}\left [  (\f u \cdot \nabla) \f u -\frac{\nu}{2} \Delta \f u -  \frac{1}{2} (\sigma^2  \cdot \nabla )\mathcal{P}[(\sigma^2  \cdot \nabla )\f u] \right]\de t = 
    \left( \sigma^1 + \mathcal{P}[( \sigma^2 \cdot \nabla) \f u]\right) \de {W}  \, . 
\end{equation}

\subsection{Notation}\label{sec:notation}
Throughout this paper, let $\D \subset \R^d$
be a {bounded}  domain of class $\C^2$ with $d =  2$ or $3$ and $T>0$ fixed.
The space of smooth solenoidal functions 
is denoted by $\mathcal{C}_{\sigma}^\infty(\D;\R^d) := \{\f v \in \C^\infty(\D; \R^d)\mid \nabla \cdot  \f v =0\text{ in } \D , \f v \cdot \f n = 0 \text{ on }\partial \D \} $. By $L^2_{\sigma}( \D) $ and  $H^1_{\sigma}(\D)$ we denote the closure of $\mathcal{C}_{\sigma}^\infty(\D;\R^d)$ with respect to the norm of $ L^2(\D) $ and $  H^1( \D) $, respectively.
Note that $L^2_{\sigma}(\D) $ can be characterized by $ L^2_{\sigma}(\D) = \{ \f v \in L^2(\D)\lvert 
\nabla \cdot 
 \f v =0 \text{ in }
\D\, , \f n \cdot \f v = 0 \text{ on }
 \partial \D \} $,
 where the first condition has to be understood in the distributional sense and the second condition in the sense of the trace in $H^{-1/2}(\partial \D )$. The characterization of $H^1_{\sigma}(\D)$ is very similar. 
 We define the linear space $ \C^2_{\sigma}(\ov\D;\R^d)$  via 
    $$ \C^2_{\sigma}(\ov\D;\R^d) : = \{ \varphi \in \C^2(\ov\D;\R^d)  | \nabla \cdot \varphi =0 \text{ in }\ov\D \text{ and }\f n \cdot  \varphi = 0 \text{ on }\partial \D\} \,. $$
We define the symmetric form $a : H^1_\sigma(\D)\times H^1_\sigma(\D) \to \R $ via 
$$ a(\f u ,\f v) = \nu \int_{\D} (\nabla \f u)_{\sym  } : (\nabla \f v)_{\sym} \de x + \nu \alpha \int_{\partial \D} \f u \cdot \f v \de S 
 $$
the form $a$ is continuous, symmetric and coercive on $H^1_\sigma(\D) $. It determines a positive selfadjoint Stokes operator $\A$ on $L^2_{\sigma}(\D)$ with $ (\A \f u ,\f v) _{L^2_\sigma(\D)} = a(\f u, \f v) $ for all $ \f u , \f v\in H^1_\sigma(\D)$. The domain of this operator can be identified via $ \DA := \{ \f v \in H^2(\D); \R^d) \cap  H^1_\sigma(\D)\mid  (( \nabla \f v )_{\sym} \f n + \alpha \f v )_\tau = 0 \text{ on } \partial \D \} $~\cite{StokesRegularity}. 
The
subscript $\tau$ denotes the tangential component of a vector, \textit{i.e.}, $\f v _\tau  := \f v-(\f v\cdot \f  n)\f n $.

    The trace operator   $(\boldsymbol{v}, \pi) \mapsto[((\nabla  \boldsymbol{v})_{\sym}) \boldsymbol{n}]_{{\tau}}$, can be extended from $\C^2_\sigma(\ov \D;\R^d) \to \C^1(\partial \D; \R^d)$ to a linear, continuous map from $H^1_{\sigma}(\D)$ to ${H}^{-{1}/{2}}(\partial \D)$. Moreover, we have the following relation: for all $(\boldsymbol{v}, p ) \in H^2(\D)\cap H^1_\sigma(\D) \times H^1(\D)$ and $\boldsymbol{\varphi} \in H_{\sigma}^{1}(\D)$,
\begin{equation*}
\int_{\D}(-\Delta \boldsymbol{v}+\nabla p ) \cdot \boldsymbol{\varphi}=2 \int_{\D} (\nabla \boldsymbol{v}) _{\sym}: (\nabla \boldsymbol{\varphi}) _{\sym}-2\left\langle[(\nabla \boldsymbol{v}) _{\sym}\boldsymbol{n}]_\tau, \boldsymbol{\varphi}\right\rangle_{H^{-1/2}(\partial \D) \times H^{1/2}(\partial \D)}\,.
\end{equation*}
 With $ \C^1_c([0,T);[0,\infty))$ and  $L^\infty(\Omega;[0,\infty))$ we denote the functions that are nonnegative everywhere and almost everywhere, respectively.  By $(A)_{\sym}$, $(A)_-$, and $(A)_{\sym,-}$, we denote the symmetric, negative definite, and symmetric negative definite part of a quadratic matrix $A\in \R^{d\times d}$. 
 We use the usual notation for strongly, weakly, and weakly-star measurable  Bochner function spaces by $ L^p(\Omega;B)$, $L^p_w(\Omega;B)$, and $ L^p_{w^*}(\Omega;B)$, respectively, for a certain Banach space $B$ and a set $\Omega$. In a slight abuse of notation, we also write $ L^p_{w^*}(\Omega;\mathcal{P}(\R^d))$, which means the usual Bochner space of measures with the additional condition that the elements are probability measures.  By $\mathcal{B}$, we denote the usual Borel $\sigma$-algebra.

Let $\mathfrak{U}$ be a separable Hilbert space and let $\left(\mathbf{e}_k\right)_{k \in \mathbb{N}}$ be an orthonormal basis of $\mathfrak{U}$. 
For a Hilbert space \(H\), we denote by \(L_2(\mathfrak U,H)\) the space of
Hilbert--Schmidt operators from \(\mathfrak U\) into \(H\). Thus
\(G\in L_2(\mathfrak U,H)\) if and only if
\[
\|G\|_{L_2(\mathfrak U,H)}^2
:=
\sum_{k=1}^\infty \|G\mathbf e_k\|_H^2
<\infty.
\]
This definition is independent of the chosen orthonormal basis
\((\mathbf e_k)_{k\in\mathbb N}\).
We denote by $L_2\left(\mathfrak{U}, L^2_{\sigma}\left(\D\right)\right)$ and $L_2(\mathfrak{U}, H^2(
\D)
\cap H^1_{\sigma}(\D))$ the set of Hilbert-Schmidt operators from $\mathfrak{U}$ to $L^2_{\sigma}\left(\D\right)$ and $H^2(
\D)
\cap H^1_{\sigma}(\D)$, respectively. 
For simplicity, we assume that the two noise terms are orthogonal in the sense
that 
$\Tr{\int_{\D} \sigma^1 \cdot (\sigma^2\cdot \nabla)\varphi \dx } = \sum_{k=1}^\infty\int_{\D} \langle\sigma^1,\mathbf{e}_k\rangle_{\mathfrak U} \cdot (\langle\sigma^2,\mathbf{e}_k\rangle_{\mathfrak U}\cdot \nabla)\varphi \dx  =0 $ for all $\varphi \in H^1_{\sigma}(\D)$.
Here and in the following, we denote by $ \langle \sigma^1,\mathbf e_k\rangle_{\mathfrak U} \in L^2_{\sigma}(\D)$ the application of the Hilbert--Schmidt operator $\sigma^1 $ to the element in $\mathbf e_k \in \mathfrak{U}$.

Throughout the paper we consider a cylindrical Wiener process $W=\left(W_t\right)_{t \geq 0}$ 
which has the form
$$
W(t)=\sum_{k \in \mathbb{N}} \beta _k(t) \mathbf{e}_k \,,
 $$
where $\left(\beta_k\right)$ is a sequence of independent real-valued Brownian motions on a probability space  $\left(\Omega, \mathcal{F^W},\mathbb{P}\right)$. For \(t\in[0,T]\), let $\mathcal{F}_t^W
:=
\sigma\bigl(\beta_k(s):0\leq s\leq t,\ k\in\mathbb N\bigr)$ and for the initial datum $u_0 \in L^2(\Omega;L^2_\sigma(\mathcal{D}))$ set $\mathcal F^{u_0}:=\sigma(u_0),$ and assume that
$
\mathcal F^{u_0}\ \text{is independent of}\ \mathcal F_\cdot^W .$ 
Throughout the paper, the filtration is the usual augmentation of the initially enlarged Wiener filtration, that is, $
\mathcal F_t
:=
\bigcap_{r>t}
\bigl(\mathcal F^{u_0}\vee \mathcal F_r^W\bigr)^{\mathbb P},$ and $\mathcal{F}:=\bigl(\mathcal F^{u_0} \vee \mathcal{F}^W \bigr)^{\mathbb P}$, 
where  \((\,\cdot\,)^{\mathbb P}\) denotes completion with respect to \(\mathbb P\). In particular, \((\mathcal F_t)_{t\ge 0}\) is right-continuous and complete, and \(W\) is a cylindrical Wiener process with respect to \((\mathcal F_t)_{t\ge 0}\). 

A process \(X:[0,T]\times\Omega\to E\) for a Banach space $E$ is adapted if
\(X(t)\) is \(\mathcal F_t/\mathcal B(E)\)-measurable for every \(t\in[0,T]\).
It is progressively measurable if, for every \(t\in[0,T]\),
\[
X|_{[0,t]\times\Omega}
\quad\text{is}\quad
\mathcal B([0,t])\otimes\mathcal F_t \,/\, \mathcal B(E)
\text{-measurable}.
\]

For a progressively measurable process $\psi\colon \Omega\times(0,T)\to L_2(\mathfrak U;H),$ where $H$ is a separable Hilbert space, the stochastic integral
$\int_0^t\psi\,\de W$ is well defined whenever
\[
 \mathbb P \left(\int_0^T\|\psi\|_{L_2(\mathfrak U;H)}^2\,\de t<\infty\right)=1.
\]
Moreover, for every $q\in[1,\infty)$,
\begin{equation}\label{eq:BDG-general}
\mathbb E\left[
\sup_{t\in[0,T]}
\left\|\int_0^t\psi\,\de W\right\|_H^q
\right]
\leq
C_q\,
\mathbb E\left[
\left(
\int_0^T
\|\psi\|_{L_2(\mathfrak U;H)}^2\,\de t
\right)^{q/2}
\right].
\end{equation}
In particular, if the right-hand side of~\eqref{eq:BDG-general}
is finite for $q=1$, then the stochastic integral is a continuous
uniformly integrable martingale.

Define further $\mathfrak{U}_0 \supset \mathfrak{U}$ as
$$
\mathfrak{U}_0:=\left\{\mathbf{e}=\sum_k \alpha_k \mathbf{e}_k \in \mathfrak{U}: \sum_k \frac{\alpha_k^2}{k^2}<\infty\right\},
$$
thus the embedding $\mathfrak{U} \hookrightarrow \mathfrak{U}_0$ is Hilbert-Schmidt and trajectories of $W$ are $\mathbb{P}$-almost surely continuous with values in  $\mathfrak{U}_0$.
With $\mathrm{Tr}$ we denote the trace in the space of Hilbert--Schmidt operators, \textit{i.e.,} for two operators $ \Phi _i \in L_2(\mathfrak{U},L^2_{\sigma}(\D))$ for $i\in \{1,2\}$,  we define
$$
\Tr{\int_{\D} \Phi_1 \cdot\Phi_2 \de x  } = \sum_{k\in \N} \int_{\D} \langle\Phi_1 ,\mathbf e_k\rangle_{\mathfrak U}\cdot \langle\Phi_2,\f e_k\rangle_{\mathfrak{U}} \de x \,.
$$
We define the Skorokhod space $\mathfrak{D}([0, T])$ as the space of the real-valued càdlàg functions defined on $[0, T]$. More precisely, $\Phi$ belongs to the space $\mathfrak{D}([0, T])$ if it is right-continuous and has left-hand limits:
\begin{enumerate}[label=\roman*)]
    \item for $0<t \leq T, \Phi(t-)=\lim _{s \uparrow t} \Phi(s)$ exists,
    \item for $0 \leq t<T, \Phi(t+)=\lim _{s \downarrow t} \Phi(s)$ exists and $\Phi(t+)=\Phi(t)$.
\end{enumerate}
It is well known that càdlàg functions are bounded, measurable, and have at most countably many discontinuities. 
\subsection{Preliminary lemmata}
\begin{lemma}\label{lem:invar}
Let $f\in L^p(\Omega;L^1(0,T))$, $g\in L^p_{w^*}(\Omega;L^\infty(0,T )) $,  $g_0\in L^p(\Omega)$ for $p\geq 1$, and let $h$ be $(\mathcal F_t)$-progressively measurable and satisfy
\begin{equation}\label{eq:h-H1}
\mathbb E\left[
\left(
\int_0^T
\|h(t)\|_{L_2(\mathfrak U;\R)}^2\,\de t
\right)^{1/2}
\right]<\infty.
\end{equation}
Then the following three statements are equivalent:
\begin{enumerate}[label=\roman*)]
\item
The inequality 
\begin{multline}
-\expect{\psi(\omega)\int_0^T\partial_t\phi(\tau) g(\tau,\omega) \de \tau}  \\+\expect{\psi(\omega)\left( \int_0^T \phi(\tau) f(\tau,\omega) \de \tau+\int_0^T \phi (\tau) h(\tau,\omega) \de W(\tau) - \phi(0)g_0 (\omega)\right)}\leq 0 
\label{ineq1}
\end{multline}
holds for all $\phi \in{\C}^1_c ([0,T))$ with $\phi \geq 0$ and $\psi \in L^\infty(\Omega)$ with $\psi \geq 0$ $\mathbb P$-almost surely.
\item
The inequality
\begin{equation}
    g(t,\omega) -g(s,\omega) + \int_s^t f(\tau,\omega) \de \tau + \int_s^t h(\tau,\omega   ) \de W(\tau)\leq 0 
    \label{ineq2}
\end{equation}
holds for a.e.~$s< t\in (0,T)$ and $\mathbb{P}$-almost surely,
including $s=0$ if we replace $g(0)$ with $g_0$.

\item 
There exists a representative of $g$ that can be identified with a function in $\mathfrak{D}([0,T])$  such that
\begin{equation}
    g(t) -g(s-) + \int_s^t f(\tau) \de \tau + \int_s^t h(\tau ) \de W(\tau)\leq 0 \,
    \label{ineq.pw}
\end{equation}
for all $s\leq t\in[0,T]$,
where we set $g(0-)\coloneqq g_0$ and $\mathbb{P}$-almost surely.
In particular, it holds $g(0)\leq g_0$ and $g(t)\leq \lim_{s\nearrow t}g(s)$ for all $t\in (0,T]$ $\mathbb{P}$-almost surely.
\end{enumerate}
\end{lemma}
\begin{proof}
    A proof of this assertion in the deterministic setting can be found in~\cite[Lem.~1]{EnvarIncomp}. 
We prove the assertion here for the reader's convenience. 
We first prove the direction $i\Rightarrow ii$: 
Applying Lemma~\ref{lem:varicalc}, we observe that
\begin{equation*}
-\int_0^T\partial_t\varphi(\tau) g(\tau,\omega) \de \tau  + \int_0^T \varphi(\tau) f(\tau) \de \tau+\int_0^T \varphi (\tau) h(\tau) \de W - \varphi(0)g_0 (\omega)\leq 0 
\end{equation*}
holds for all $\varphi \in {\C}^1_c ([0,T))$ with $\varphi \geq 0$ $\mathbb{P}$-almost surely.
Now we can approximate the indicator function $ \chi_{[s,t]}$ by an approximate sequence $ \varphi_\varepsilon \to \chi_{[s,t]}$ for all $ s<t$. Standard arguments give rise to the inequality~\eqref{ineq2}.

In order to observe $ ii \Rightarrow iii $, we find from the inequality~\eqref{ineq2} that the left-hand side $ g(\cdot) -g_0 + \int_0^\cdot f(\tau) \de \tau + \int_0^\cdot h(\tau ) \de W$ is a monotonically decreasing function in time and as such of bounded variation. As both integrals are continuous, the function $g$ can be viewed as the sum of a function of bounded variation and a continuous function such that it has at most countably many jumps and we can choose a representative in~$\mathfrak{D}([0,T])$, which in turn implies~\eqref{ineq.pw}. 

In order to prove the  direction $iii\Rightarrow i$, we use an equidistant partition of the domain $[0,T]$ for any $N\in \N $ via $ \tau = T/N$ and $t_n = n \tau$  for $n\in \{ 0,\ldots , N\}$. For $\phi $
 as above, we sum over~\eqref{ineq2} with $t=t_n$ and $s=t_{n-1}$
 multiplied by $ \phi(t_{n-1} ) \psi(\omega)$ for any $\omega
 \in \Omega$ from $1$ to $N$. This gives rise to the inequality
\begin{equation*}
    \begin{aligned}
       \psi(\omega) \sum_{n=1}^N \phi(t_{n-1})\left[ g(t_n,\omega) -g(t_{n-1}-,\omega) + \int_{t_{n-1}}^{t_n} f(r,\omega) \de r + \int_{t_{n-1}}^{t_n} h(r,\omega ) \de W (r)\right] \leq 0 \,.
    \end{aligned}
\end{equation*}
for all $\phi \in L^{\infty}(\Omega,\mathcal{F};{\C}^1_c ([0,T)))$ with $\phi \geq 0$ and $\mathbb P$-almost surely. 
Applying a discrete integration-by-parts in time formula on the first term
\begin{align*}
    \sum_{n=1}^N \phi(t_{n-1}) [g(t_n,\omega) -g(t_{n-1}-,\omega)] + \phi(t_0)g_0(\omega)\geq {}& - \sum_{n=1}^N g(t_n-,\omega) [ \phi(t_{n}) - \phi(t_{n-1})] \\={}& -\int_0^T \partial _t \hat{\phi}(\tau) \overline{g}(\tau, \omega )  \de \tau 
\end{align*} 
and taking expectations leads to
\begin{multline*}
        \expect{- \psi(\omega)\left(\int_0^T \partial _t \hat{\phi}(\tau) \overline{g}(\tau, \omega )  \de \tau + \phi(0)g_0(\omega)\right)} \\+\expect{\psi(\omega)\left( \int_{0}^{T} \underline{\phi}(\tau  ) f(\tau,\omega) \de \tau + \int_{0}^{T}\underline{\phi}(\tau  ) h(\tau,\omega ) \de W(\tau)\right)} \leq 0 \,.
\end{multline*}
with the usual notation for the linear prolongations
$$
\hat{\phi}(t):= \begin{cases}
    \phi(t_{n-1}) + (t-t_{n-1} )\frac{\phi(t_n)-\phi(t_{n-1})}{\tau} &\text{if } t \in (t_{n-1}, t_n]\\
    \phi(0) &\text{if } t  = 0
\end{cases}
$$
as well as constant prolongations
$$
\overline{g}(r, \omega ) := \begin{cases}
    g(t_{n}-,\omega) & \text{if } r \in (t_{n-1}, t_{n}] \\
    g_0(\omega)&\text{if } r =0 \,,
\end{cases}\,,\quad
    \underline{\phi}(r) : = \begin{cases}
        \phi(t_{n-1})& \text{if } r \in [t_{n-1}, t_{n}) \\
   \phi(T)& \text{else} \,.
    \end{cases}
$$
 These prolongations converge strongly to their respective limit almost everywhere in $(0,T)$ and $\P$-almost surely  such that we end up with~\eqref{ineq1} due to dominated convergence and BDG together with dominated convergence for the quadratic-variation integrant.
 \end{proof}

\begin{definition}[Progressively measurable Young measure]\label{def:progmeasYoung}
Let $E {\hookrightarrow}  L^r\bigl(D\times(0,T);\mathbb{R}^d\bigr),$ $1\leq r<\infty$, and, for \(t\in[0,T]\), define the truncation operator $R_t:E\to E, R_t v:=\mathbf{1}_{(0,t]}v.$
A Young measure
\[
\nu:\Omega\to\mathcal{P}(E)
\]
is called \((\mathcal{F}_t)_{t\in[0,T]}\)-progressively measurable if, for
every \(t\in[0,T]\), the map
\[
\Omega\ni\omega\longmapsto
\nu_\omega^t:=(R_t)_{\#}\nu_\omega\in\mathcal{P}(E)
\]
is \(\mathcal{F}_t/\mathcal{B}(\mathcal{P}(E))\)-measurable, where
\(\mathcal{P}(E)\) is equipped with the topology of weak convergence.

Equivalently, for every \(t\in[0,T]\) and every
\(\Phi\in C_b(E)\), the map
\[
\Omega\ni\omega\longmapsto
\int_E \Phi(R_t v)\,\mathrm{d}\nu_\omega(v)
\]
is \(\mathcal{F}_t\)-measurable.
\end{definition}


\begin{lemma}[Fundamental lemma of variational calculus]\label{lem:varicalc}
    Let $f\in L^1(\Omega)$. Then the following two statements are equivalent:
    \begin{enumerate}[label=\roman*)]
\item the inequality 
$ \expect{\psi f }\geq 0$
holds for all $\psi \in L^\infty(\Omega)$ with $\psi \geq 0$.
\item 
the inequality $ f \geq 0 $ holds $\mathbb{P}$-almost surely. 
\end{enumerate}
\end{lemma}
The fundamental theorem for Young measures is a standard tool in modern Analysis. Usually, it is only considered on an underlying space in $\R^n$ equipped with the Lebesgue measure. But most of the results can be generalized to more general measure spaces. We refer to~\cite{Berliocchi,Warga,Valadier} for such results.  

\begin{theorem}[Fundamental theorem for Young measures on $\Omega\times  L^r(\D\times (0,T))$]\label{thm:Young}
Let $(\Omega,\mathcal{F},(\mathcal{F}_t)_{t\geq 0} ,\mathbb P)$ be a filtered probability space and $\D\subset\mathbb R^d$ a bounded domain. 
Let $\mathbb X, E$ be  reflexive Banach spaces such that $ \mathbb X  \hookrightarrow^c  E \hookrightarrow L^r(\D\times(0,T); \R^d) $ with $ r > 2$ where the first  embedding is assumed to be compact. Let $ p>2$, $q >1$,
 $(u_n)_{n\in\mathbb N}$ and $\f u$ be a family of $(\mathcal{F}_t)$-progressively-measurable maps  such that 
$\sup_{n\in\N} \left[ \| u_n \|_{L^p(\Omega;E)}
+ \| u_n\| _{L^q(\Omega
; \mathbb X)} 
\right] <\infty$ and $ u_n \rightharpoonup \f u$ in $L^p(\Omega ; E) \cap L^q(\Omega; \mathbb X)$.

Then, there exist a subsequence (not relabeled) and progressively measurable Young measure in the sense of Definition~\ref{def:progmeasYoung},  
$
\nu_{\omega}\in \mathcal P(E )$ $ \P\text{-a.s. } 
$, 
such that:
\begin{enumerate}
  \item[i)]
  it holds 
for all $\eta\in L^\infty(\Omega)$ and all
$\Phi\in C_b(E)$,
\begin{equation}\label{convYoung}
    \lim_{n\to\infty}
\mathbb E\left[
\eta(\omega)\Phi(u_n(\omega))
\right]
=
\mathbb E\left[
\eta(\omega)
\int_{E} \Phi(v)\de \nu_\omega(v)
\right].
\end{equation}
\item[ii)]
Moreover it holds for all elements $ \f f \in L^{\frac{r}{r-1}}(\D \times (0,T))$ that 
\begin{equation}
    \int_{0}^T \int_{\D} \f u \cdot \f f \de x \de t = \int_0^T \int_{\D} \int_{E
    } \f f \cdot v \de \nu_{\omega}(v) \de x\de t 
\end{equation}
$\P$-almost surely, which is nothing else than $\f u =\int_{E}  v \de \nu_{\omega}(v)  $ in $E$ $\P$-almost surely for the progressively measurable weak limit $\f u$. 

  \item[iii)] 
  It holds that
  for every bounded measurable $\phi: \Omega \times (0,T)\times \D\to\mathbb R^{d\times d}$,  it holds 
  \begin{multline*}
     \lim_{n\to \infty} \int_\Omega\int_0^T \int_{\D} \phi(\omega,t,x): u_n(\omega,t,x)\otimes u_n(\omega,t,x) \,\de x\de t\de \mathbb P(\omega)\\
  =
  \int_\Omega\int_0^T \int_{\D} \phi(\omega,t,x) \int_{E}   v\otimes v  \de \nu_{\omega}(v)\de x\de t \de\mathbb P(\omega)\,.
  \end{multline*}
  Moreover, it holds that $ \int_{E}   v\otimes v  \de \nu_{\omega}(v) \in L^{p/2}(\Omega; L^{r/2}(\D\times (0,T); \R^{d\times d}))$ and is progressively measurable. 
  
    \item[iv)] 
    
    The Young measure $\nu$ is a progressively measurable Young measure in the sense of Definition~\ref{def:progmeasYoung}.    
  \end{enumerate}
\end{theorem}

\begin{proof}
Since $1\le r<\infty$, $E$ is a separable Banach space, hence a Polish
space. In particular compact subsets of $E$ are metrizable.

For $R>0$, define
$
K_R
:=
\left\{
v\in\mathbb X:\|v\|_{\mathbb X}\le R
\right\}.
$
Since the embedding $\mathbb X\hookrightarrow^c E$ is compact, $K_R$ is compact
in $E$. It is also closed in the strong topology of $E$. Indeed, for any sequence $ \{ v_n\}\subset K_R$, we may extract a converging subsequence $ v_n \rightharpoonup v $ in $\mathbb X$. Due to the uniqueness of weak limits and compactness, $ v_n \to v$ in $\E$. Weakly lower semicontinuity of the norm in $\mathbb X$ guarantees that $ v \in K_R$. 

We now prove tightness. By Markov's inequality and the uniform bounds,
\[
\mathbb P(u_n\notin K_R)
\le
\mathbb P(\|u_n\|_{\mathbb X}>R)
\le
\frac{\|u_n\|_{L^{q}(\mathbb X)}^q}{R^{q}}\le
\frac{C}{R^{q}}.
\]
Therefore, for every $\varepsilon>0$, one can choose $R>0$ such that
$
\sup_{n\in\mathbb N}
\mathbb P(u_n\notin K_R)<\varepsilon
$.
Equivalently,
$
\sup_{n\in\mathbb N}
\int_\Omega \delta_{u_n(\omega)}(K_R^c) \de\mathbb P(\omega)
<\varepsilon$. 
Thus the family of Young measures
$
\omega\mapsto \delta_{u_n(\omega)}
$
is  tight in the stable Young-measure sense.
By the Prokhorov compactness criterion for Young measures in the stable
topology~\cite[Theorem~4.3.5,A)1.]{Valadier}, there exist a subsequence, not relabeled, and a Young measure
$\nu:\Omega\to\mathcal P(E)$ such that~\eqref{convYoung} holds. 

The limiting Young measure inherits progressive measurability. Indeed, fix
\(t\in[0,T]\) and let \(R_t v:=\mathbf{1}_{[0,t]}v\). For every
\(\Phi\in C_b(E)\), the random variables \(\Phi(R_tu_n)\) are
\(\mathcal{F}_t\)-measurable and uniformly bounded. By the stable convergence
in \eqref{convYoung},
\[
 \Phi(R_tu_n)
 \stackrel{*}{\rightharpoonup}
 \int_E \Phi(R_tv)\,\mathrm{d}\nu_\omega(v)
 =
 \int_E \Phi(v)\,
 \mathrm{d}\bigl((R_t)_{\#}\nu_\omega\bigr)(v)
 \qquad\text{in }L^\infty(\Omega).
\]
Since \(L^\infty(\Omega,\mathcal{F}_t)\) is weakly-* closed in
\(L^\infty(\Omega,\mathcal{F})\), the limit is
\(\mathcal{F}_t\)-measurable. As \(\Phi\in C_b(E)\) and \(t\in[0,T]\) were
arbitrary, the map
\[
 \omega\longmapsto (R_t)_{\#}\nu_\omega
\]
is \(\mathcal{F}_t/\mathcal{B}(\mathcal{P}(E))\)-measurable for every \(t\).
Consequently, \(\nu\) is \((\mathcal{F}_t)\)-progressively measurable.

The result $ii)$ can be observed  from the young measure-convergence~\eqref{convYoung} together with the choice $ \Phi : L^r(\D\times (0,T)) \to \R$ via $ \Phi (v) = \int_0^T \int_{\D} \f f \cdot v \de x \de t $. This is not immediately an admissible functions as it is not bounded, but according to\cite[Prop.~2.4.1]{Valadier} this can be extended to functionals $\Phi$ with growth $r$. 
Deducing the weak convergence of a subsequence of $\{u_n\}$ in the reflexive space $ L^p(\Omega ; E)$, we infer from the continuous embedding $  E \hookrightarrow L^r(\D\times (0,T))$ we also infer that 
$ u_n \rightharpoonup u $ in $L^p(\Omega; L^r( \D \times (0,T)))$. 
Applying Fubini, we may now conclude for all $ \psi \in L^\infty(\Omega)$ that
$$
\begin{aligned}
     \expect{\psi\int_{0}^T \int_{\D} \f u \cdot \f f \de x \de t }= {}&
 \lim_{n\to \infty} \expect{\psi\int_{0}^T \int_{\D} u_n \cdot \f f \de x \de t }\\
 ={}&
 \expect{\psi\int_{E} \int_0^T \int_{\D}\f f \cdot v \de x\de t\de \nu_{\omega}(v) } \\={}&\expect{\psi\int_0^T \int_{\D} \int_{E} \f f \cdot v \de \nu_{\omega}(v) \de x\de t}\,.
\end{aligned}
$$
With the usual density and continuity arguments, we infer $ii)$.

In order to infer  result $iii)$,  
we may choose $ \Phi \in \C ( L^r(\D\times (0,T); \R^d); \R)$
induced by $$ \Phi(v):=
\int_0^T \int_{\D} \phi(\omega,t,x):v \otimes v \,\de x\de t \,.$$
As $2<p$ the considered sequence of functions will be asymptotically uniformly $2$-integrable such that the result is inferred from~\cite[Prop.~2.4.1, 3)$\Rightarrow$1)]{Valadier}.  The similar $1$-integrable case is dealt with in~\cite[Lem.~6.2.1]{Valadier}. Fubini allows again to change the order of integration. 
Using this, we may estimate 
\begin{multline}\label{est:quadtermYoung}
       \int_\Omega\int_0^T \int_{\D} \phi(\omega,t,x) \int_{E}   v\otimes v  \de \nu_{\omega}(v)\de x\de t \de\mathbb P(\omega)
       \\
       = \int_\Omega\int_{E} \int_0^T \int_{\D} \phi(\omega,t,x)   v\otimes v  \de x\de t\de \nu_{\omega}(v) \de\mathbb P(\omega) 
       \\=\lim_{n\to \infty} \int_\Omega\int_0^T \int_{\D} \phi(\omega,t,x): u_n(\omega,t,x)\otimes u_n(\omega,t,x) \,\de x\de t\de \mathbb P(\omega)\\
       \leq \lim_{n\to \infty}\| \phi \|_{L^{p/(p-2)}(\Omega; L^{r/(r-2)}(\D\times (0,T)))} \| u_n \|_{L^p(\Omega ; L^r(\D\times (0,T)))}^2\,,
  \end{multline}
where the first equality is due to Fubini, the second due to the Young measure convergence, and the last one is H\"older's inequality. This inequality implies the asserted regularity for $ \int_{E}v\otimes v \de \nu(v) $.
As the approximating sequences of $\f u$ and  $ \int_{E}v\otimes v \de \nu(v) $ are $(\mathcal{F}_t)$-progressively measurable and progressively measurable functions are a closed convex set in $L^p(\Omega \times (0,T))$ for any $p\in (1,\infty)$~\cite{Predict}, we may infer from the weak convergence  that also the limits are $(\mathcal{F}_t)$-progressively measurable.
\end{proof}

\begin{lemma}\label{lem:weakcont}
Let $\mathbb X$ be a Hilbert space and $\mathbb Y$ be a   Banach space such that $\mathbb Y$ embeds densely into $\mathbb X$, \textit{i.e.,} $ \mathbb Y \hookrightarrow^d \mathbb X$. For any 
     $ u \in L^p_{w^*}(\Omega; L^\infty(0,T;\mathbb X)) \cap L^q(\Omega; W^{\gamma,r}(0,T;\mathbb Y ^*))$ with $p$, $q$, $r\geq 1$ such that $\gamma r>1$ it holds that  $ u (\omega) \in \C_w([0,T];\mathbb X)$ $\mathbb P$-almost surely. 
\end{lemma}
\begin{proof}
First, we observe that  $ L^q(\Omega;W^{\gamma,r}(0,T;\mathbb Y^* )) \hookrightarrow L^q(\Omega;C([0,T];\mathbb Y^*))$ as $\gamma r>1$. Let now $\omega \in \Omega$ be such that $u(\omega)\in L^\infty(0,T;\mathbb X) \cap C([0,T];\mathbb Y^*)$. Let $\varphi \in \mathbb X$. We aim to show that $t \mapsto \langle u(\omega,t),\varphi \rangle $ is continuous. For any given $\varepsilon >0 $, by density we may choose an element $ \varphi_\varepsilon\in\mathbb Y$ such that $ \| \varphi_\varepsilon - \varphi \|_{\mathbb X} < \frac{\varepsilon}{4\| u(\omega) \| _{L^\infty(0,T;\mathbb X)} }$. Now, we may choose
$ \delta >0 $ such that for all $|s-t|<\delta $ it holds $ \| u(s)-u(t)\|_{\mathbb Y^*} < \frac{\varepsilon}{2 \| \varphi_\varepsilon\|_{\mathbb Y}}$ due to the continuity of $u(\omega)$ in $\mathbb Y^*$.
For any $s$, $t\in [0,T]$ with $|s-t|<\delta$ it now holds 
$$\begin{aligned}
    |\langle \varphi ,  u(t)-u(s)\rangle _{\mathbb X}| \leq {}&
| \langle \varphi_\varepsilon , u(t)-u(s)\rangle _{\mathbb Y}| + | \langle \varphi_\varepsilon -\varphi , u(t)-u(s)\rangle _{\mathbb X}| 
\\
\leq {}& \| \varphi_\varepsilon \|_{\mathbb Y} \| u(t)-u(s)\|_{\mathbb Y^*}+ \| \varphi_\varepsilon-\varphi\|_{\mathbb X} 2 \| u \| _{L^\infty(0,T;\mathbb X)} < \varepsilon\,.
\end{aligned}
$$
\end{proof}
\begin{lemma}[Stochastic variational lemma]\label{lem:variationprog}
    Let $ f \in L^1(\Omega \times (0,T))$ progressively measurable such that 
    $$
    \expect{\int_0^T f g \de t } =0 \quad 
    $$
     for all $g\in L^\infty(\Omega \times (0,T))$ that are progressively measurable, then it follows that $ f =0$ a.e. in $(0,T)$ and $\P$-almost surely. 
\end{lemma}
\begin{proof}
    We can observe the assertion  by choosing $g = \mathrm{sign} f $. 
\end{proof}

\begin{lemma}[Progressively measurable representatives of dual elements]\label{lem:dual}
Let $(\Omega,\mathcal F,\mathbb P)$ be a probability space equipped with
a filtration $(\mathcal F_t)_{t\in[0,T]}$ satisfying the above conditions.
Let $V$ be a separable Banach space, let $1<p<\infty$, and let
$q\in(1,\infty)$ be the conjugate exponent, i.e.
$
\frac1p+\frac1q=1$. 

Then it holds that the dual space of all progressively measurable functions in  $L^p\bigl(\Omega;L^1(0,T;V)\bigr) $ can be identified with all progressively measurable functions in
$ L^q_w\bigl(\Omega;L^\infty_{w^*}(0,T;V^*)\bigr)
$.
\end{lemma}
\begin{proof}
    By the usual Hahn--Banach argument, every linear continuous functional on the progressively measurable subset of $ L^p\bigl(\Omega;L^1(0,T;V)\bigr)$, $X_{\mathrm{prog}}$ can be extended norm-consistently  to the full $ L^p\bigl(\Omega;L^1(0,T;V)\bigr)=X$, as this subset is closed. 
Thus for every element     $\ell \in X_{\mathrm{prog}}^*$ there exists an element $ Z \in X^* $ 
such that $ \langle \ell , f \rangle_X = \expect{\int_0^T \langle Z ,f \rangle _{V} \de t }$ for all $ f \in X_{\mathrm{prog}}$.  
Now we define  $\hat{Z}(t,\omega) \in V^*  $ via the conditional expectation in $V^*$ $ \hat{Z}(t,\omega)= \expect{ {Z}(t,\omega)| \mathcal{F}_t}$. 
Defining $ M(t):= \esssup_{0\leq s \leq t}\|Z(s)\|_{V^*}$, we find that $ \expect{\|Z(t)\|_{V^*}|\mathcal{F}_t}\leq \expect{M(t)|\mathcal{F}_t}$ and that $\expect{M(t)|\mathcal{F}_t}$ is a sub-martingale. 
We observe that $\hat{Z} \in X^*$ 
$$
\begin{aligned}
     \expect{\Big|\esssup_{t\in(0,T)} \| \hat{Z} (t,\omega)\|_{V^*} \Big|^q}
     \leq{}& \expect{\Big|\esssup_{t\in(0,T)} \expect{M(t)|\mathcal{F}_t} \Big|^q}\\
     \leq{}&  \left(\frac{q}{q-1} \right)^q\expect{\Big| \expect{M(T)\Big| \mathcal{F}_T}\Big|^q}
     \\
     \leq {} & \left(\frac{q}{q-1} \right)^q
     \expect{\Big|M(T)\Big|^q}\,,
\end{aligned}
$$
where the second inequality follows from 
Doob's maximal inequality~\cite[Chap.~II, Thm.~1.7]{DoopMaximal} due to the fact that $M$ is a sub-martingale and the third one from the definition of conditional expectation and as $Z$ and also $M$ are $\mathcal{F}_T$ measurable. 
As $Z$ is weak$^*$ measurable, there exists a sequence of simple functions $Z_n(\omega,t):= \sum_{i=1}^{N_n} v^*_i \chi_{A_i}(\omega ,t)$ such that $Z_n \to Z $ a.e.~in $V^*$. 
This implies that 
$$\expect{\int_0^T \expect{\langle {Z}, f\rangle _V\Big| \mathcal{F}_t} \de t } 
= \lim_{n\to \infty} \expect{\int_0^T \expect{\langle {Z_n}, f\rangle _V\Big| \mathcal{F}_t} \de t }\,.
$$
For any  progressively measurable processes 
$f \in L^p\bigl(\Omega;L^1(0,T;V)\bigr)$ there also exists a sequence of simple function  $\{f_n \}$ converging a.e.~to$f$. On each element of the joint partition of $f_n$ and $Z_n$ the set $ \{ B_n\}$ of $\Omega \times (0,T)$ $ f_n $ and $Z_n$ are constant such that we may introduce the conditional expectation, \textit{i.e.,}
$$
\begin{aligned}
\expect{\int_0^T \expect{\langle {Z}, f\rangle _V\Big| \mathcal{F}_t} \de t } 
= {}&
\lim_{n\to \infty} \sum_{l=1}^{N_n} \expect{\int_0^T  \expect{\langle {Z_n}, f_n \chi_{B_n} \rangle _V\Big| \mathcal{F}_t} \de t } \\
= {}&
\lim_{n\to \infty} \sum_{l=1}^{N_n} \expect{\int_0^T \langle \expect{ {Z_n}| \mathcal{F}_t}, f_n \chi_{B_l} \rangle _V \de t } 
={} \expect{\int_0^T \langle \hat{Z}, f\rangle _V \de t }\,.    
\end{aligned}
$$

From the definition of the conditional expectation, we also find for all progressively measurable processes 
$f \in L^p\bigl(\Omega;L^1(0,T;V)\bigr)$ that 
$$
\expect{\int_0^T \langle \hat{Z}, f\rangle _V \de t }
= \expect{\int_0^T \expect{\langle {Z}, f\rangle _V\Big| \mathcal{F}_t} \de t } 
= \expect{\int_0^T \langle {Z}, f\rangle _V \de t } = \langle \ell , f \rangle_X \,, 
$$
which proves the assertion as $ \hat Z $ is progressively measurable. 
\end{proof}

\begin{lemma}[Martingale representation with independent enlargement]
\label{lem:mr-independent-initial-enlargement}
Let \(\mathcal H, \mathfrak U\) be separable Hilbert spaces. Let \(W\) be a cylindrical Wiener process on \(\mathfrak U\), and $\mathcal F^{u_0}$ be a $\sigma$ field independent of \(\mathcal F_\cdot^W\). Let $\mathcal F_t
=
\bigcap_{r>t}
\bigl(\mathcal F^{u_0}\vee \mathcal F_r^W\bigr)^{\mathbb P}, T>0$. Then every square-integrable continuous \((\mathcal F_t)\)-martingale
\(M=(M_t)_{t\in[0,T]}\) with values in \(\mathcal H\) admits a representation
\[
M_t
=
M_0+\int_0^t \Phi(s)\,dW(s),
\qquad t\in[0,T],
\]
where
\[
\Phi\in L^2\bigl(\Omega\times(0,T),\mathcal P;L_2(\mathfrak U,\mathcal H)\bigr),
\]
and \(\mathcal P\) denotes the progressive \(\sigma\)-field associated with
\((\mathcal F_t)_{t\in[0,T]}\).
\end{lemma}

\begin{proof}
It is enough to prove the assertion for martingales of the form
\[
M_t^X:=\mathbb E[X\mid \mathcal F_t],
\qquad X\in L^2(\Omega,\mathcal F_T;\mathcal H).
\]
First let \(X=YZ\), where \(Y\in L^\infty(\Omega,\mathcal F^{u_0})\) and
\(Z\in L^2(\Omega,\mathcal F_T^W;\mathcal H)\). By the martingale representation theorem in the Wiener filtration~\cite[Theorem~2.5, p.51]{MartingaleRepresentation}, there exists a
\[
\Psi\in L^2\bigl(\Omega\times(0,T),\mathcal P^W;L_2(\mathfrak U,\mathcal H)\bigr)
\]
such that
\[
\mathbb E[Z\mid \mathcal F_t^W]
=
\mathbb E[Z]+\int_0^t \Psi(s)\,dW(s).
\]
Using the independence of \(\mathcal F^{u_0}\) and \(\mathcal F_\cdot^W\), we obtain
\[
\mathbb E[YZ\mid \mathcal F_t]
=
Y\,\mathbb E[Z\mid \mathcal F_t^W]
=
\mathbb E[YZ\mid \mathcal F^{u_0}]
+
\int_0^t Y\Psi(s)\,dW(s).
\]
Since finite sums of such random variables \(YZ\) are dense in
\(L^2(\Omega,\mathcal F_T;\mathcal H)\), and using Itô isometry, we obtain that for every
\(X\in L^2(\Omega,\mathcal F_T;\mathcal H)\), there exists a \( \Phi\in L^2\bigl(\Omega\times(0,T),\mathcal P;L_2(\mathfrak U,\mathcal H)\bigr),\) such that
\[
\mathbb E[X\mid \mathcal F_t]
=
\mathbb E[X\mid \mathcal F^{u_0}]
+
\int_0^t \Phi(s)\,dW(s).
\]
Applying this to \(X=M_T\) gives
\[
M_t
=
M_0+\int_0^t \Phi(s)\,dW(s),
\]
because \(M_0=\mathbb E[M_T\mid\mathcal F_0]\) and
\(\mathcal F_0=\mathcal F^{u_0}\) up to completion.
\end{proof}

\section{Main results for the stochastic incompressible Navier--Stokes equations\label{sec:3}}

At first, we expand on our observation that in case of the Navier--Stokes equations the usual  analytically and probabilistically weak solutions are not the only possible relevant concept for these equations. One can allow for probabilistically strong solutions 
 by generalizing the solution concept. 
Therefore, we introduce not only analytically and probabilistically weak solutions but also probabilistically strong analytically measure-valued solutions and establish how both solutions are connected. 

\begin{assumption}\label{Ass:1}
      Let $\mathcal{D}$ be a bounded  domain of class $\C^2$.  
      Let $\mathfrak{U}$ be a separable Hilbert space and let $\left(\mathbf{e}_k\right)_{k \in \mathbb{N}}$ be an orthonormal basis of $\mathfrak{U}$ and let $\sigma^1 \in L_2\left(\mathfrak{U} ; L^2_\sigma \left(\D\right)\right)$ as well as $\sigma^2\in L_2\left(\mathfrak{U} ;(H^2(\D) \cap H^1_{\sigma}(\D))\right)$. 
      Let  $W=\left(W_t\right)_{t \geq 0}$ be a cylindrical Wiener process
which has the form
$$
W(t)=\sum_{k \in \mathbb{N}} \beta _k(t) \mathbf{e}_k \,,
 $$
where $\left(\beta_k\right)$ 
are independent real-valued Brownian motions on a complete probability space $\left(\Omega, \mathcal{F},\mathbb{P}\right)$. In addition, assume that $\f u_0 \in L^2(\Omega;L^2_{\sigma}(\D))$ is independent of $W$, and let $\left(\mathcal{F}_t\right)_{t \geq 0}$ be the canonically augmented filtration as defined in Section \ref{sec:notation}.
\end{assumption}

\begin{definition}[Weak-weak solutions to stochastic Navier--Stokes equations] \label{def:weakweakNav}
Let $\f u_0 \in L^4(\Omega, \mathcal{F}_0;L^2_{\sigma}(\D))$. 
The tuple consisting of a new stochastic basis and an associated process
$$
\left(\left(\tilde\Omega, \tilde{\mathcal{F}},\left(\tilde{\mathcal{F}}_t\right), \tilde{\mathbb{P}}\right), \tilde {\f{u}}, \tilde W\right)
$$
is called a finite energy probabilistically weak, weak solution, or in short weak-weak solution, to~\eqref{eq:stochastic_euler_additive}  with the initial data $\f{u}_0$ and $\nu>0$ provided $\tilde{\f u}$ is a weak solution on the new 
stochastic basis, \textit{i.e.,}
\begin{enumerate}[label=\roman*)]
    \item $\left(\tilde\Omega, \tilde{\mathcal{F}},\left(\tilde{\mathcal{F}}_t\right), \tilde{\mathbb{P}}\right)$ is a stochastic basis with a complete right-continuous filtration;
    \item $\tilde W$ is an $\left(\tilde{\mathcal{F}}_t\right)$-cylindrical Wiener process;
    \item  \label{reg:weakweak} The velocity field $\tilde{\f{u}}$ is $\left(\tilde{\mathcal{F}}_t\right)$-progressively measurable and satisfies
$$
\tilde{\f{u}} \in  
L^4\left (\tilde\Omega;L^\infty\left(0,T ; L_{\sigma}^2\left(\D \right)\right) \cap L^2\left(0, T ; W_{\sigma}^{1,2}\left(\D \right)\right) \right)
$$
such that $\tilde{\f u} \in \C_w([0,T];L^2_\sigma(\D))$ $\tilde{\mathbb P}$-almost surely;
\item $\mathbb{P}\circ (\f{u}_0)^{-1}=\tilde{\mathbb{P}} \circ(\tilde{\f{u}}(0))^{-1}$;
\item For all $\f\varphi \in C_{\sigma}^{2}\left(\D \right)$ and all $t \geq 0$ there holds $\tilde{\mathbb{P}}$-almost surely 
$$
\begin{aligned}
\int_{\D } \tilde{\f{u}}(t) \cdot \boldsymbol{\varphi} \mathrm{d} x& =\int_{\D } \tilde{\f{u}}(0) \cdot \boldsymbol{\varphi} \mathrm{d} x+\int_0^t \int_{\D } \tilde{\f{u}} \otimes \tilde{\f{u}}: \nabla \boldsymbol{\varphi} + \frac{1}{2}[ \mathcal P(\sigma^2\cdot \nabla)]^2 \f \varphi \cdot \tilde{\f u}\mathrm{d} x \mathrm{~d} s \\
& \quad -\nu \int_0^t \int_{\D } (\nabla \tilde{\f{u}})_{\sym}: \nabla \boldsymbol{\varphi} \mathrm{d} x +\alpha  \int_{\partial \D} \tilde{\f u} \cdot \f \varphi \de S \mathrm{~d} s\\ & \quad +\int_0^t \int_{\D } \boldsymbol{\varphi} \cdot \sigma ^1 - 
(\sigma^2 \cdot \nabla ) \f \varphi \cdot \tilde{\f u} \mathrm{~d} x \mathrm{~d} \tilde W(s)
\end{aligned}
$$
\item The energy inequality holds in the sense that
\begin{multline}
\frac{1}{2}\int_{\D}|\tilde{\f{u}}(t)|^2\de x +\nu\int_s^t  \int_{\D }|(\nabla \tilde{\f{u}})_{\sym}|^2 \mathrm{~d}  x + \alpha \int_{\partial \D} 
|\tilde{\f u}|^2\mathrm{~d} S \mathrm{~d} \tau \\
\leq \frac{1}{2}\int_{\D}|\tilde{\f{u}}(s)|^2\de x+\frac{1}{2} \int_s^t\|\sigma^1 \|_{L_2\left(\left(\mathfrak{U}, L^2\left(\D \right)\right)\right)}^2 \mathrm{~d} \tau +\int_s^t \int_{\D } \tilde{\f{u}} \cdot \sigma^1 \mathrm{~d} \tilde W(\tau) 
\end{multline}
$\tilde{\mathbb{P}}$-almost surely~for almost all~$s \geq 0$ (including $s=0$ ) and all $t \geq s$.
\end{enumerate}

  \end{definition}
  
For the following definitions we use stochastic test processes with a
prescribed It\^o decomposition. The deterministic test functions occurring
in Definition~\ref{def:weakweakNav} form a subclass of the test processes
below. 
In the present setting, both formulations are equivalent. Indeed, for a test-process $\varphi$ in the sense of Definition \ref{def:stoch} below, first consider the product $(u(t),\f v_j)_{L^2_\sigma(\D)}(\varphi(t),\f v_j)_{L^2_\sigma(\D)}$, where $\f v _j$ is an $L^2_{\sigma}(\D)$-orthonormal basis of eigenvectors of the Stokes operator with Navier-slip boundary conditions with friction. Then, using Definition \ref{def:weakweakNav} (v), the scalar Itô product rule can be applied to this product. Subsequently, summing over $j$ one can pass to the limit and obtain the analog of \eqref{eq:defMeasweakNav}.

\begin{definition}[Test process]\label{def:stoch}
 We call $\f \varphi$ a test process, if it is $(\mathcal{F}_t)$-progressively measurable such that 
 \begin{align}\label{eq:test-regularity}
 \f\varphi  &\in  L^2\bigl(\Omega;\C([0,T];L^2_\sigma(\D))\bigr) \cap  L^2\bigl(  \Omega;
 L^2(0,T;\C^1(\D)\cap H^2(\D))  \bigr),
 \end{align}
     with 
     $$
     \de \f \varphi =  A \de t + B \de W_t 
     $$
     for  some   $(\mathcal{F}_t)$-progressively measurable  functions $$ A \in L^{2}(\Omega ; L^2(0,T;L^2_{\sigma}(\D))) \quad \text{ and }\quad B \in L^2 (\Omega; L^2(0,T; L_2 (\mathfrak{U}; 
     H^1_{\sigma}(\D)))\,.$$
\end{definition}

\begin{remark}\label{rem:test-H1}
Every test process satisfies
\[
\f\varphi
\in
L^2\bigl(
\Omega;
L^4(0,T;H^1(\D))
\bigr).
\]
Indeed, pathwise interpolation gives
\[
\|\f\varphi\|_{L^4(0,T;H^1)}^2
\leq
C
\|\f\varphi\|_{\C([0,T];L^2)}
\|\f\varphi\|_{L^2(0,T;H^2)},
\]
and the assertion follows from the Cauchy--Schwarz inequality in
$\Omega$.
\end{remark}

\begin{lemma}[Stochastic integrability for test processes]
\label{lem:test-stoch-integrability}
Let $p\geq2$, let
\[
\f u
\in
L^p_{w^*}
\bigl(
\Omega;
L^\infty(0,T;L^2_\sigma(\D))
\bigr),
\]
and let $\f\varphi$ be a test process with coefficients $A,B$.
Set $\rho_p:=\frac{2p}{p+2}.$ Each stochastic integrand occurring in
\eqref{eq:defMeasweakNav},
\eqref{eq:envarNav}, and
\eqref{eq:envar}
belongs to
\[
L^{\rho_p}
\bigl(
\Omega;
L^2(0,T;L_2(\mathfrak U;\R))
\bigr).
\]
More precisely, if $G$ denotes any one of these stochastic
integrands, then
\begin{multline}\label{eq:test-G-estimate}
\left\|
\left(
\int_0^T
\|G(t)\|_{L_2(\mathfrak U;\R)}^2\,\de t
\right)^{1/2}
\right\|_{L^{\rho_p}(\Omega)}
\leq
C
\bigl(
1+\|\f u\|_{L^p(\Omega ;L^\infty(0,T;L^2(\D)))}
\bigr)\\
\cdot
\bigl(
1+\|\f\varphi\|_{L^2(\Omega ;L^2(0,T;H^1_\sigma(\D)))}
+\|B\|_{L^2(\Omega; L^2(0,T;L_2(\mathfrak U;H^1_{\sigma}(\D))))}
\bigr).
\end{multline}
Consequently, for $p=2$ the corresponding stochastic integrals
are continuous uniformly integrable martingales, and for $p>2$
they are continuous $L^{\rho_p}$-martingales.
\end{lemma}
\begin{proof}
The assertion follows from Remark~\ref{rem:test-H1}, the embedding
$H^2(\D)\hookrightarrow L^\infty(\D)$ for $d\leq3$, and
H\"older's inequality with
\[
\frac1{\rho_p}=\frac1p+\frac12.
\]
The martingale assertions follow from~\eqref{eq:BDG-general}.
\end{proof}

\begin{remark}
    We use the usual abuse of notation for Young measures, where the Young measure $\mu$ applied to the function $\lambda_{\f u} \mapsto \lambda_{\f u}\otimes \lambda _{\f u}$ is denoted by the dummy variable $ \lambda_{\f u}$, \textit{i.e.,} 
$$\langle \mu ,| \lambda_{\f u} |^2 \rangle = \int_{E } |v |^2 \de \mu(v)\, ,$$
    where $E =L^r(0,T; L^r_{\sigma}(\D)) $.
\end{remark}

\begin{definition}[Strong measure-valued  solutions to Navier--Stokes]
\label{def:measstrNav}
Let $\f u_0 \in L^4(\Omega;L^2_{\sigma}(\D))$. 
 An $(\mathcal{F}_t)$-progressively measurable Young measure {in the sense of Definition~\ref{def:progmeasYoung}} on the original stochastic basis 
\begin{align}
    \mu_\omega  \in \mathcal{P}(L^2(0,T; L^2_{\sigma}(\D)))  \quad \P\text{-almost surely,}
\end{align} 
 is called a probabilistically strong, measure-valued, or in short strong-measure-valued, solution to~\eqref{eq:stochastic_euler_additive}   with  $\nu>0$ and the initial data $\f{u}_0$ provided 
\begin{enumerate}[label=\roman*)]
\item the velocity field $\f u := \langle \mu , \lambda_{\f u}\rangle $ is $(\mathcal{F}_t)$-progressively measurable with 
$$
\f{u} \in L^4_{w^*}\left(\Omega; 
L^\infty\left(0, T ; L_{\sigma}^2\left(\D \right)\right) \cap L^2\left(0, T ; W_{\sigma}^{1,2}\left(\D \right)\right) \right)
$$
such that $\f u \in \C_w([0,T];L^2_\sigma(\D))$ $\mathbb P$-almost surely; 

\item   the second moment of the Young measure $\mu$ is progressively measurable and  it holds
$$\langle \mu,    \lambda_{\f u} \otimes \lambda_{\f u}    \rangle 
 \in L^{2}(\Omega; L^{4/3}(0,T;L^2(\D;\R^{d\times d} )) \cap 
L^2(\Omega;L^\infty(0,T;L^1(\D;\R^{d\times d})))\,;$$ 
 \item For all test  processes in the sense of Definition~\ref{def:stoch} 
     the following equation holds \( \mathbb{P} \)-almost surely for all $t\in [0,T]$ 
     with $\f u(0)=\f u_0$:
 \begin{equation}\label{eq:defMeasweakNav}
            \begin{aligned}
      \int_{\D} \f{u}\cdot \boldsymbol{\varphi} \, dx \Big|_{0}^{t}-&\int_{0}^t  \left[
    \int_{\D} \left \langle \mu ,  \left[ \lambda_{\f{u}} \otimes \lambda_{\f{u}} \right] : \nabla \boldsymbol{\varphi} \right \rangle  \, \de x \,\right] +\nu \alpha  \int_{\partial\D} \f u  \cdot \f \varphi_{\tau} \de S \de \tau  \\&+ \int_0^t \int_{\D} \nu (\nabla \f u)_{\sym} : \nabla \f \varphi - A \cdot \f{u} - \frac{1}{2} \f u \cdot [\mathcal{P}(\sigma^2 \cdot \nabla )]^2 \f \varphi \de x \de \tau 
    \\
   & ={}  \int_0^t  
       \int_{\D} \sigma ^1 \cdot \boldsymbol{\varphi} - [\mathcal{P}(\sigma^2 \cdot \nabla )] \f \varphi \cdot \f u  + B \cdot \f{u} \, \de x \de W (\tau)
       \\&\quad + \int_0^t\Tr{
       \int_{\D} \sigma ^1 \cdot B - \f u \cdot [\mathcal{P}(\sigma^2 \cdot \nabla )] B 
       \de x} \de \tau ,
    \end{aligned}
        \end{equation}
    \item the energy inequality 
    \begin{multline}\label{eq:defMEasinNav}
           \frac{1}{2}\int_{\D}\langle \mu , |\lambda_{\f{u}}|^2 \rangle \de x \Big|_{s}^{t}  +\nu \int_s^t  \int_{\D} |(\nabla \f u)_{\sym}|^2 \de x + \alpha \int_{\partial \D} | \f u|^2 \de S   \de \tau
  \\ \leq {} \int_s^t  \int_{\D}  \f u \cdot \sigma^1 \de x  dW({\tau}) + \frac{1}{2}\int_s^t
 \|\sigma^1 \|_{L_2(\mathfrak{U}, L^2_\sigma(\D))}^2  
  \de \tau .
    \end{multline}
    holds for almost all $0\leq s <t\leq T $ and \( \mathbb{P} \)-almost surely with $\int_{\D}\langle \mu_0 , |\lambda_{\f{u}}|^2 \rangle \de x = \| \f u_0\|_{L^2(\D)}^2$.
\end{enumerate}

  \end{definition}

The following definition and theorem represent the first key point of this work. The introduced concept of solutions combines the usual existence of probabilistically weak analytically weak solutions for the stochastic Navier--Stokes equations with the observation that probabilistically strong solutions can be obtained in the measure-valued sense. The construction presented in this work combines both of these approaches, thereby constructing probabilistically strong, measure-valued solutions, that in law, in a sense to be made precise, coincide with a weak-weak solution.

\begin{definition}[Strong-measure-valued--weak-weak solutions to Navier--Stokes] \label{def:allNav}

For an initial value  $\f u_0\in L^{4}\left(\Omega;L_{\sigma }^2\left(\D\right)\right)$, we call 
the tuple consisting of a new stochastic basis and an associated process
$$
\left(\left(\tilde\Omega, \tilde{\mathcal{F}},\left(\tilde{\mathcal{F}}_t\right), \tilde{\mathbb{P}}\right), \tilde {\f{u}}, \tilde W\right)
$$
as well as a Young measure on the original stochastic basis 
\begin{align*}
    \mu_\omega  \in \mathcal{P}(L^2(0,T; L^2_{\sigma}(\D)))  \quad \P\text{-almost surely,} 
\end{align*} 
probabilistically strong, measure-valued, and probabilistically weak, weak solution
to~\eqref{eq:stochastic_euler_additive}  with the initial data $\f u_0$ provided $\tilde{\f u}$ is a weak-weak solution on the new stochastic basis in the sense of Definition~\ref{def:weakweakNav} and $\mu$ is a measure-valued solution on the given stochastic basis in the sense of Definition~\ref{def:measstrNav} and both are connected via:
    for all $ \Phi \in \C_b(L^2(0,T; L^2_{\sigma}(\D))) $ 
    it holds 
    \begin{equation}\label{eq:idenstrongweak}
        \expect{\langle \mu_{\omega} , \Phi(\lambda_{\f u}) \rangle } = \tilde{\mathbb{E}}\left[\Phi(\tilde{\f u}(\tilde{\omega})) \right] 
        \,.    
    \end{equation}

  \end{definition}
\begin{theorem}[Existence of solutions to Navier--Stokes]\label{thm:exMeasNav}
Let Assumption~\ref{Ass:1} be fulfilled as well as $\f {u}_0\in L^p ({\Omega}; L_{\sigma }^2\left(\D\right))$ for some $p>4$.
    Then there exists a probabilistically strong, measure-valued, and probabilistically weak, weak solution 
    in the sense of Definition~\ref{def:allNav} such that in addition 
$$ \begin{aligned}
    \f u &\in 
L^{p}_{w^*}(\Omega; L^\infty(0,T; L^2_{\sigma}(\D))\cap L^2(0,T;H^1_{\sigma}(\D)))
        \cap L^{p/2}(\Omega;W^{\alpha,p/2}(0,T;(H^2(\D)\cap H^1_{\sigma}(\D))^*))\\
        \tilde{\f u} &\in 
L^{p}_{w^*}(\tilde\Omega; L^\infty(0,T; L^2_{\sigma}(\D))\cap L^2(0,T;H^1_{\sigma}(\D)))
        \cap L^{p/2}(\tilde\Omega;W^{\alpha,p/2}(0,T;(H^2(\D)\cap H^1_{\sigma}(\D))^*)) 
        \end{aligned}$$ 
        for any  
       $\alpha <\frac{1}{2}$ as well as 
       $$\mu_\omega  \in 
         \mathcal P(L^{s} (\D\times (0,T);\R^d))  \quad \text{with }s < 
         2(d+2)/d \text{ $\P$-almost surely}
         \,.$$
\end{theorem}

\begin{proof}
We structure the  proof of existence of these solutions to the Navier--Stokes equations in different steps. For the proof of existence of weak-weak solutions in the sense of Definition~\ref{def:weakweakNav} we mainly rely on the result and proof of~\cite{FlandoliGatarek} and choose the same approximate problem, namely a Galerkin scheme based on Eigenfunctions of the Stokes operator. 

\textit{Step 1: Discrete setting and existence.}
   Consider the Stokes operator with the Navier-slip with friction and impermeability boundary conditions, \textit{i.e.,}  the operator associated to the map $ \f h \mapsto (\f w , p)$ for the following elliptic problem, 
    \begin{align*}
-\Delta \f w + \nabla p  = \f h \text{ in } \D \, , \quad 
\nabla \cdot  \f w = 0 \text{ in } \D \,,\quad
\f w \cdot \f n = 0 \text{ and } ((\nabla \f w)_{\sym} \f n + \alpha \f w)_{\tau} = 0  \quad\text{on  } \partial\D \,.
\end{align*} 
For any $\alpha>0$ 
it holds~\cite[Thm.~2.1]{StokesRegularity} 
$$
\| \f w \|_{H^2(\D)} + \| p \|_{H^1(\D)} \leq C \| \f h \| _{L^2(\D)}\,. 
$$
The associated Stokes operator $\A$ is self-adjoint and the inverse of a compact operator, therefore there exists a sequence of eigenfunctions~\cite[Section~9]{brezis}
 \( \{\f v _k\}_{k=1}^\infty \subset  \DA \subset 
     L^2_{\sigma}(\mathcal{D}) 
    \). 
This basis can be chosen to be orthonormal in $L^2_\sigma(\D)$ and additionally orthogonal in $H^1_{\sigma}(\D)$~\cite[Sec.~6]{stokes2}.
For each \( N \in \mathbb{N} \), define the Galerkin subspace
\[
V_N := \text{span} \{ \f v _1, \dots, \f v _N \}
\]
such that $\mathrm{clos}_{\| \cdot\|_{H^1(\D)}} \left( \bigcup_{N=1}^\infty V_N\right)= H^1_{\sigma}(\D)$.
We define the usual $L^2$-projection on the finite dimensional spaces via 
\begin{align*}
P_N : L^2_{\sigma}(\D)  \longrightarrow V_N\,,\quad P_N \f v = \sum_{i=1}^N \int_{\D} \f v \cdot  \f v _i\de x   \f v _i\,. 
\end{align*}
We observe by the construction of the projection and the basis that $ \| \A P_N \f v \|_{L^2(\D)} \leq \| \A  \f v \|_{L^2(\D)}$.  There exists a constant $c>0$ such that for all $N\in \N$ and $\f v \in 
\DA
$ such that 
\begin{equation}
\|P_N \f v\|_{\f H^2} \le c \|\f v\|_{\f H^2} \, \label{PnH2}
\end{equation}
(cf.%
~\cite[Theorem~4.5]{StokesRegularity} together with~\cite[Sec.~6.]{stokes2}). 
For this property the regularity assumption $\partial \D \in \C^2 $ (or rather $\partial \D \in \C^{1,1}$) is essential.

We seek an approximate solution \( u^N(t) \in V_N \) of the form
\[
u^N(t,x) := \sum_{k=1}^N a_k^N(t) \f v _k(x),
\]
where the coefficients \( a_k^N(t) \) satisfy the following finite-dimensional system:

\begin{equation} \label{eq:galerkin}
\begin{aligned}
\de \int_{\D}  u^N(t)\cdot \varphi_N \de x  
&+ \int_{\D} (u^N(t) \cdot \nabla) u^N(t)\cdot  \varphi_N + \nu (\nabla u^N(t) )_{\sym}: \nabla \varphi_N \de x + \nu \alpha \int_{\partial \D} u^N \cdot \varphi_N\de S  \de t \\
&+\int_{\D} \frac{1}{2} \mathcal P (\sigma^2\cdot \nabla )u^N (t)\cdot \mathcal P(\sigma^2\cdot \nabla ) \varphi_N \de x \de t 
\\&= \int_{\D} (\sigma^1\cdot  \varphi_N -  u^N(t) \cdot  (\sigma^2  \cdot \nabla)\varphi_N)    \de x \de{W}(t) , 
\end{aligned}
\end{equation} 
for all $ \varphi_N \in V_N$,
with initial data
$
u^N(0) = P_N u_0 
$.
Rewriting the discrete equation~\eqref{eq:galerkin} via the representation  \( u^N(t,x) = \sum_{k=1}^N a_k^N(t) \f v _k(x) \in V_N \) satisfies the finite-dimensional SDE:
\begin{equation} \label{eq:sde_finite}
\de  a_j^N(t) = - \sum_{i,k=1}^N b_{i,k,j} \, a_i^N(t) a_k^N(t) \, \de t - \sum_{i=1}^N d_{i,j} \, a_i^N(t)  \, \de t 
+ \sum_{\ell=1}^\infty\left (\eta_{j,\ell}+ \sum_{i=1}^N \zeta_{j,\ell,i}a_i^N(t)\right) \,  \de \beta _\ell(t), 
\end{equation}
for $j\in \{ 1,\ldots,N\}$
with initial condition
\[
a_j^N(0) = \int_{\D} u_0\cdot  \f v _j \de x , 
\]
where
 \( b_{i,k,j} := \int_{\D} (\f v _i \cdot \nabla) \f v _k\cdot \f v _j \de x \), \( d_{i,j} := \int_{\D} \nu (\nabla \f v _i)_{\sym} :\nabla \f v _j + \frac{1}{2}\mathcal{P}(\sigma^2\cdot\nabla )\f v _i \cdot \mathcal P(\sigma^2\cdot\nabla)\f v_j  \de x + \alpha \nu \int_{\partial \D
 } \f v_i \cdot \f v_j \de S \) are the structure coefficients of the drift terms,
 \( \{\beta _\ell(t)\}_{\ell=1}^\infty \) are independent real-valued Wiener processes,
  \( (\sigma^1,\sigma^2) \in L_2(\mathfrak{U}, L^2_\sigma(\mathcal{D})\times H^2(\D)\cap H^1_{\sigma}(\D)) \), and \( \eta_{j,\ell} := \int_{\D} (\sigma^1\cdot\mathbf e_{\ell})\cdot  \f v _j \dx  \) as well as $ \zeta_{j,\ell,i} :=\int_{\D} ((\sigma^2\cdot\mathbf e_\ell) \cdot\nabla) \f v _i\cdot  \f v _j \dx  $ are the matrices of  the noise operators.
This system is a standard SDE in \( \mathbb{R}^N \) with locally Lipschitz drift and   noise.
By classical results (e.g., in~\cite[Thm.~3.1.1]{SPDEBook}), it admits a unique local-in-time  solution
\( a^N(t) = (a_1^N(t), \dots, a_N^N(t)) \in \mathbb{R}^N \), and thus a unique approximate velocity field \( u^N(t) \in V_N \). The \textit{a priori} estimates and standard arguments guarantee that these local-in-time solutions can be extended to a global-in-time solution $\P$-almost surely. 

\textit{Step 2: a priori estimates.}
Applying Itô's formula to the \( L^2 \)-norm \( \|u^N(t)\|_{L^2(\D)}^2 \), we obtain:
\begin{align}\label{eq:ito_energy}
\begin{split}
    \frac{1}{2}\|u^N(t)\|_{L^2(\D)}^2 &+\nu\int_0^t \|(\nabla u^N(s))_{\sym}\|_{L^2(\D)}^2+ \alpha \| u^N \|_{L^2(\partial\D)}^2 \de s 
- \frac{1}{2}\|u^N(0)\|_{L^2(\D)}^2 \\
&=  \int_0^t \int_{\D} \sigma^1 \cdot  u^N(s) \de x \cdot  \de W(s)  
 + \frac{1}{2}\int_0^t \|\sigma^1\|_{L_2(\mathfrak{U}, L^2_\sigma(\D))}^2 \, \de s
 \\
 & \quad  - \frac{1}{2}\int_0^t\int_{\D}| \mathcal P (\sigma^2\cdot \nabla )u^N (s)|^2 \de x \de s - \frac{1}{2}\int_0^t\int_{\D} [\mathcal P (\sigma^2\cdot \nabla )]^2u^N (s) \cdot u^N (s)   \de x \de s
 \\
 &\quad + \int_0^t \Tr{\int_{\D} ( \sigma^2 \cdot \nabla ) u^N (s)\cdot  \sigma^1 \de x }\de s 
 \\
 &=  \int_0^t \int_{\D} \sigma^1 \cdot  u^N(s) \de x \cdot  \de W(s)  
 + \frac{1}{2}\int_0^t \|\sigma^1\|_{L_2(\mathfrak{U}, L^2_\sigma(\D))}^2 \, \de s
\, ,
\end{split}
\end{align}
where we used the cancellation of the Itô-Stratonovich correction term with  the additional term from the Itô formula after an integration-by-parts using that $\sigma^2$ is divergence free, \textit{i.e.}, 
\begin{align*}
    -\frac{1}{2}\int_0^t\int_{\D}| \mathcal P (\sigma^2\cdot \nabla )u^N (s)|^2 \de x \de s =\frac{1}{2}\int_0^t\int_{\D} [\mathcal P (\sigma^2\cdot \nabla )]^2u^N (s) \cdot u^N (s)   \de x \de s\,.
\end{align*}
Moreover, we use
that it holds for two noise operators $ \Tr{\int_{\D} ( \sigma^2 \cdot \nabla ) \sigma^1\cdot u^N (s) \de x}= 0$
by assumption.
Note that the nonlinear convective term vanishes due to skew-symmetry.

For  \( p >4\), we are taking the $p/2$-th
 power of~\eqref{eq:ito_energy}. Using that there are only nonnegative terms on the left-hand side, we find after taking expectations 
\begin{equation}
\begin{aligned}
\mathbb{E} &\left[ \sup_{t \in [0,T]} \|u^N(t)\|_{L^2(\D)}^{p} \right] + \nu^{p/2}\expect{ \|(\nabla u^N )_{\sym}\|_{L^2(\D\times (0,T))}^{p} +\alpha^{p/2} \| u^N \|_{L^2(\partial \D\times (0,T))}^p }
\\&\leq C \left( \expect{\|u_0\|_{L^2(\D)}^{p} }
+ \mathbb{E} \left[ \sup_{t \in [0,T]} \left| \int_0 ^t \int_{\D} \sigma^1 \cdot u^N  \de x \de W( s)  \right|^{p/2} \right] \right. \\
&\qquad \left. + \mathbb{E} \left[ \left( \int_0^t \|\sigma^1\|_{L_2(\mathfrak{U}, L^2(\D))}^2 \, \de t  \right)^{p/2} \right] \right)
\end{aligned}
\end{equation}
 and by the BDG inequality:
\begin{equation*}
\begin{aligned}
    \mathbb{E} \left[ \left( \sup_{t \in [0,T]} \int_0^t \int_{\D} \sigma  ^1 u^N \de x \de W(s)  \right)^{p/2} \right]
&\leq C_p \, 
\expect{\left(
\int_0^T \Tr{\left(\int_{\D} \sigma^1 \cdot u^N \de x\right)^2 }\de t \right)^{p/4}
}
\\
&\leq C_p \, \|\sigma^1\|_{L_2(\mathfrak{U}, L^2(\D))}^{p/2} 
\, \mathbb{E} \left[ \left( \int_0^T \|u^N \|_{L^2(\D)}^2 \, ds \right)^{p/4} \right].
\end{aligned}
\end{equation*}
We note the coercivity estimate~\cite[Prop.~3.13]{StokesRegularity}: there exists $\eta >0$ such that for all $ u \in H^1_\sigma(\D)$ it holds 
\begin{equation}\label{eq:coercivityNavier}
    \|(\nabla u )_{\sym}\|_{L^2(\D\times (0,T))}^{2} + \| u \|_{L^2(\partial \D\times (0,T))}^2 
\geq \eta \| \nabla u\| _{L^2(\D\times (0,T))} ^2 \,.
\end{equation}

Using Young's inequality and Gronwall-type arguments, one then obtains:
\begin{equation}\label{aprioribounds}
\mathbb{E} \left[ \sup_{t \in [0,T]} \|u^N(t)\|_{L^2(\D)}^p \right]+ \nu^{p/2}\expect{ \|\nabla u^N \|_{L^2(\D\times (0,T))}^p }
\leq C(p,T, \|u_0\|_{L^p(\Omega;L^2(\D))}, \|\sigma^1\|_{L_2(\mathfrak{U}, L^2(\D))}),
\end{equation}
with a constant \( C \) independent of \( N \). This shows that the previously constructed local solution can be extended to a global one. 

In order to infer an estimate of $ u^N$ in a Sobolev--Slobodeckij space, we choose $ \varphi_N= P_N \varphi$ for $\varphi\in \DA
$ in~\eqref{eq:galerkin} and estimate using Fubini's theorem~\cite[Thm.~4.18]{DaPratoZabczyk}
\begin{align*}
\begin{split}
        \int_{\D}& (u^N(t)-u^N(s)) \cdot P_N\varphi \de x\\
    = {}& -\int_s^t   \int_{\D} (u^N \cdot \nabla) u^N\cdot  P_N\varphi + \nu (\nabla u^N )_{\sym} : \nabla P_N\varphi
    +\frac{1}{2} \mathcal{P}(\sigma^2\cdot \nabla) u^N \cdot \mathcal{P}(\sigma^2\cdot \nabla)P_N\varphi 
    \de x  \de \tau  \\
{}& +  \int_s^t\int_{\D} \sigma^1  \cdot  P_N\varphi - u^N \cdot ( \sigma^2 \cdot \nabla ) P_N\varphi  \de x \de{W}(\tau ) - \nu \alpha \int_s^t \int_{\partial \D} u^N \cdot P_N \varphi \de S \de \tau 
\\
\leq{}& \int_s^t \| u^N \|_{L^2(\D)}\|\nabla u^N\|_{L^2(\D)}\| P_N\varphi \|_{L^\infty(\D)} + \nu \alpha \| u^N \|_{L^2(\partial\D)} \| P_N \varphi \|_{L^2(\partial\D)} \de \tau    \\&+\int_s^t\left( \nu +\frac{1}{2}\| \sigma^2\|_{L_2(\mathfrak{U}; H^2(\D))}^2  \right)  \|\nabla u^N\|_{L^2(\D)}\|\nabla P_N\varphi\|_{L^2(\D)} \de \tau 
\\ 
& +   \left \|\int_s^t\sigma^1\de W \right \|_{ L^2(\D)} \| P_N \varphi \|_{L^2(\D) } \\&+ 
\left \| \int_s^tu^N\otimes \sigma^2\de W \right \|_{L^2(\D)} \| \nabla P_N \varphi \|_{L^2(\D) } \,.
\end{split}
\end{align*}
Due to the stability of the projection, we conclude 
\begin{align}\label{eq:timeest}
\begin{split}
     \| u^N(t)&-u^N(s)\|_{(\DA)^*}\\
\leq{}& 
C\|u^N \|_{L^\infty(0,T;L^2_{\sigma}(\D))} \| u^N\|_{L^2(0,T;H^1_{\sigma}(\D))}  
(t-s)^{1/2} 
\\ &+ C \left( \nu +\frac{1}{2}\| \sigma^2\|_{L_2(\mathfrak{U}; H^2(\D))}^2 + \nu \alpha C \right) \| u^N\|_{L^2(0,T;H^1_{\sigma}(\D))} (t-s)^{1/2} 
\\
& + C  \left \|\int_s^t\sigma^1\de W \right \|_{ L^2(\D)} 
+ C 
 \left \| \int_s^t u^N \otimes \sigma^2\de W \right \|_{ L^2(\D)} 
 \,.
\end{split}
\end{align}

Using the fact that the stochastic integrals are known to be bounded in a Sobolev--Slobodeckij space~\cite[Lem.~2.1]{FlandoliGatarek}, \textit{i.e.,}
$$
\begin{aligned}
    \left\| \int_0^\cdot \sigma ^1  \de W \right\|_{L^q(\Omega;W^{\gamma,q}(0,T; L^2_{\sigma}(\D)))}^q &\leq  C 
\left \| \sigma^1 \right\|_{
L_2(\mathfrak{U};L^2_\sigma(\D))}^q
\\
 \left\| \int_0^\cdot
  u^N\otimes 
 \sigma ^2  \de W \right\|_{L^q(\Omega;W^{\gamma,q}(0,T; L^2(\D)))}^q 
& \leq  C \expect{\int_0^T\| u^N \otimes \sigma^2 \|_{L^2(\D)}^q \de t  }\\
 &\leq  C\| u^N\|_{L^q(\Omega; L^q(0,T;L^2_{\sigma}(\D)))}^q
 \left \| \sigma^2 \right\|^q_{
L_2(\mathfrak{U};H^2(\D))} 
\end{aligned}
$$
for $\gamma < 1/2$ and any $q\in [2,\infty)$. 
Dividing~\eqref{eq:timeest} by $|t-s|^{1/q+\gamma}$, taking the $q$-th power, integrating over $t $ and $s$, 
and taking 
expectations let us conclude that there exists a constant $C>0$ such that
\begin{multline*}
    \| u^N \|_{L^{q}(\Omega; W^{\gamma,q}(0,T;(\DA)^*))} \\
    = \expect{\left[  \int_0^T \int_0^T \frac{\| u^N(t)-u^N(s)\|_{(H^2(\D)\cap H^1_{\sigma}(\D))^*}^q}{|t-s|^{1+\gamma q}}\de s\,\de t
    \right]}^{{1}/{q}}
    \\ \leq  C (\|u^N \|_{L^{2q}(\Omega;L^\infty(0,T;L^2_{\sigma}(\D)))}+1)( \| u^N\|_{L^{2q}(\Omega;L^2(0,T;H^1_{\sigma}(\D)))} +1)  \,
\end{multline*}
for any $\gamma < 1/2 $ and $q=p/2>2$. We note that the right-hand side is bounded due to~\eqref{aprioribounds}.  
By~\cite[Thm.~2]{FlandoliGatarek}, we deduce the continuous embedding into $\C([0,T];(\DA)^*)$ that becomes a compact embedding into $\C([0,T];(H^3(\D)\cap \DA )^*)$. 
Finally, by the Gagliardo--Nirenberg inequality, we find
\begin{align*}
L^\infty(0,T;L^2_{\sigma}(\D))\cap L^2(0,T;H^1_{\sigma}(\D)) \hookrightarrow L^{\frac{2(d+2)}{d}} (\D\times (0,T)) \cap L^{8/3}(0,T;L^4(\D)) 
\end{align*}
so that there exists a constant $C>0$ fulfilling
\begin{align*}
\| u^N\|_{L^{p}(\Omega; L^{\frac{2(d+2)}{d}} (\D\times (0,T)) \cap L^{{8}/{3}}(0,T;L^4(\D;\R^d)) )} \leq C \,.
\end{align*}
\textit{Step 3: Compactness and convergence to a martingale solution.}
Using the above \textit{a priori} bounds, it is now a standard matter to construct via compact embeddings and a version of the Skorohod representation theorem~\cite{Jakubowski_2} or~\cite[p.~9] {EmbeddingCompact} a weak martingale solution in the sense of Definition~\ref{def:weakweakNav}. 
We only note that after changing the stochastic basis to 
$$
\left(\left(\tilde\Omega, \tilde{\mathcal{F}},\left(\tilde{\mathcal{F}}_t\right), \tilde{\mathbb{P}}\right), \tilde {\f{u}}, \tilde W\right)
$$
and the subsequence $ \tilde{u}^N$ such that $ \tilde{u}^N$ has the same law as $ u^N$,  we infer the convergence 
\begin{equation}\label{eq:strongconv}
    \tilde{ u}^N \to \tilde{\f u}  \quad \text{ in }  L^2(0,T;L^2_\sigma(\D))\cap \C([0,T];(H^3(\D)\cap H^1_{\sigma}(\D))^*)
 \end{equation}
$\tilde{\mathbb P}$-almost surely.
Observing that the $ L^2(0,T;H^1_{\sigma}(\D))$-norm is lower semicontinuous with respect to the strong topology in $L^2(0,T;L^2_\sigma(\D))$, we find by the 
 $\tilde{\mathbb P}$-almost sure-convergence and Fatou's lemma 
 that
$$
\begin{aligned}
    \int_{\tilde\Omega}\left[\int_0^T \| \tilde{\f u}\|_{H^1
     _{\sigma}
    (\D)
    }^2\de t \right]^{p/2}\,\de\tilde{\mathbb{P}}(\tilde\omega)
&\leq \liminf_{N\to \infty} \int_{\tilde\Omega}\left[ \int_0^T \| \tilde{u}^N\|_{H^1_{\sigma}(\D)}^2\de t \right] ^{p/2} \,\de\tilde{\mathbb{P}}(\tilde\omega)\\
& = \liminf_{N\to \infty} \int_{\Omega}\left[ \int_0^T \| {u}^N\|_{H^1_{\sigma}(\D)}^2\de t \right]^{p/2}\,\de{\mathbb{P}}(\omega)\leq C \,.
\end{aligned}
$$
Furthermore, we observe that the $L^\infty(0,T;L^2_\sigma(\D))$-norm is lower semicontinuous as it is the supremum of lower semicontinuous functionals $\| \cdot \| _{L^\infty(0,T;L^2_\sigma(\D))} = \lim_{p\to \infty}\| \cdot \| _{L^p(0,T;L^2_\sigma(\D))} $ on a finite-measure domain, with possible normalization factors. With the same argument as above, we find now that $ \tilde{\f u}\in L^p_{w^*}(\tilde \Omega; L^\infty(0,T;L^2_\sigma(\D)))$. 
Together with Lemma~\ref{lem:weakcont} this implies item~\ref{reg:weakweak} in Definition~\ref{def:weakweakNav}.  
The pointwise a.e.~in $ (0,T)$ and a.s.~$\tilde \Omega$ convergence with respect to the strong convergence in $L^2_\sigma(\D)\hookrightarrow ( \DA)^*$ together with Fatou's lemma implies
\begin{multline*}
    \int_{\tilde{\Omega}} 
    \int_0^T \int_0^T \frac{\| \tilde{\f u}(t)-\tilde{\f u}(s)\|_{(\DA)^*}^q}{|t-s|^{1+\gamma q}}\de s\,\de t
    \de \tilde{\mathbb{P}} (\tilde \omega) 
    \\
    \leq \liminf_{N\to \infty}
    \int_{\tilde{\Omega}} 
    \int_0^T \int_0^T \frac{\| u^N(t)-u^N(s)\|_{(\DA)^*}^q}{|t-s|^{1+\gamma q}}\de s\,\de t
    \de \tilde{\mathbb{P}} (\omega) \,.
\end{multline*}
In the following, we will focus on the construction of measure-valued-strong solutions in the sense of Definition~\ref{def:measstrNav}. 

\textit{Step 4: Convergence to a measure-valued solution.} 
From the \textit{a priori} estimates and Theorem~\ref{thm:Young}, we may extract a converging subsequence, which we do not relabel such that 
\begin{align}
     u^N &\stackrel{*}{\rightharpoonup} \f u \quad && \text{in } L^{p}_{w^*}\left(\Omega; L^\infty(0,T; L^2_{\sigma}(\D))\cap L^2(0,T;H^1_{\sigma}(\D))\right)
      \,,\label{weak:conv}\\
       u^N &\stackrel{*}{\rightharpoonup} \f u \quad && \text{in }L^{p/2}(\Omega;W^{\gamma,p/2}(0,T;(\DA)^*)) \quad \text{with }\gamma <\frac{1}{2}\,,\label{weak:conv2}
    \\
   \delta_{u^N
   }&\stackrel{*}{\rightharpoonup} \mu  \quad && \text{in } \mathcal M(\Omega; \mathcal{P}(L^{s} (\D\times (0,T); \R^d)) )\quad \text{with }s < 
   2(d+2)/d
   \,. \label{weak:convE}
\end{align}
We note that $p>4\geq 2(d+2)/d$ for $d\in \{2,3\}$. 
Define $ E = L^{s}(\D \times (0,T))\cap L^2 (0,T; L^2_{\sigma}(\D))$ for a certain $s <    2(d+2)/d$ and 
$
\mathbb X
:=
L^2(0,T;H^1_{\sigma})
\cap
W^{\gamma,q}(0,T;(\DA)^*)$.
From the usual Aubin--Lions compactness criterion~\cite{simon} it follows that $ \mathbb X $ is compactly embedded into $ E$. 
Due to the uniqueness of the weak limit, we have $\langle \mu , I \rangle = \f u $, where the limits inherit the $(\mathcal{F}_t)$-progressive measurability according to Theorem~\ref{thm:Young}.

Now we consider $ u^N$ fulfilling the stochastic Galerkin formulation~\eqref{eq:galerkin} and $\varphi$ a  stochastic test process  in the sense of Definition~\ref{def:stoch}. Then, we apply the finite dimensional Itô formula to the product $\int_{\D}  u^N \cdot P_m \varphi \de x $ for $m\leq N$ to infer that
\begin{equation}\label{eq:disweak}
\begin{aligned}
    \int_{\D}   u^N \cdot P_m\varphi \de x \Big|_s^t = {}&\int_s^t \int_{\D} u^N \otimes  u^N : \nabla P_m \varphi -\nu (\nabla  u^N)_{\sym}:\nabla P_m \varphi\de x 
    + \alpha \nu \int_{\partial \D}  u^N \cdot P_m \varphi 
    \de \tau \\& +\int_s^t \int_{\D}  u^N \cdot P_m A -\frac{1}{2} [\mathcal{P}(\sigma^2\cdot\nabla)]u^N [\mathcal{P}(\sigma^2\cdot\nabla)] P_m\varphi   \de x \de \tau
    \\
     &+ \int_s^t \int_{\D} \sigma^1 \cdot  P_m \varphi - u^N \cdot (\sigma^2 \cdot\nabla ) P_m\varphi   +  u^N \cdot P_mB  \de x   \de W(\tau)  \\& + \int_s^t 
     \Tr{\int_{\D} 
     \sigma^1\cdot P_m B - u^N \cdot (\sigma^2\cdot\nabla) P_m B 
     \de x} \, \de \tau 
\end{aligned}
\end{equation}
 for all $m\leq N$ and all $s<t\in [0,T]$. 
As in~\eqref{eq:ito_energy}, we observe the energy equality 
\begin{equation} \label{eq:energydis}
\begin{aligned}
\frac{1}{2}\|u^N(t)\|_{L^2(\D)}^2 +\nu\int_s^t \|(\nabla u^N)_{\sym}\|_{L^2(\D)}^2+ \alpha \| u^N\|_{L^2(\partial \D)}^2\de \tau 
- \frac{1}{2}\|u^N(s)\|_{L^2(\D)}^2 \\
=  \int_s^t \int_{\D} \sigma^1 \cdot   u^N \de x \,\de W(\tau)
 + \frac{1}{2}\int_s^t \|\sigma^1\|_{L_2(\mathfrak{U}, L^2_\sigma(\D))}^2 \, \de\tau 
\end{aligned}
\end{equation}
for all $s<t\in [0,T]$, where we used again that the two noise operators  $\sigma^1$ and $\sigma^2$ are orthogonal in the sense made precise in Section~\ref{sec:2}.
Using Lemma~\ref{lem:invar}, 
we infer
\begin{equation}\begin{split}\label{eq:Gal}
    -&\expect{\psi\int_0^T \partial_t \phi\left[  \int_{\D} u^N \cdot P_m\varphi \de x \right]\de t } - \expect{ \psi\int_0^T \phi \int _{\D}  u^N \otimes  u^N : \nabla P_m\varphi \de x \de t} \\
    &+ \expect{ \psi\int_0^T \phi \int _{\D}\nu  (\nabla u^N)_{\sym}:\nabla P_m \varphi - P_m  A \cdot  u^N -\frac{1}{2}u^N [\mathcal P(\sigma^2\cdot\nabla)] ^2 P_m \varphi \de x \de t} \\
  &  -\expect{\psi
     \phi(0) \int_{\D}  u^N(0) \cdot P_m \varphi \de x - \nu\alpha\psi \int_0^T \phi \int_{\partial \D} u^N \cdot P_m \varphi \de S \de t  } 
    \\
& = \expect{\psi\int_0^T \phi \int_{\D} \sigma^1 \cdot P_m\varphi - u^N \cdot (\sigma^2 \cdot \nabla) P_m\varphi + P_m B \cdot u^N \de x    \de W(t) }  
\\& \quad + \expect{\psi
\int_0^T \phi 
 \Tr{\int_{\D} \sigma^1 \cdot P_m B - u^N \cdot (\sigma^2\cdot\nabla) P_m B 
 \de x  }\, \de t }
\end{split}\end{equation}
for all $m\leq N$, all $ \phi \in  \C^1_c([0,T))$ and all $\psi \in L^\infty(\Omega)$
as well as 
\begin{equation}\label{eq:enindis}
\begin{split}
    -&\expect{\psi\int_0^T \partial_t \phi \frac12\| u^N \|_{L^2(\D)}^2 \de t } 
    + \nu \expect{ \psi\int_0^T \phi \int _{\D}   |(\nabla u^N)_{\sym}|^2 \de x +\alpha \int_{\partial \D} |u^N|^2\de S  \de t}
    \\
& \leq\expect{ \psi\phi(0) \frac12\int_{\D}  |u^N(0)|^2 \de x } \\& \quad +  \expect{\psi\int_0^T \phi \int_{\D} \sigma^1 \cdot  u^N \de x   \de W(t) }  
+ \expect{\psi\frac{1}{2}\int_0^T \phi \|\sigma^1\|_{L_2(\mathfrak{U}, L^2_\sigma)}^2  \de t  }\,.
\end{split}
\end{equation}
for all $ \phi \in  \C^1_c([0,T);[0,\infty))$ and all $\psi \in L^\infty(\Omega;[0,\infty))$. Set $\rho_p:=\frac{2p}{p+2}>1.$ By Lemma~\ref{lem:test-stoch-integrability}, the stochastic integral in~\eqref{eq:Gal} belongs to $L^{\rho_p}(\Omega)$. Hence all expectations in~\eqref{eq:Gal} are well defined.

Since $p>4$ we have 
 $
\nabla P_m\f\varphi
\in
L^2\bigl(
\Omega;
 \C([0,T]; \DA )
\bigr)
\subset
L^{p/(p-2)}
\bigl(
\Omega;
 \C(\overline\D\times [0,T] )
\bigr),$ 
we use the convergence in the sense of Young measures in order to pass to the limit in the quadratic term 
\begin{multline*}
    \lim_{N\to \infty } \expect{\psi \int_0^T \phi\left[ \int _{\D} u^N \otimes u^N : \nabla P_m \varphi \de x 
    \right] \de t} \\
    = \expect{ \psi\int_0^T \phi\left[ \int _{\D} \left \langle \mu  , \lambda_{\f u} \otimes \lambda_{\f u}:  \nabla P_m \varphi \right \rangle  \de x 
    \right] \de t}\,. 
\end{multline*}
We note the additional regularity that this quadratic term inherits due to \textit{iii)} of Theorem~\ref{thm:Young}. 
Additionally, similar to the estimate~\eqref{est:quadtermYoung}, we infer
that $ \langle \mu , \lambda_{\f u} \otimes\lambda_{\f u}\rangle \in L^{p/2}(\Omega; L^{4/3}(0,T; L^2(
D; \R^{d\times d})))$ via 
\begin{multline*}
       \int_\Omega\int_0^T \int_{\D} \phi(\omega,t,x) \langle \mu , \lambda_{\f u} \otimes \lambda _{\f u}\rangle \de x\de t \de\mathbb P(\omega)\\
       \leq \lim_{n\to \infty}\| \phi \|_{L^{p/(p-2)}(\Omega; L^4(0,T;L^2(\D)))} \| u_n \|_{L^p(\Omega ; L^{8/3}(0,T;L^4(\D)
       ))}^2\,.
  \end{multline*}
Moreover, the stochastic integrands converge weakly in
\[
L^{\rho_p}
\bigl(
\Omega;
L^2(0,T;L_2(\mathfrak U;\R))
\bigr).
\]
Indeed, multiplication by the fixed processes $P_mB$ and
$(\sigma^2\cdot\nabla)P_m\f\varphi$ defines bounded linear maps
from $L^p(\Omega; L^\infty_tL^2_x)$ into this space. Hence we can pass to the limit in the stochastic integral in~\eqref{eq:Gal}.

Using this convergence, we may pass to the limit in \eqref{eq:Gal}  using~\eqref{weak:conv} and~\eqref{weak:convE} such that we  find
\begin{align*}
    -&\expect{\psi\int_0^T \partial_t \phi \int_{\D}\f u \cdot P_m\varphi  \de x \de t } -  \expect{\psi \int_0^T \phi \int _{\D} \left \langle \mu  , \lambda_{\f u} \otimes \lambda_{\f u}:  \nabla P_m \varphi \right \rangle \de x \de t} \\
   & +\expect{\psi \int_0^T \phi \int_{\D}\nu (\nabla 
   \f u)_{\sym}: \nabla P_m \varphi  - P_m A \cdot \f u -\frac{1}{2}\f  u \cdot [\mathcal P(\sigma^2\cdot\nabla)] ^2 P_m\varphi  \de x\de t }
    \\&-\expect{\psi  \phi(0) \int_{\D} \f  u_0 \cdot P_m  \varphi \de x  - \nu\alpha \psi \int_0^T \phi \int_{\partial \D  } \f u \cdot P_m \varphi \de S\de t   }
    \\
={}& \expect{\psi\int_0^T \phi \int_{\D} \sigma^1 \cdot P_m \varphi - 
\f u \cdot (\sigma^2\cdot\nabla)P_m\varphi  + P_m B \cdot \f u  \de x   \de W(t) }  \\
&\quad + \expect{\psi\int_0^T \phi 
\Tr{\int_{\D} \sigma ^1\cdot  P_mB - \f u\cdot (\sigma^2\cdot\nabla )P_m B
\de x}\, \de t }
\end{align*}
for all $m \in  N$, all $ \phi \in  \C^1_c([0,T);[0,\infty))$, all $\psi \in L^\infty(\Omega;[0,\infty))$ and all test processes in the sense of Definition~\ref{def:stoch}. Thus, due to the choice of the projection, we know  
$ P_m \f v  \to \f v $   in $ H^1_{\sigma}(\D)$ for all $ \f v \in H^1_\sigma(\D)$ such that
 especially $ P_m \varphi \to \varphi $ in $L^4(0,T;H^1_\sigma(\D))$,
we may pass to the limit in the previous inequality with $m\to \infty$. 
Lemma~\ref{lem:invar} now implies~\eqref{eq:defMeasweakNav}. 

Passing to the limit in~\eqref{eq:enindis}, we find by using lower semicontinuity in the second term that
\begin{align*}
    -&\expect{\psi\int_0^T \partial_t \phi \int_{\D} \langle \mu , |\lambda_{\f u} |^2 \rangle \de t } 
    + \nu \expect{\psi \int_0^T \phi  \int _{\D} |(\nabla\f u)_{\sym}|^2 \de x + \alpha \int_{\partial \D} | \f u|^2\de S \de t}
    \\
& \leq 
\expect{\psi \phi(0) \int_{\D}  |\f u_0|^2 \de x } \\&\quad + 
\expect{\psi\int_0^T \phi \int_{\D} \sigma^1 \cdot \f u \de x   \de W(t) }  
+ \expect{\psi\int_0^T \phi \|\sigma^1\|_{L_2(\mathfrak{U}, L^2_\sigma(\D))}^2  \de t  }\,,
\end{align*}
for all $ \phi \in  \C^1_c([0,T);[0,\infty))$ and all $\psi \in L^\infty(\Omega;[0,\infty))$.
Lemma~\ref{lem:invar}  
now implies~\eqref{eq:defMEasinNav} first only for a.e. $t\in(0,T)$. 

Finally, the strong convergence~\eqref{eq:strongconv} and the Young measure convergence~\eqref{weak:convE} allows to infer
 for all $ \Phi \in \C_b( L^{s}(\D\times (0,T))) $
  that  
    \begin{equation}
    \begin{aligned}
\expect{\langle \mu_{\omega} , \Phi(\lambda_{\f u}) \rangle } = {}&
\lim_{N\to \infty} {\mathbb{E}}\left[\Phi({ u^N}(\tilde{\omega})) \right] 
={}
\lim_{N\to \infty}\tilde{\mathbb{E}}\left[\Phi(\tilde{u}^N(\tilde{\omega})) \right]
={}
\tilde{\mathbb{E}}\left[\Phi(\tilde{\f u}(\tilde{\omega})) \right] 
    \,,
       \end{aligned}
    \end{equation}
   which implies~\eqref{eq:idenstrongweak}. We note that the time-regularity of the solution immediately follows from the convergences~\eqref{weak:conv} and~\eqref{weak:conv2} together with Lemma~\ref{lem:weakcont}.
\end{proof}

As a next main point, we introduce the energy-variational solution concept for SPDEs. In this framework the singular limit of vanishing viscosity can be identified without the use of Young-measures. Within the energy-variational framework, we may pass to the limit using only lower semicontinuity. 
Nevertheless, we can infer an equality with a defect measure, similar to the concept in~\cite{HZZ_2}, from the energy-variational inequality via a Hahn--Banach argument (see Theorem~\ref{thm:equi}). 

The existence of the above measure-valued solutions also implies the existence of energy-variational solution in the following sense. 
\begin{definition}[Energy-variational solution to Navier--Stokes]\label{def:EnvarNav}
Let $\f u_0 \in L^4(\Omega;L^2_{\sigma}(\D))$. 
An $(\mathcal{F}_t)$-progressively measurable stochastic process 
\[
(\f{u},E) \in L^4_{w^*}\left( \Omega;L^\infty(0,T; L^2_\sigma(D))\cap L^2(0,T;H^1_{\sigma}(\D)) \right) \times  L^{2}_{w^*}(\Omega;L^\infty(0,T))
\]
such that $ \f u \in \C_w([0,T];L^2_\sigma(\D))$ and $E \in \mathfrak{D}([0,T])$, $\mathbb{P}$-almost surely~is called an \emph{energy-variational solution} for the stochastic incompressible Navier--Stokes equations  if:

\begin{enumerate}[label=\roman*)] 
    \item $ E(t) \geq \frac{1}{2}\int_{\D}| \f{u}(x,t)|^2  \de x $ for  every $ t\in [0,T]$, \( \mathbb{P} \)-almost surely,\label{def:NavenVar4}
    \item \label{def:NavenVar5} For all  test processes in the sense of Definition~\ref{def:stoch}   the  energy-variational inequality
\begin{equation}\label{eq:envarNav}
            \begin{aligned}
     \left[ E  - \int_{\D} \f{u} \cdot \boldsymbol{\varphi} \, \de x \right ] \Big|_{s-}^{t}&+  \int_s^t 
    \int_{\D} \nu ( \nabla \f u)_{\sym} : (\nabla \f u - \nabla \f \varphi) + \left[ \f{u} \otimes \f{u} \right] : \nabla \boldsymbol{\varphi} \, \de x \, \de \tau  \\ &+\int_s^t 
    2 \|(\nabla \f \varphi)_{\sym,-}\|_{L^\infty(\D;\R^{d\times d})}\left[ \frac{1}{2}\int_{\D} | \f{u}|^2 - E  \right]  \de \tau \\ & +\int_s^t  \int_{\D} A \cdot \f{u} + \frac{1}{2}[\mathcal{P}(\sigma^2 \cdot \nabla )]^2 \f \varphi \cdot \f u   \de x +\nu \alpha  \int_{\partial \D} \f u \cdot ( \f u - \f \varphi) \de S  \de 
    \tau\\
       \leq{}& \int_s^t   \int_{\D}  (\mathbf u -  \f\varphi )\cdot \sigma^1 + \f u \cdot \mathcal{P}(\sigma^2 \cdot \nabla ) \f \varphi   - \f{u}\cdot B \de x\, \de W({\tau})\\ &  + \int_s^t \frac{1}{2}
 \|\sigma^1\|_{L_2(\mathfrak{U}, L^2_\sigma(\D))}^2  -
 \Tr{\int_{\D} \sigma ^1 \cdot B - \f u \cdot \mathcal{P}(\sigma^2 \cdot \nabla ) B \de x
} 
 \de 
\tau     \end{aligned}
    \end{equation}
         holds for all $s<t \in [0,T]$ and \( \mathbb{P} \)-almost surely with $ \f u(0)= \f u_0$.
\end{enumerate}
\end{definition}

\begin{theorem}[Existence of energy-variational solutions to Navier--Stokes]\label{thm:exNav}
Let Assumption~\ref{Ass:1} be fulfilled as well as $\f {u}_0\in L^p ({\Omega}; L_{\sigma }^2\left(\D\right))$ 
     for some $p>4$.
    Then there exists an energy-variational solution in the sense of Definition~\ref{def:EnvarNav} such that $ E(0) =\frac{1}{2}\| \f u_0\|_{L^2(\mathcal{D})}^2 $ almost surely and 
    $$ \f u \in 
L^{p}_{w^*}\left(\Omega; L^\infty(0,T; L^2_{\sigma}(\D))\cap L^2(0,T;H^1_{\sigma}(\D))\right)
        \cap L^{p/2}(\Omega;W^{\gamma,p/2}(0,T;(H^2(\D)\cap H^1_{\sigma}(\D))^*)) $$ 
        for any  
       $\gamma <\frac{1}{2}$ as well as 
       $$ E \in L^{p/2}_{w^*}(\Omega; L^\infty(0,T))\,.$$
\end{theorem}
\begin{proof}[Proof of Theorem~\ref{thm:exNav}]
We define $ E : = \frac{1}{2}\int_{\D} \langle \mu , |\lambda_{\f u}|^2 \rangle \de x$. 
From this definition and inequality~\eqref{eq:defMEasinNav}, we first find that the function $ 
t\mapsto E (t) - \int_0^t \int_{\D}\f u \cdot \sigma^1 \de x \,\de W - \frac{1}{2}\int_0^t \|\sigma^1\|_{L_2(\mathfrak{U}, L^2_\sigma(\D))}^2 \de t $ is non-increasing. 
This implies that this function is bounded in $L^{p/2}_{w^*}(\Omega;\mathrm{BV}([0,T]))$. The integral terms are even bounded in $L^{p/2}(\Omega;\C([0,T]))$ 
due to the BDG inequality for the $\f u$-dependent term  as in the derivation of the energy estimates above.
This implies that we may find a representative of $E\in \mathfrak{D}([0,T]) $ $\P$-almost surely with $E\in L^{p/2}_{w^*}(\Omega;L^\infty(0,T)) $. Definition~\ref{def:EnvarNav},~\ref{def:NavenVar4} is a direct consequence of Jensen's inequality. 
For the second term in~\eqref{eq:defMeasweakNav}, we observe that
\begin{align*}
\int_{\D} \langle\mu, \lambda_{\f u}\otimes \lambda_{\f u} :\nabla \varphi\rangle  \de x ={}& \int_{\D} {\f u}\otimes {\f u} :\nabla \varphi \de x + \int_{\D} \langle\mu, (\lambda_{\f u}- \f u )\otimes (\lambda_{\f u}- \f u)  :\nabla \varphi \rangle \de x \\
\geq{}& \int_{\D} {\f u}\otimes {\f u} :\nabla \varphi \de x - \int_{\D} \langle\mu, |\lambda_{\f u}- \f u |^2 \rangle \de x \| (\nabla \varphi )_{\sym,-} \|_{L^\infty(\D;\R^{d\times d})} 
\\
={}& \int_{\D} {\f u}\otimes {\f u} :\nabla \varphi \de x - \left[ \int_{\D} \langle\mu, |\lambda_{\f u} |^2 \rangle - | \f u|^2 \de x \right] \| (\nabla \varphi )_{\sym,-} \|_{L^\infty(\D;\R^{d\times d})} \\
={}& \int_{\D} {\f u}\otimes {\f u} :\nabla \varphi \de x + \left[ \frac{1}{2}\int_{\D} |\f u|^2 \de x - E  \right] 2  \| (\nabla \varphi )_{\sym,-} \|_{L^\infty(\D;\R^{d\times d})}\,.
\end{align*}
Subtracting~\eqref{eq:defMeasweakNav} from~\eqref{eq:defMEasinNav} and inserting the previous inequality immediately implies~\eqref{eq:envarNav}. 
First the two inequalities only hold a.e. in time, but from the additional time-regularity of $E$, we may infer that they hold in the sense given in point~\ref{def:NavenVar5} of Definition~\ref{def:EnvarNav}. 
\end{proof}

\section{Energy-variational solutions for the stochastic incompressible Euler equations\label{sec:4}}

In this section, we introduce the concept of energy-variational solutions for the stochastic incompressible Euler equations in order to identify the limit of vanishing viscosity for solutions to the Navier--Stokes equations. This gives the first existence result of probabilistically strong, measure-valued solutions for the three dimensional incompressible Euler equations obtained from vanishing viscosity of Navier--Stokes.

\begin{definition}[Energy-variational for the stochastic Euler equations]\label{def:Envar}
Let $\f u_0 \in L^4(\Omega,\mathcal{F}_0;L^2_\sigma(\D))$. An $(\mathcal{F}_t)$-progressively measurable stochastic process 
\[
(\f{u},E) \in L^4_{w^*}\left( \Omega;L^\infty (0,T; L^{2}_\sigma(D))
\right) \times  L^{2}(\Omega;L^\infty(0,T))
\]
such that $\f u \in \C_w([0,T];L^2_\sigma(\D))$ and $E\in\mathfrak{D}([0,T]) $ $\mathbb P$-almost surely
is called an \emph{energy-variational solution} to the incompressible Euler equations with noise, \textit{i.e.,}~\eqref{eq:stochastic_euler_additive} with $\nu=0$, if:

\begin{enumerate}[label=\roman*)] 
    \item $ E(t) \geq \frac{1}{2}\int_{\D}| \f{u}(x,t)|^2  \de x $ for  every $ t\in [0,T]$, \( \mathbb{P} \)-almost surely,
    \item For all  test processes  \( \boldsymbol{\varphi} \) in the sense of Definition~\ref{def:stoch},
     the  energy-variational inequality 
\begin{equation}\label{eq:envar}
            \begin{aligned}
     \left[ E  - \int_{\D} \f{u} \cdot \f{\varphi} \, \de x \right ] \Big|_{s-}^{t}&+  \int_s^t 
    \int_{\D} \left[ \f{u} \otimes \f{u} \right] : \nabla \f{\varphi}  +\f u \cdot A   + \frac{1}{2}[\mathcal{P}(\sigma^2\cdot\nabla )]^2 \f \varphi \cdot \f u  \, \de x \, \de \tau  \\ &+\int_s^t  2  \| (\nabla \f{\varphi})_{\mathrm{sym},-} \|_{L^\infty(\D;\mathbb R^{d\times d})} \left[ \frac{1}{2}\int_{\D} | \f{u}|^2 - E  \right] \, \de 
    \tau\\
       \leq{}& \int_s^t   \int_{\D}   (\f u - \f \varphi )\cdot \sigma^1 + (\sigma^2\cdot \nabla)\f \varphi \cdot \f u -  B\cdot  \f u \de x\, \de W({\tau}) \\&+ \int_s^t \frac{1}{2}
 \|\sigma^1\|_{L_2(\mathfrak{U}, L^2_\sigma(\D))}^2  
\de 
\tau \\&- \int_s^t \Tr{ \int_{\D} \sigma^1 \cdot B  -  \f u \cdot (\sigma^2\cdot \nabla) B  \de x}  \de \tau    \end{aligned}
    \end{equation}
       holds  for all $s<t \in [0,T]$ \( \mathbb{P} \)-almost surely with $ \f u(0)= \f u_0$. 
\end{enumerate}
\end{definition}
\begin{remark}[Regularity in time]\label{rem:time}
    We note that the point-wise in time inequality~\eqref{eq:envar} can also be written as an integrated inequality in time, via Lemma~\ref{lem:invar}. 
    The auxiliary variable $E$ is a c\`{a}dl\`{a}g function, $E
    \in 
    \mathfrak{D}([0,T])$, due to the energy inequality. Indeed, from~\eqref{eq:envar} with $\f \varphi =0$, we find  that the function
    \begin{align*}
        E\Big|_{s-}^{t} &- \int_s^t\int_{\D}\f u \cdot \sigma^1 \de x \de W({\tau}) - \frac{1}{2}\int_s^t \| \sigma^1\|_{L_2(\mathfrak{U},L^2_\sigma)}^2 \de \tau 
        \\
        &= \left[ E - \int_0^{\cdot}\int_{\D}\f u \cdot \sigma^1 \de x \de W(\tau) - \frac{1}{2}\int_0^\cdot \| \sigma^1\|_{L_2(\mathfrak{U},L^2_\sigma)}^2 \de \tau \right]\Bigg|_{s-}^t
    \end{align*}
    is monotone, and, thus, a $\mathrm{BV}$- function, which has a c\`{a}dl\`{a}g representation. As the integral  $$ t \mapsto \int_0^t\int_{\D}\f u \cdot \sigma^1 \de x \de W({s}) + \frac{1}{2}\int_0^t \| \sigma^1\|_{L_2(\mathfrak{U},L^2_\sigma)}^2 \de s $$ is continuous in time, also $ E$ is a c\`{a}dl\`{a}g function.
\end{remark}

\begin{theorem}[Energy-variational solutions to Euler]\label{thm:ex}
Let Assumption~\ref{Ass:1} be fulfilled as well as $\f {u}_0\in L^p ({\Omega}; L_{\sigma }^2\left(\D\right))$
    for some $p>4$.
    Then there exists an energy-variational solution in the sense of Definition~\ref{def:Envar} such that
    $ E(0) =\frac{1}{2}\| \f u_0\|_{L^2(\mathcal{D})}^2 $, $\mathbb P$-almost surely, and, in addition, 
\begin{equation}\label{reg:envarEul}
\begin{aligned}
    \f u  &\in  L^p_{w^*}(\Omega; L^\infty(0,T; L^2_{\sigma}(\D))) \cap L^{p/2}(\Omega; W^{\gamma,p/2}(0,T;(W^{2,r}(\D)\cap H^1_{\sigma}(\D))^*)\,, 
\\
E & \in 
   L^{p/2}_{w^*}(\Omega;L^\infty(0,T))\,.
\end{aligned}    
\end{equation}
    Moreover, for any sequence $ \{ \f u_{\nu}\} $ such that $ \f u_{\nu}$ is an energy-variational solution for the stochastic Navier--Stokes equations~\eqref{eq:stochastic_euler_additive} with $\nu>0$ in the sense of Definition~\ref{def:EnvarNav} such that 
    \begin{align*}
     \f u_\nu  &\stackrel{*}{\rightharpoonup} \f u \quad && \text{in } L^p_{w^*}(\Omega; L^\infty(0,T; L^2_{\sigma}(\D)))\,,\\
   \f u_\nu &\stackrel{*}{\rightharpoonup} \f u \quad&& \text{in }L^{p/2}(\Omega; W^{\gamma,p/2}(0,T;(W^{2,r}(\D)\cap H^1_{\sigma}(\D))^*)) \text{ for any }r>d,
     \\ E^\nu &\stackrel{*}{\rightharpoonup} E \quad && \text{in } L^{p/2}_{w^*}(\Omega; L^\infty(0,T))\,. 
\end{align*}
the limit $(\f u, E)$ will be an energy-variational solution in the sense of Definition~\ref{def:Envar}
enjoying the additional regularity 
\eqref{reg:envarEul}. 
\end{theorem}

\begin{proof}[Proof of Theorem~\ref{thm:ex}]
    We structure the  proof of existence of energy-variational solutions to the incompressible Euler equations in different steps. As already mentioned, we consider the Navier--Stokes equations~\eqref{eq:stochastic_euler_additive} with $\nu>0$ as an approximate system and the existence of energy-variational solutions to the Navier--Stokes equations in the sense of Definition~\ref{def:EnvarNav} was already proven in Theorem~\ref{thm:exNav}.

\textit{Step 1: a priori estimates.}
Choosing $s=0$ and $\varphi\equiv 0$ in~\eqref{eq:envarNav}, we infer 
\begin{equation} \label{eq:ito_energy_eul}
\begin{aligned}
E^\nu(t)  &+\nu\int_0^t \|(\nabla \f u_\nu(s))_{\sym}\|_{L^2(\D)}^2 + \alpha \| \f u_{\nu} \|_{L^2(\partial \D)}^2 \de s 
- \frac{1}{2}\|\f u_0 \|_{L^2(\D)}^2 \\
&\leq   \int_0^t \int_{\D} \sigma^1  \,  \f u_\nu(s) \de x   \,\de W(s)  
 + \frac{1}{2}\int_0^t \|\sigma^1\|_{L_2(\mathfrak{U}, L^2_\sigma(\D))}^2 \, \de s,
\end{aligned}
\end{equation}
for all $t\in(0,T]$ and $\mathbb P$ almost surely. 

For \( p > 4 \) we take $p/2$-th power and expectations in order to infer
\begin{equation*}
\begin{aligned}
\mathbb{E} &\left[ \sup_{t \in [0,T]} E_{\nu} ^{p/2}
\right] + \expect{\nu \|(\nabla \f u_\nu)_{\sym} \|_{L^2(\D\times (0,T))}^p + \alpha \| \f u _\nu \|_{L^2(\partial\D\times (0,T))}^p }
\\&\leq C \left( \expect{\|\f u_0\|_{L^2(\D)}^p }
+ \mathbb{E} \left[ \sup_{t \in [0,T]}  \left( \int_0^t \int_{\D} \sigma^1 \cdot  \f u_\nu(s) \de x \de W(s) \right)^{p/2} \right] \right. \\
&\qquad \left. + \mathbb{E} \left[ \left( \int_0^T \|\sigma^1\|_{L_2(\mathfrak{U}, L^2(\D))}^2 \, \de s \right)^{p/2} \right] \right)
\end{aligned}
\end{equation*}
 and by the BDG inequality:
 \begin{align*}
\mathbb{E} \left[ \sup_{t \in [0,T]} \left( \int_0^t \int_{\D} \sigma^1  \f u_\nu(s)\de x \de W(s)  \right)^{p/2} \right]
\leq{}& C_p 
\, \mathbb{E} \left[ \left( \int_0^T  \Tr{\left( \int_{\D} \sigma^1 \cdot \f u_\nu \de x \right)^ 2} \, \de s \right)^{p/4} \right]
\\
\leq{}&C_p \, \|\sigma^1\|_{L_2(\mathfrak{U}, L^2(\D))}^{p/2} 
\, \mathbb{E} \left[ \left( \int_0^T \|\f u_\nu (s)\|_{L^2(\D)}^2 \, \de s \right)^{p/4} \right]
.
\end{align*}
Using Young's inequality,~\eqref{eq:coercivityNavier}, Gronwall-type arguments, and~\ref{def:NavenVar4} of Definition~\ref{def:EnvarNav}, one then obtains
\begin{multline}\label{aprioriEuler}
\mathbb{E} \left[ \sup_{t \in [0,T]} \|\f u_\nu(t)\|_{L^2(\D)}^p \right] + \expect{\sup_{t \in [0,T]} E_{\nu}(t)^{p/2}}+ \nu^{p/2}\expect{ \|\nabla \f u_\nu \|_{L^2(\D\times (0,T))}^p }
\\\leq C(p,T, \|u_0\|_{L^2(\D)}, \|\sigma^1\|_{L_2(\mathfrak{U}, L^2(\D))}),
\end{multline}
with a constant \( C \) independent of \( \nu \).

Choosing $\f \varphi = \beta \psi$ for $\psi  \in C^2(\D)\cap H^1_{\sigma}(\D)$ in~\eqref{eq:envarNav} and multiplying the resulting limit by $\frac{1}{\beta}$ implies after taking the limit $\beta \nearrow \infty $ that only the linear terms in $\f \varphi$ remain in~\eqref{eq:envarNav}. Note that  $A\equiv 0 \equiv B$, leads to 
\begin{equation*}
            \begin{aligned}
      -\int_{\D} (\f{u}_\nu(t)-\f u_\nu (s) )\cdot \psi  \, \de x  \leq{}& -    \int_s^t 
    \int_{\D} \nu (\nabla \f u_\nu )_{\sym} :  \nabla \psi + \left[ \f{u}_\nu \otimes \f{u}_\nu \right] : \nabla \psi \, \de x + \nu \alpha \int_{\partial \D} \f u_\nu \cdot \psi \de S \de \tau  \\ &-\int_s^t 
    2 \|(\nabla \psi)_{\sym,-}\|_{L^\infty(\D;\R^{d\times d})}\left[ \frac{1}{2}\int_{\D} | \f{u}_\nu|^2\dx  - E  \right]  \de \tau \\ &+ \int_s^t  \int_{\D}  \frac{1}{2}[\mathcal{P}(\sigma^2 \cdot \nabla )]^2 \psi \cdot \f u _\nu  \de x \, \de 
    \tau\\
       &- \int_s^t   \int_{\D}     \psi \cdot \sigma^1 - \f u _\nu\cdot \mathcal{P}(\sigma^2 \cdot \nabla ) \psi  \de x\, \de W({\tau}).  
\end{aligned}
    \end{equation*}
Estimating the right-hand side using Fubini~\cite[Thm.~4.18]{DaPratoZabczyk}, we find  
\begin{align*}
\begin{split}
        \int_{\D}& (\f u_\nu (t)-\f u_\nu (s)) \cdot \psi  \de x\\
\leq{}& \int_s^t \| \f u_\nu \|_{L^2(\D)}^2 \|\nabla \psi  \|_{L^\infty(\D)} + 
    2 \|(\nabla \psi)_{\sym,-}\|_{L^\infty(\D;\R^{d\times d})}\left[ E_\nu -\frac{1}{2}\int_{\D} | \f{u}_\nu|^2 \dx \right]  \de \tau \\&+\int_s^t c \| \sigma^2\|_{L_2(\mathfrak{U}; H^2(\D))}^2   \| \f u_\nu \|_{L^2_\sigma (\D)}\|\psi \|_{H^2(\D)} +
    C \nu \left( 1  +  \alpha  \right)  \|  \f u_\nu \|_{H^1 (\D)}\|\psi \|_{H^1(\D)} \de \tau 
\\ 
& +   \left \|\int_s^t\sigma^1\de W \right \|_{ L^2(\D)} \| \psi \|_{L^2(\D) } + 
\left \| \int_s^t\f u_\nu \otimes \sigma^2\de W \right \|_{L^2(\D)} \| \nabla \psi  \|_{L^2(\D) } \,.
\end{split}
\end{align*}
This allows to estimate the dual norm of $W^{2,r}(\D)\cap H^1_{\sigma}(\D) $ for $r>d$ as $W^{2,r}(\D) $   is embedded into $W^{1,\infty}(\D)$, leading to 
\begin{align}\label{eq:timeestEul}
\begin{split}
     \| \f u_\nu (t)&-\f u_\nu (s)\|_{(W^{2,r}(\D)\cap H^1_{\sigma}(\D))^*}\\
\leq{}& C 
\left(\| \f u_\nu \|_{L^\infty(0,T;L^2(\D))}^2 + \| E _\nu\|_{L^\infty(0,T)} \right) (t-s)
     \\&+C  \nu\| \f u_\nu \|_{L^2(0,T;H^1_{\sigma}(\D))}(t-s)^{1/2}   + C\| \sigma^2\|_{L_2(\mathfrak{U}; H^2(\D))}^2   \| \f u_\nu \|_{L^\infty(0,T;L^2(\D))}(t-s)
\\ 
& + C\left(  \left \|\int_s^t\sigma^1\de W \right \|_{ L^2(\D)}  +
\left \| \int_s^t\f u _\nu\otimes \sigma^2\de W \right \|_{L^2(\D)} \right)  \,.
\end{split}
\end{align}
Using the fact that the stochastic integrals are known to be bounded in a Sobolev--Slobodeckij space~\cite[Lem.~2.1]{FlandoliGatarek}, \textit{i.e.,}
$$
\begin{aligned}
    \left\| \int_0^\cdot \sigma ^1  \de W \right\|_{L^q(\Omega;W^{\gamma,q}(0,T; L^2_{\sigma}(\D)))}^q &\leq  C 
\left \| \sigma^1 \right\|_{
L_2(\mathfrak{U};L^2_\sigma(\D))}^q\,,
\\
 \left\| \int_0^\cdot
  \f u_\nu \otimes 
 \sigma ^2  \de W \right\|_{L^q(\Omega;W^{\gamma,q}(0,T; L^2(\D)))}^q 
& \leq  C \expect{\int_0^T\| \f u_\nu \otimes \sigma^2 \|_{L^2(\D)}^q \de t  }\\
 &\leq  C\| \f u_\nu\|_{L^q(\Omega; L^q(0,T;L^2_{\sigma}(\D)))}^q
 \left \| \sigma^2 \right\|^q_{
L_2(\mathfrak{U};H^2(\D))} 
\end{aligned}
$$
for $\gamma < 1/2$ and any $q\in [2,\infty)$. 
Dividing~\eqref{eq:timeestEul} by $|t-s|^{1/q+\gamma}$, taking the $q$-th power, integrating over $t $ and $s$,
and taking 
expectations let us conclude that there exists a constant $C>0$ such that
$$    \| \f u_\nu \|_{L^{p/2}(\Omega; W^{\gamma,p/2}(0,T;(W^{2,r}(\D)\cap H^1_{\sigma}(\D))^*))} \\ \leq  C (\|\f u _\nu\|_{L^{p}(\Omega;L^\infty(0,T;L^2_{\sigma}(\D)))}^2+1)  \,
$$
for any $\gamma < 1/2 $, any $r>d$. 
The right-hand side is bounded due to~\eqref{aprioriEuler}. 
We note that~\cite[Thm.~2]{FlandoliGatarek} implies the continuous embedding of $W^{\gamma,p/2}(0,T;(W^{2,r}(\D)\cap H^1_{\sigma}(\D))^*)$ into $\C([0,T];(W^{2,r}(\D)\cap H^1_{\sigma}(\D))^*))$ for $\gamma p>2 $.

\textit{Step 2: Convergence.} 
From the \textit{a priori} estimates, we may extract a converging subsequence, which we do not relabel such that 
\begin{align}
     \f u_\nu  &\stackrel{*}{\rightharpoonup} \f u \quad && \text{in } L^p_{w^*}(\Omega; L^\infty(0,T; L^2_{\sigma}(\D)))\,,\label{weak:convnu}\\
\f u_\nu &\stackrel{*}{\rightharpoonup} \f u \quad&& \text{in }L^{p/2}(\Omega; W^{\gamma,p/2}(0,T;(W^{2,r}(\D)\cap H^1_{\sigma}(\D))^*)),\label{weak:convnut}
     \\
    E^\nu &\stackrel{*}{\rightharpoonup} E \quad && \text{in } L^{p/2}_{w^*}(\Omega; L^\infty(0,T))\,. \label{weak:convEnu}
\end{align}

Observing that $ \nu |(\nabla \f u_\nu)_{\sym}|^2 \geq 0$ and $ | u _\nu |^2 \geq 0 $, we infer from inequality~\eqref{eq:envarNav}
 and Lemma~\ref{lem:invar} that 
   \begin{align*}
   &\expectp{ -\int_0^T \partial_t\phi  \left[ E_\nu  - \int_{\D} {\f u_\nu} \cdot  \f{\varphi} \, \de x \right ] \, \de t - \phi(0) \left[ E_{\nu}(0) -  \int_{\D} \f {u}_0 \cdot  \f{\varphi}(0) \, \de x\right] }\\&+ \expectp{ \int_0^T \phi 
    \int_{\D} \left[   \f{u}_\nu \otimes  \f {u}_\nu \right] : \nabla  \f{\varphi}- \nu (\nabla \f u_\nu)_{\sym} :\nabla \f\varphi + A \cdot \f u_\nu +\frac{1}{2}\f u_\nu \cdot [\mathcal{P}( \sigma^2\cdot\nabla)]^2\f \varphi \, \de x  \, \de t } \\& +\expectp{\int_0^T \phi 
     2  \| (\nabla  \f{\varphi})_{\mathrm{sym},-} \|_{L^\infty(\D;\mathbb R^{d\times d})} 
    \left[ \frac{1}{2}\int_{\D} |   {\f u_\nu}|^2 \dx - E_\nu   \right]  - \nu \alpha\int_{\partial \D}\f u_\nu \cdot \f \varphi \de S  \de t  }\\
   & \leq{} \expectp{\int_0^T \phi  \int_{\D}  (   \f u_\nu -  \f\varphi )\cdot \sigma ^1 + ( \sigma^2 \cdot \nabla )\f\varphi\cdot \f u_\nu   -  \f u_\nu  \cdot B \de x\, \de W(t)} \\&\quad + \expectp{\int_0^T \phi \frac{1}{2}
 \|\sigma^1 \|_{L_2(\mathfrak{U}, L^2_\sigma(\D))}^2  - \phi \Tr{ \int_{\D} \sigma^1 B  - \f u_\nu \cdot ( \sigma ^2 \cdot \nabla) B \de x } \de t }
    \end{align*}
for all test processes $\f\varphi $ in the sense of Definition~\ref{def:stoch}, for all $ \phi \in  \C^1_c([0,T);[0,\infty))$, and for all $\psi \in L^\infty(\Omega;[0,\infty))$.

We next verify that we may pass to the limit using the convergences~\eqref{weak:convnu} and~\eqref{weak:convEnu}. 
Note that, since $p>4$, $\nabla\f\varphi\in L^2\bigl(\Omega; L^1(0,T;L^\infty(\D))\bigr)\subset L^{p/(p-2)}\bigl(\Omega;L^1(0,T;L^\infty(\D))\bigr)$. Therefore, and since it is quadratic and nonnegative, the function 
$$
\f u_\nu  \mapsto \expectp{ \int_0^T \phi\left[ \int _{\D} \f u_\nu \otimes \f u_\nu : \nabla \f\varphi \de x + 2  \| (\nabla  \f{\varphi})_{\mathrm{sym},-} \|_{L^\infty(\D;\mathbb R^{d\times d})}\frac{1}{2}\| \f u_\nu \|_{L^2(\D)}^2\right] \de t}
$$
is a convex and lower semicontinuous mapping on $ L^p_{w^*}(\Omega; L^\infty(0,T; L^2_{\sigma}(\D))) $. Therefore, we may pass to the limit using~\eqref{weak:convnu} 
in order to infer
\begin{align*}
    \liminf_{\nu\to 0 } \expectp{ \int_0^T \phi\left[ \int _{\D} \f u_\nu \otimes \f u_\nu : \nabla \f\varphi \de x  +2  \| (\nabla  \f{\varphi})_{\mathrm{sym},-} \|_{L^\infty(D;\mathbb R^{d\times d})} \frac{1}{2}\| \f u_\nu \|_{L^2(\D)}^2\right] \de t} \\
    \geq \expectp{ \int_0^T \phi\left[ \int _{\D} \f u \otimes \f u : \nabla \f\varphi  \de x +2  \| (\nabla  \f{\varphi})_{\mathrm{sym},-} \|_{L^\infty(D;\mathbb R^{d\times d})} \frac{1}{2}\|\f  u \|_{L^2(\D)}^2\right] \de t}\,. 
\end{align*}
Moreover, for the last remaining  term incorporating $\nu$, we find
\begin{multline*}
    \nu \expectp{\Big | \int_0^T \phi (t) 
    \int_{\D}   (\nabla \f u_\nu )_{\sym}:\nabla \f\varphi\, \de x +\int_{\partial \D} \f u_\nu \cdot \f \varphi \de S  \de t \Big|} \\ \leq C\sqrt{\nu} \expect{\sqrt{\left( \nu \|  \f u_\nu\|_{L^2(0,T; H^1_\sigma(\D))}^2 \right)} \| \f \varphi\|_{L^2(0,T; H^1(\D))} }\to 0 
\end{multline*}
as $\nu \to 0$.

All other terms appear linearly such that the convergences~\eqref{weak:convnu}--\eqref{weak:convEnu} suffice to pass to the limit so that
\begin{align*}
    -&\expectp{\int_0^T \partial_t \phi\left[ E - \int_{\D}\f  u \cdot \f \varphi \de x \right]\de t } + \expectp{ \int_0^T \phi \int _{\D} \f u \otimes \f u : \nabla \f \varphi  \de x \de t} \\
    +&\expectp{\int_0^T\phi  2  \| (\nabla  \f{\varphi})_{\mathrm{sym},-} \|_{L^\infty(\D;\mathbb R^{d\times d})} \left[\frac{1}{2}\| \f u \|_{L^2(\D)}^2- E\right] \de t} \\
    +&\expectp{ \int_0^T \phi \int_{\D} A \cdot \f u +\frac{1}{2}[\mathcal{P}(\sigma^2 \cdot \nabla) ]^2 \f\varphi \cdot \f u \de x \de t }
    \\
    -&\expectp{  \phi(0) \left[ E (0) - \int_{\D} \f u_0  \cdot \f\varphi \de x\right]  } 
    \\
&\leq \expectp{\int_0^T \phi \int_{\D} \sigma^1 \cdot(  \f u- \f\varphi )+ (
\sigma^2 \cdot \nabla) \f\varphi \cdot \f u -B \cdot \f u  \de x   \de W (t) } \\
&\quad+ \expectp{\int_0^T \phi \|\sigma^1\|_{L_2(\mathfrak{U}, L^2_\sigma(\D))}^2 - \phi \Tr{ \int_{\D} ( \sigma^1 B - \f u \cdot (\sigma^2 \cdot \nabla) B)  \de x} \, \de t } 
\end{align*}
for all test processes $\f \varphi $ in the sense of Definition~\ref{def:stoch}, all $ \phi \in  \C^1_c([0,T);[0,\infty))$, and all $\psi \in L^\infty(\Omega;[0,\infty))$.
Lemma~\ref{lem:invar} allows to deduce~\eqref{eq:envar}. First this holds only almost everywhere in time, but from the time regularity of Remark~\ref{rem:time}, we first infer that $E\in \mathfrak{D}([0,T])$. 
Moreover from Lemma~\ref{lem:weakcont}, we infer that
 $ \f u \in \C_w([0,T]; L^2_\sigma(\D))$ $\mathbb P$-almost surely such that~\eqref{eq:envar} holds  everywhere in $[0,T]$ $\mathbb P$-almost surely. 
 \end{proof}

\begin{definition}[Dissipative weak solution]\label{def:Dissweak}
Let $\f u_0 \in L^4(\Omega;\mathcal{F}_0;L^2_\sigma(
\D))$. 
    A $(\mathcal{F}_t)$-progressively measurable stochastic process 
\[
(\f{u},E
) \in L^4_{w^*}\left( \Omega; L^\infty(0,T; L^2_\sigma(\D)) \right) \times  L^{2}_{w^*}(\Omega; L^\infty(0,T))
\]
such that $ \f u \in \C_w([0,T];L^2_\sigma(\D))$ and $E \in \mathfrak{D}([0,T])$ $\mathbb P$-almost surely 
is called a \emph{dissipative weak solution} to the incompressible Euler equations with  noise, \textit{i.e.,}~\eqref{eq:stochastic_euler_additive} with $\nu=0$, if:
 there exists an $(\mathcal{F}_t)$-progressively measurable  $ \mathfrak{R}\in  L^{2}_{w^*} (\Omega;  L^\infty_{w^*}( 0,T; \mathcal{M}(\ov{\D}; \R^{d\times d}_{\sym,+})))$ 
 such that:
\begin{enumerate}[label=\roman*)]
     \item $ E(t) \geq \frac{1}{2}\int_{\D}| \f{u}(x,t)|^2  \de x + \frac{1}{2}\int_{\ov{\D}}\de \tr{(\mathfrak{R}(t))}$ for almost   every $ t\in (0,T)$ \( \mathbb{P} \)-almost surely,\label{def:dissweak4}
    \item \label{eq:weakdisseq}For all  \( {\varphi} \in \C^2_{\sigma}(\ov\D;\R^d) 
\)  
    the following weak formulation holds  for all $t\in (0,T)$ \( \mathbb{P} \)-almost surely,
    \begin{equation}\label{eq:defdissweak}
            \begin{aligned}
       -  \int_{\D} (\f{u}(t)
- \f{u}_0)
\cdot {\varphi}
\, \de x 
 &+  
    \int_0^t\int_{\D} \left[ \nabla {\varphi}:  \left[ \f{u}
    \otimes \f{u}
    \right] 
    +  \frac{1}{2}[\mathcal P(\sigma^2 \cdot \nabla) ]^2 \varphi \cdot \f u \right] \de x \,\de s  \\&+  
       \int_0^t  \int_{\D} \left[ {\varphi}\cdot  \sigma^1    -  (\sigma^2\cdot \nabla ) \varphi:  \f u \right]
        \de x \,\de W (s)
       \\  
 &+  \int_0^t\int_{\overline{\D}}  (\nabla \varphi)_{\sym} : \de \mathfrak{R}(x)\de s = 0 \,,
    \end{aligned}
        \end{equation}
    \item the energy inequality 
    \begin{equation}\label{eq:defDissenin}
           E \Big|_{s-}^{t}  
   \leq {} \int_s^t  \int_{\D}  \f u \cdot \sigma^1 \de x  \de W({\tau}) + \frac{1}{2}\int_s^t
 \|\sigma^1\|_{L_2(\mathfrak{U}, L^2_\sigma(\D))}^2  
  \de \tau 
    \end{equation}
    holds for all  $T\geq t>s\geq 0$  \( \mathbb{P} \)-almost surely.
       \end{enumerate}
\end{definition}

\begin{theorem}[Energy-variational imply dissipative weak solutions]\label{thm:equi}
   Let $p \geq 4$ and  let $(\f u,E)$ be an energy-var\-ia\-tional solution in the sense of Definition~\ref{def:Envar} with $ \f u \in L^p_{w^*}(\Omega;L^\infty(0,T;L^2_\sigma(\D)))$ and $E\in L^{p/2}_{w^*}(\Omega;L^\infty(0,T))$. Then, there exists a Reynolds-stress $\mathfrak{R}\in L^{p/2
}_{w^*} (\Omega;  L^\infty_{w^*}( 0,T; \mathcal{M}({\ov\D}; \R^{d\times d}_{\sym,+})))$ such that $ (\f u, E)$ is a dissipative weak solution with Reynolds defect~$\mathfrak{R}$ in the sense of Definition~\ref{def:Dissweak}.     
\end{theorem}
\begin{remark}[Equivalence of energy-variational and dissipative weak solutions]
Theorem~\ref{thm:equi} shows that every energy-variational solution
in the sense of Definition~\ref{def:Envar} gives rise to a dissipative
weak solution in the sense of Definition~\ref{def:Dissweak}. The two
concepts are therefore equivalent, since the converse implication
follows more directly.

Indeed, adding \eqref{eq:defdissweak} to~\eqref{eq:defDissenin}
and subtracting \eqref{eq:defdissweak} evaluated at $s$ instead of
$t$ yields~\eqref{eq:envar} with $A=B=0$, except that the term
$
    \int_s^t
    2\|(\nabla \f\varphi)_{\sym,-}\|_{L^\infty(\D;\R^{d\times d})}
    \left[\frac12\|\f u\|_{L^2(\D)}^2-E\right]
    \,\de\tau
$
is replaced by
$
    \int_s^t\int_{\ov{\D}}
    (\nabla \f\varphi)_{\sym}:\de\mathfrak R(x)\,\de\tau.
$
By point~\ref{def:dissweak4} of Definition~\ref{def:Dissweak}
and the nonnegativity of $\mathfrak R$, we have
\[
\begin{aligned}
    \int_{\ov{\D}}
    (\nabla \f\varphi)_{\sym}:\de\mathfrak R(x)
    &\geq
    -\|(\nabla \f\varphi)_{\sym,-}\|_{L^\infty(\D;\R^{d\times d})}
    \int_{\ov{\D}}\de\tr\mathfrak R(x) \\
    &\geq
    2\|(\nabla \f\varphi)_{\sym,-}\|_{L^\infty(\D;\R^{d\times d})}
    \left[\frac12\|\f u\|_{L^2(\D)}^2-E\right].
\end{aligned}
\]
This gives~\eqref{eq:envar} for test processes that are constant in
time. The extension to arbitrary test processes follows by the same
argument as in the comment after Definition~\ref{def:weakweakNav}.
\end{remark}
As an immediate consequence of Theorem~\ref{thm:equi} and Theorem~\ref{thm:ex}, we deduce 
\begin{corollary}[Existence of dissipative weak solutions]
Let Assumption~\ref{Ass:1} be fulfilled as well as $\f {u}_0\in L^p ({\Omega}; L_{\sigma }^2\left(\D\right))$
     for some $p>4$.
    Then, there exists a dissipative weak solution in the sense of Definition~\ref{def:Dissweak} such that
    $ E(0) =\frac{1}{2}\| \f u_0\|_{L^2(\mathcal{D})}^2 $ $\mathbb P$-almost surely and the solution enjoys the additional regularity 
    \[
(\f{u},E
) \in L^p_{w^*}\left( \Omega; L^\infty(0,T; L^2_\sigma(\D)) \right) \times  L^{p/2}_{w^*}(\Omega; L^\infty(0,T))
\]
as well as 
\[
\mathfrak{R}\in L^{p/2
}_{w^*} (\Omega;  L^\infty_{w^*}( 0,T; \mathcal{M}(\ov{\D}; \R^{d\times d}_{\sym,+})))
\]
\end{corollary}

 \begin{proof}[Proof of Theorem~\ref{thm:equi}]
    Let $(\f u, E)$ be an energy-variational solution in the sense of Definition~\ref{def:Envar}.  
    Choosing~$\f \varphi = 0$ implies the energy inequality~\eqref{eq:defDissenin}. 
We choose $\f \varphi = \beta \f \psi $ in~\eqref{eq:envar} with $ \f \psi $ being a test process according to Definition~\ref{def:stoch} with $ (\psi_0, A_{\psi}, B_{\psi})$
 and fulfilling $\f \psi(t)= \psi_0 +\int_0^t A_{\psi}\de s + \int_0^t  B_{\psi}\de W(s)$ for all $t\in [0,T]$. 
Multiplying the resulting inequality by $1/\beta$, sending $\beta \nearrow \infty $ implies 
\begin{equation}\label{eq:envar2}
            \begin{aligned}
-\int_{\D} \f{u}(T)&
\cdot \boldsymbol{\psi}(T)
\de x 
+\int_{\D} \f{u}_0
\cdot {\psi}(0)
\, \dx\\&+  \int_0^T 
    \int_{\D} \left[ \f{u}
    \otimes \f{u}
    \right] : \nabla \boldsymbol{\psi}
    + A_\psi\cdot \f u + \frac{1}{2}[\mathcal P(\sigma^2 \cdot \nabla) ]^2 \f \psi\cdot \f u \, \dx \, \de s  \\&+ \int_0^T 
    \left[   
        \int_{\D} \sigma^1   \cdot \boldsymbol{\psi} - \f u \cdot ( \sigma^2 \cdot \nabla) \f \psi + B_\psi \cdot \f u 
        \de x \,\right] \de W (s)
        \\&+ 
        \int_0^T \Tr{
        \int_{\D}\sigma^1 \cdot B_\psi - \f u \cdot (\sigma^2 \cdot \nabla) B_\psi \de x }
        \de s  \\  \leq{}&\int_0^T 
        2  \| (\nabla \boldsymbol{\psi})_{\mathrm{sym},-} \|_{L^\infty(\D;\mathbb R^{d\times d})} \left[E-  \frac{1}{2}\int_{\D} | \f{u}|^2  \dx  \right] \, \de s \,.
    \end{aligned}
    \end{equation}  

    We define the linear space $ \mathcal{V}$ of $(\mathcal{F}_t)$-progressively measurable  functions $(\psi_0,A_\psi,B_\psi)$
via
\begin{align*}
    \f\psi_0 & \in L^{2}(\Omega,\mathcal{F}_0; L^2_\sigma (\D))\,, \\ A_\psi&\in   L^{2}(\Omega ; L^2(0,T; L^2_\sigma(D)))\,,\\
    B_\psi &\in  L^2 (\Omega; L^2(0,T; L_2 (\mathfrak{U}; 
      H^2(\D)\cap 
     H^1_{\sigma}(\D)))\,,
\end{align*}
such that $ \f \psi (t)=   \psi_0+\int_0^t A_{\f\psi} \de s + \int_0^t B_{\f\psi} \de W(s)$ 
    satisfies~\eqref{eq:test-regularity}. 
     
     On this linear space, we define the linear form $\f l : \mathcal{V}\to \R$ via 
     \begin{align*}
    \langle \f l, (\f \psi_0, A_{\f\psi}, B_{\f\psi})  \rangle _{\mathcal{V}}:= {}&\mathbb{E}\Bigg[%
-\int_{\D} \f{u}(T)
\cdot \boldsymbol{\psi}(T)
\de x 
+\int_{\D} \f{u}_0
\cdot {\psi}(0)
\, \dx\\&+  \int_0^T 
    \int_{\D} \left[ \f{u}
    \otimes \f{u}
    \right] : \nabla \boldsymbol{\psi}
    + A_{\f\psi} \cdot \f u + \frac{1}{2}[\mathcal P(\sigma^2 \cdot \nabla) ]^2 \f \psi\cdot \f u \, \dx \, \de s  \\&+ \int_0^T 
    \left[   
        \int_{\D} \sigma^1   \cdot \boldsymbol{\psi} - \f u \cdot ( \sigma^2 \cdot \nabla) \f \psi + B_{\f\psi} \cdot \f u 
        \de x \,\right] \de W (s)
        \\&+ 
        \int_0^T \Tr{
        \int_{\D}\sigma^1 \cdot B_{\f \psi} - \f u \cdot (\sigma^2 \cdot \nabla) B_{\f\psi} \de x }
        \de s 
        \Bigg]
\end{align*}
where every occurrence of $\f \psi $ has to be interpreted as a mapping of $(\f\psi_0,A_{\f \psi},B_{\f \psi}) $ via $ \f \psi(t) = \f\psi_0+ \int_0^t A _\psi\de s + \int_0^t B _\psi \de W (s)$. 
This linear form is exactly the expectation of the left-hand side of ~\eqref{eq:envar2}. We kept the stochastic integral part in the third line to make this obvious, even though it vanishes in expectation.

Furthermore, 
we define the space $\mathcal{W}$ of $(\mathfrak{F}_t)$-progressively measurable functions $\Psi$ such that $\Psi \in L^{p/(p-2)}(\Omega; L^1  (0,T; \mathcal C(\ov{\D};\R^{d\times d}_{\sym}))) $.  Since $p\geq4$, one has $\frac{p}{p-2}\leq2.$ Consequently, $\mathcal I(\mathcal V)\subset\mathcal W$ by~\eqref{eq:test-regularity}. On the space $  \mathcal{V} $, we consider  the linear map $\mathcal{I}:\mathcal{V}\to \mathcal{W}$  via $$ \mathcal{I}(\f \psi_0, A_{\f\psi}, B_{\f\psi}):= ( \nabla \f \psi) _{\sym} = \left( \nabla \left(\f\psi(0)+ \int_0^t A_{\f\psi} \de s + \int_0^t B_{\f\psi} \de W (s)\right)\right) _{\sym}$$ and the sublinear mapping $\mathfrak{p}: \mathcal{W}
\to \R$
via
\begin{equation}\label{p}
    \mathfrak{p}(\Psi):= \mathbb{E}\left [ \int_0^T 2 \| ( \Psi )_-\|_{L^\infty({\D};\R^{d\times d})} \left[ E - \frac{1}{2}\| \f u\|_{L^2(\D)}^2 \right]\de t \right]\,.
\end{equation}
Note that this mapping is well defined due to the regularity of $E$ and $\f u$ using the Hölder inequality. 
Via an extension lemma for linear maps~\cite[Lemma~2.3]{NewPreprint} relying on the Hahn--Banach theorem, there exists a linear mapping $ L : \mathcal{W} \to \R $ such that 
\begin{equation}\label{linearformest}
\langle L,\mathcal{I}(\f\psi_0,A_{\f\psi},B_{\f\psi})\rangle=\langle \f l,(\f \psi_0,A_{\f\psi},B_{\f\psi})\rangle \quad \text{and}\quad \langle L, \Psi \rangle \leq \mathfrak p(\Psi)
\end{equation}
for all $(\f\psi_0,A_{\f\psi},B_{\f\psi}) \in \mathcal{V}$ and for all $\Psi\in \mathcal{W}$.
Using the Riesz representation theorem
we may identify this extension with an object $-\mathfrak R \in L^{p/2
}_{w^*}(\Omega; L^\infty_{w^*}(0,T;\mathcal{M}(\overline {\D};\R^{d\times d}_{\sym})))$ 
such that 
\begin{equation}
\begin{aligned}
&\expect{
-\int_{\D} \f{u}(T)
\cdot \boldsymbol{\psi}(T)
\de x 
+
\int_{\D} \f{u}_0
\cdot \boldsymbol{\psi}_0
\, \de x}\\&+  \expect{\int_0^T 
    \int_{\D} \left[ \f{u}
    \otimes \f{u}
    \right] : \nabla \boldsymbol{\psi}
    + A _{\f\psi} \cdot \f u + \frac{1}{2}[ \mathcal P(\sigma^2 \cdot \nabla) ] ^2\f \psi\cdot \f u \, \de x  \, \de s } 
        \\&+ \expect{
        \int_0^T \Tr{
        \int_{\D}\sigma^1 \cdot B_{\f\psi}  - \f u \cdot (\sigma^2 \cdot \nabla) B_{\f\psi} \de x }
        \de s 
        }\\
 &+ \mathbb{E}\left [ \int_0^T \int_{\overline{\D}} (\nabla \f\psi)_{\sym} : \de \mathfrak{R}(x) \de s  \right ]  = 0 ,
       \end{aligned}
 \end{equation}
 for all $(\f\psi_0,A_{\f\psi},B_{\f\psi}) \in \mathcal{V}$. 
 This version of the Riesz representation theorem follows from Lemma~\ref{lem:dual}.
 From this, we observe by Lemma~\ref{lem:weakdiss} that Definition \ref{def:Dissweak},~\ref{eq:weakdisseq} is fulfilled. 
 The lemma is applicable as the velocity field $\f u$ and the Reynolds stress $\mathfrak{R}$ have the required regularity. 
 
The second point of~\eqref{linearformest} implies that
\begin{align*}    
    -\mathbb{E}\left [ \int_0^T \int_{\overline{\D}} (\Psi)_{\sym} : \de \mathfrak{R}\de t  \right ]  \leq \mathfrak{p}(\Psi) = \mathbb{E}\left[ \int_0^T  2 \| ( \Psi  )_{\sym,-}\|_{\C(D;\R^{d\times d})} \left[ E - \frac{1}{2}\| \f u\|_{L^2(\D)}^2 \right]\de t \right]
\end{align*}
    for all $\Psi \in \mathcal{W}$.  Hence, for $\eta$ a nonnegative progressively measurable process and $ \zeta \in \C (\ov\D;\R^{d\times d}_{\sym})$, with $ \Psi(\omega, t, x) := \eta(\omega,t) \zeta (x)$ we get
    \begin{align*}    
    -\mathbb{E}\left [ \int_0^T \eta \int_{\overline{\D}} (\zeta)_{\sym} : \de \mathfrak{R}\de t  \right ]  \leq \mathfrak{p}(\Psi) = \mathbb{E}\left[ \int_0^T \eta  2 \| ( \zeta   )_{\sym,-}\|_{L^\infty(D;\R^{d\times d})} \left[ E - \frac{1}{2}\| \f u\|_{L^2(\D)}^2 \right]\de t \right].
\end{align*}
    By the fundamental lemma for  progressively measurable processes \textit{cf.~Lemma~\ref{lem:variationprog}} this implies that
\begin{align*}    
    - \int_{\overline{\D}} (\zeta)_{\sym} : \de \mathfrak{R}  \leq   2 \| ( \zeta  )_{\sym,-}\|_{\C(D;\R^{d\times d})} \left[ E - \frac{1}{2}\| \f u\|_{L^2(\D)}^2 \right]
\end{align*}
    for all $ \zeta \in \C (\ov\D;\R^{d\times d}_{\sym})$ a.e.~in $(0,T)$ and $\mathbb P$-almost surely. 
For all $ \zeta \in C (\ov\D;\R^{d\times d}_{\sym,+})$, we observe that $\int_{\overline{\D}} (\zeta)_{\sym} : \de \mathfrak{R}\geq 0$, which implies  
$ \mathfrak{R}\in  L^{p/2
}_{w^*} (\Omega;  L^\infty_{w^*}( 0,T; \mathcal{M}(\ov{\D}; \R^{d\times d}_{\sym,+})))$. 
    Choosing $\zeta ( x ) \equiv  - \frac{1}{2} I $,
    we observe  that
    \begin{equation*}
        \frac{1}{2} \int_{\overline{\D}}\de \tr{(\mathfrak{R})}  \leq  \left[ E - \frac{1}{2}\| \f u\|_{L^2(\D)}^2\right]
    \end{equation*}
    a.e.~in $(0,T)$ and $\mathbb P$-almost surely. Note that we take the spectral norm $|\cdot|_2$ of the matrix $(\zeta)_{\sym} $ in the definition of $\mathfrak{p}$ in~\eqref{p} such that $|I|_2=1$. Indeed, the correct dual norm of the spectral norm with respect to the Frobenius product is exactly the trace norm, giving rise to the term on the left-hand side of the previous inequality. This implies point~\textit{i)} of Definition~\ref{def:Dissweak}.
\end{proof}

\begin{lemma}\label{lem:weakdiss}
 Let $\f u$ be an adapted process with values in $L^{4}_{w^*}(\Omega;L^\infty(0,T;L_{\sigma}^{2}(\mathcal{D})))$, and $\mathfrak{R}\in L_{w^*}^{2}(\Omega;L_{w^{*}}^{\infty}(0,T;\mathcal{M}(D;\mathbb{R}_{sym,+}^{d\times d})))$ progressively measurable, such that, $ \f u \in C_w([0,T];L^2_{\sigma}(\D))$ $\P$-almost surely and for all test-processes $\f \psi$ in the sense of Definition~\ref{def:stoch} we have 
\begin{equation}
\begin{aligned} & \mathbb{E}\left[-\int_{\mathcal{D}}\f u(T)\cdot\f \psi(T)\dx+\int_{\mathcal{D}}\f u_{0}\cdot\f \psi(0)\dx\right]\\
 & +\mathbb{E}\left[\int_{0}^{T}\int_{\mathcal{D}}[\f u\otimes \f u]:\nabla\f \psi+A_{\f \psi}\cdot u+\frac{1}{2}[\mathcal{P}(\sigma^{2}\cdot\nabla)]^{2}\f \psi\cdot \f u~\dx \ds\right]\\
 & +\mathbb{E}\left[\int_{0}^{T}\Tr{\int_{\mathcal{D}}\sigma^{1}\cdot B_{\f \psi}-\f u\cdot(\sigma^{2}\cdot\nabla)B_{\f \psi}\dx }\ds\right]\\
 & +\mathbb{E}\left[\int_{0}^{T}\int_{\ov{\mathcal{D}}}(\nabla\f \psi)_{\sym}:\de \mathfrak{R}(x)\ds\right]=0.
\end{aligned}
\label{eq:4.10-3}
\end{equation}
Then, $\P$-a.s, for all $t\in [0,T],$ and all test functions $\varphi \in \C^2_{\sigma}(\ov\D;\R^d)$, we have 
\begin{equation}
\begin{aligned} & -\int_{\mathcal{D}}\f u(t)\cdot \varphi \dx +\int_{\mathcal{D}}\f u_{0}\cdot\varphi \dx \\
 & +\int_{0}^{t}\int_{\mathcal{D}}[\f u\otimes \f u]:\nabla\varphi+\frac{1}{2}[\mathcal{P}(\sigma^{2}\cdot\nabla)]^{2}\varphi\cdot \f u~\dx \ds\\
 &+\int_{0}^{t}\int_{\mathcal{D}}\sigma^{1}\cdot\varphi-\f u\cdot(\sigma^{2}\cdot\nabla)\varphi \dx \,\de W(s)\\
 & +\int_{0}^{t}\int_{\ov{\mathcal{D}}}(\nabla\varphi)_{\sym}:\de \mathfrak{R}(x)\ds=0\,.
\end{aligned}
\label{eq:4.10-3-1}
\end{equation}
\end{lemma}
\begin{proof}The proof is divided into several steps. 

    \textit{Step 1: Generalize the formulation to arbitrary time intervals.}

Let $0<s<t<T$ and $\f \phi $ be a test-process in the sense of Definition~\ref{def:stoch}. Set $\f \psi(\tau):=\eta^\varepsilon_{_{[s,t]}}(\tau)\f \phi (\tau)$ for $\tau\in [0,T]$, where $\{\eta^\varepsilon_{_{[s,t]}}\}_{\varepsilon\in(0,1)}\subset  C^{2}([0,T])$ is an appropriate approximation of $ \mathds{1}_{[s,t]}$ such that $ \eta^\varepsilon_{_{[s,t]}}(0)=0=\eta^\varepsilon_{_{[s,t]}}(T)$ for all $\varepsilon >0$ and $ \eta^\varepsilon_{_{[s,t]}} \to \mathds{1}_{[s,t]}$ pointwise a.e.~in $[0,T]$ as $\varepsilon\searrow 0$.  Then, by It\^o's product rule, we have 
\begin{align*}
\f \psi(r)-\f \psi(0)= & \eta^\varepsilon_{_{[s,t]}}(r)\f \phi (r)-\eta^\varepsilon_{_{[s,t]}}(0)\f \phi (0)\\
= & \int_{0}^{r}\eta^\varepsilon_{_{[s,t]}}(\tau)\de \f \phi (\tau)+\int_{0}^{r}\f \phi (\tau)(\eta^\varepsilon_{_{[s,t]}})'(\tau)\de \tau\\
= & \int_{0}^{r}\eta^\varepsilon_{_{[s,t]}}(\tau)A_{\f \phi }(\tau)+\f \phi (\tau)(\eta^\varepsilon_{_{[s,t]}})'(\tau)\,\de\tau+\int_{0}^{r}\eta^\varepsilon_{_{[s,t]}}(\tau)B_{\f \phi }(\tau)\de W(\tau).
\end{align*}
Hence, $\f \psi$ is a test-process in the sense of Definition~\ref{def:stoch} and we get from~\eqref{eq:4.10-3} that
\begin{equation}
\begin{aligned} & \mathbb{E}\left[\int_{0}^{T}(\eta^\varepsilon_{_{[s,t]}})'
\int_{\mathcal{D}}\f \phi 
\cdot \f u \de x \de \tau   \right]\\
 & +\mathbb{E}\left[\int_{0}^{T}\eta^\varepsilon_{_{[s,t]}}\int_{\mathcal{D}}[\f u\otimes \f u]:\nabla\f \phi 
 +A_{\f \phi }\cdot \f u+\frac{1}{2}[\mathcal{P}(\sigma^{2}\cdot\nabla)]^{2}\f \phi 
 \cdot \f u~\de x \de \tau   \right]\\
 & +\mathbb{E}\left[\int_{0}^{T}\eta^\varepsilon_{_{[s,t]}}
 \Tr{\int_{\mathcal{D}}\sigma^{1}\cdot B_{\f \phi }-\f u\cdot(\sigma^{2}\cdot\nabla)B_{\f \phi }\de x }\de \tau   \right]\\
 & +\mathbb{E}\left[\int_{0}^{T}\eta^\varepsilon_{_{[s,t]}}
 \int_{\ov{\mathcal{D}}}(\nabla\f \phi 
 )_{\sym}:\de \mathfrak{R}(x)\de \tau   \right]=0.
\end{aligned}
\label{eq:4.10-1-2}
\end{equation}
Now letting $\varepsilon \searrow 0$ such that $\{\eta^\varepsilon_{_{[s,t]}}\}_{\varepsilon \in (0,1)}$ converges to $\mathds 1_{[s,t]}$, we get that 
\begin{equation}
\begin{aligned} & \mathbb{E}\left[-\int_{\mathcal{D}}\f \phi (t)\cdot \f u(t)+\int_{\mathcal{D}}\f \phi (s)\cdot \f u(s)\right]\\
 & +\mathbb{E}\left[\int_{s}^{t}\int_{\mathcal{D}}[\f u\otimes \f u]:\nabla\f \phi 
 +A_{\f \phi }\cdot \f u+\frac{1}{2}[\mathcal{P}(\sigma^{2}\cdot\nabla)]^{2}\f \phi 
 \cdot \f u~\de x \de \tau  \right]\\
 & +\mathbb{E}\left[\int_{s}^{t}\Tr{\int_{\mathcal{D}}\sigma^{1}\cdot B_{\f \phi }-\f u\cdot(\sigma^{2}\cdot\nabla)B_{\f \phi }\de x }\de \tau  \right]\\
 & +\mathbb{E}\left[\int_{s}^{t}\int_{\ov{\mathcal{D}}}(\nabla\f \phi 
 )_{\sym}:\de \mathfrak{R}(x)\de \tau  \right]=0
\end{aligned}
\label{eq:4.10-1-1-1}
\end{equation}
for all $\f \phi $ test processes in the sense of Definition~\ref{def:stoch}. The case $s=0$ can be treated analogously by choosing $\f \psi(\tau):=\eta^\varepsilon_{{[0,t]}}(\tau)\f \phi (\tau)$ such that  $\f\psi(\tau)(0)=1$. 

\textit{Step 2: For any $\varphi\in C_{\sigma}^{2}(\mathcal{D})$,
 $M^{\varphi}$ is a continuous $(\mathcal{F}_{t})$-martingale.}

For $\varphi\in (C_{\sigma}^{2}\cap H^1_{\sigma})(\mathcal{D})$ we set 
\[
\begin{aligned}
    M^{\varphi}(t):={}&-\int_{\mathcal{D}}\varphi\cdot \f u(t)\de x 
    +\int_{0}^{t}\int_{\mathcal{D}}[\f u\otimes \f u]:\nabla\varphi+\frac{1}{2}[\mathcal{P}(\sigma^{2}\cdot\nabla)]^{2}\varphi\cdot \f u~\de x\de s\\&+\int_{0}^{t}\int_{\ov{\mathcal{D}}}(\nabla\varphi)_{\sym}:\de \mathfrak{R}(x)\de s.
\end{aligned}
\]

Fix now $0<s<t<T$, let $\eta\in L^{\infty}(\Omega,\mathcal{F}_{s})$, and set $\f \phi (r,x,\omega):=\varphi(x)\mathbb E[\eta|\mathcal{F}_{r}](\omega)$ for $r\in [0,T]$. Then $\f \phi $ is a continuous $(\mathcal{F}_{t})$-martingale, by Lemma~\ref{lem:mr-independent-initial-enlargement}, applied with \(\mathcal H=\mathbb R\), 
 implies that there is a $B_{\f \phi }\in L^2(\Omega\times (0,T);L_2(\mathfrak{U}, \R))$ such that 
\[
\f \phi (r,x)=\varphi(x)\mathbb E[\eta\mid\mathcal F_0]+\int_{0}^{r}\varphi(x)B_{\f \phi }dW.
\]
Hence, $\f\phi$ is a test process in the sense of Definition~\ref{def:stoch}, with $A_{\f\phi}=0$. We note that, since, for $t\ge s$, we have $\f \phi (t,x,\omega)=\varphi(x)\mathbb E[\eta|\mathcal{F}_{t}](\omega)=\varphi(x)\eta(\omega)$ , we have $B_{\f \phi }(r)=0$ for $r\ge s$. Therefore, we obtain that
\[
\begin{aligned} & \mathbb{E}\left[-\int_{\mathcal{D}}\varphi\eta\cdot \f u(t)\dx+\int_{\mathcal{D}}\varphi\eta\cdot \f u(s)\dx\right]\\
 & +\mathbb{E}\left[\int_{s}^{t}\int_{\mathcal{D}}[\f u\otimes\f  u]:\nabla\varphi\eta+\frac{1}{2}[\mathcal{P}(\sigma^{2}\cdot\nabla)]^{2}\varphi\eta\cdot \f u~\dx \de \tau\right]\\
 & +\mathbb{E}\left[\int_{s}^{t}\int_{\ov{\mathcal{D}}}(\nabla\varphi\eta)_{\sym}:\de \mathfrak{R}(x)\de \tau  \right]=0.
\end{aligned}
\]

By definition of $M^\varphi$, this implies
\[
\mathbb E[
(M^\varphi(t)-M^\varphi(s))\eta
]=0.
\]
Since $\eta\in L^\infty(\Omega,\mathcal F_s)$ is arbitrary, it follows that $M^\varphi$ is a continuous $(\mathcal F_t)$-martingale.

\textit{Step 3: Identification of the stochastic integral.}

Since $M^{\varphi}$ is a continuous $(\mathcal{F}_{t})$-martingale,  Lemma~\ref{lem:mr-independent-initial-enlargement}, applied with \(\mathcal H=\mathbb R\),
  implies that there exists a progressively measurable $B^{\varphi}\in L^2(\Omega\times (0,T);L_2(\mathfrak{U}, \R))$ such that, $\P$-almost surely, for all $t\ge0$,
\begin{align*}
 & M^{\varphi}(t)-M^{\varphi}(0)=\int_{0}^{t}B^{\varphi}(s)\de  W(s),
\end{align*}
that is,
\begin{align}
-\int_{\mathcal{D}}\varphi\cdot \f u(t)\dx+\int_{\mathcal{D}}\varphi\cdot \f u_0\
\dx +\int_{0}^{t}\int_{\mathcal{D}}[\f u \otimes \f u ]  :\nabla\varphi ~\de x\de s\label{eq:M-1}&\\+\int_{0}^{t}\int_{\mathcal{D}}\frac{1}{2}[\mathcal{P}(\sigma^{2}\cdot\nabla)]^{2}\varphi\cdot \f u ~\de x\de s+\int_{0}^{t}\int_{\ov{\mathcal{D}}}(\nabla\varphi)_{\sym}:\de \mathfrak{R}(x)\de s
 & =\int_{0}^{t}B^{\varphi}(s)\de W(s).\nonumber 
\end{align}
Note that we have used here that $u(0)=u_0$ which follows from weak continuity of $u$ and \eqref{eq:4.10-3}. It remains to identify the noise coefficient $B^{\varphi}(s)$. By \eqref{eq:M-1}, It\^o's product rule, and taking expectation, for any  process  of the form 
\[
\psi(t)=\psi_{0}+\int_{0}^{t}B_{\psi}\de W(s)
\]
with $B_{\psi}\in L^2(\Omega\times(0,T);L_{2}(\mathfrak{U},\R))$ and $\psi_0\in \R$, we have that
\begin{align*}
 & \mathbb{E}\left[\psi(t)\int_{\mathcal{D}}\f u (t)\cdot\varphi\dx-\psi(0)\int_{\mathcal{D}}\f u (0)\cdot\varphi\dx\right]\\
 & =\mathbb E\left[\int_{0}^{t}\psi(s)\left[\int_{\mathcal{D}}[\f u \otimes \f u ]:\nabla\varphi+\frac{1}{2}[\mathcal{P}(\sigma^{2}\cdot\nabla)]^{2}\varphi\cdot \f u \,\de x+\int_{\ov{\mathcal{D}}}(\nabla\varphi)_{\sym}:\de \mathfrak{R}(x)\right]\,\de s\right]\\
 & -\mathbb E\left[\int_{0}^{t}\Tr{
 B^{\varphi}(s)B_{\psi}(s)}\,\de s\right].
\end{align*} 
By the assumption \eqref{eq:4.10-3} in the localized form~\eqref{eq:4.10-1-1-1} with $s=0$ and $\f \phi=\varphi\psi$ as a test-process with $\varphi \in (C^2_{\sigma}\cap H^1_{\sigma})({\D}; \R^d)$, we also have that
\[
\begin{aligned} & \mathbb{E}\left[\psi(t)\int_{\mathcal{D}}\f u (t)\cdot\varphi \de x-\psi(0)\int_{\mathcal{D}}\f u _{0}\cdot\varphi \de x\right]\\
 & =\mathbb{E}\left[\int_{0}^{t}\psi(s)\left[\int_{\mathcal{D}}[\f u \otimes \f u ]:\nabla\varphi+\frac{1}{2}[\mathcal{P}(\sigma^{2}\cdot\nabla)]^{2}\varphi\cdot \f u ~\de x+\int_{\ov{\mathcal{D}}}(\nabla\varphi)_{\sym}:\de \mathfrak{R}(x)\right]\de s\right]\\
 & \quad +\mathbb{E}\left[\int_{0}^{t}\Tr{\int_{\mathcal{D}}\sigma^{1}\cdot\varphi B_{\psi}-\f u \cdot(\sigma^{2}\cdot\nabla)\varphi B_{\psi}\de x}\de s\right].
\end{aligned}
\]
Hence, we obtain
\begin{align*}
-\mathbb E\int_{0}^{t}\Tr{B^{\varphi}(s)B_{\psi}(s)}\,\de s & =\mathbb{E}\left[\int_{0}^{t}\Tr{\int_{\mathcal{D}}\sigma^{1}\cdot\varphi B_{\psi}-\f u \cdot(\sigma^{2}\cdot\nabla)\varphi B_{\psi}\de x}\de s\right].
\end{align*}
Now, for any progressively measurable $b_{\psi} \in L^\infty(\Omega \times [0,T]) $ and $ \mathbf{e}_k$ element of the orthonormal basis of $\mathfrak{U}$, setting $B^k_{\psi}(h):=b_{\psi}\langle h,\mathbf e_{k}\rangle _{\mathfrak{U}}$ for any $ h \in \mathfrak{U}$ we obtain that
\begin{align*}
-\mathbb E\int_{0}^{t}\langle B^{\varphi}(s), \mathbf e_{k}\rangle_{\mathfrak{U}}b_{\psi}(s)\,\de s & =\mathbb{E}\left[\int_{0}^{t}b_{\psi}(s)\left[\int_{\mathcal{D}}\langle\sigma^{1},\mathbf e_{k} \rangle_{\mathfrak{U}}\cdot\varphi-\f u \cdot (\langle \sigma^{2},\mathbf e_{k}\rangle_{\mathfrak{U}}\cdot\nabla)\varphi \de x\right]\de s\right]
\end{align*}
for all progressively measurable processes $b_{\psi} \in L^\infty(\Omega \times [0,T]) $ and all $k\in\N$. 
Since $\langle B^{\varphi}(s),\mathbf e_{k}\rangle_{\mathfrak{U}}$ is a progressively measurable process, the fundamental lemma implies that, $\P$-almost surely, for a.e.~$s\in[0,T]$ we have that
\begin{align*}
 & -\langle B^{\varphi}(s),\mathbf e_{k}\rangle_{\mathfrak{U}}=\int_{\mathcal{D}}\langle \sigma^{1},\mathbf e_{k}\rangle_{\mathfrak{U}}\cdot\varphi-\f u\cdot(\langle \sigma^{2}, \mathbf e_{k}\rangle _{\mathfrak{U}}\cdot\nabla)\varphi \de x 
\end{align*}
Hence, from \eqref{eq:M-1}, we conclude that, 
\begin{align*}
-\int_{\mathcal{D}}\varphi\cdot \f u(t)+ \int_{\mathcal{D}}\varphi\cdot \f u(0)+\int_{0}^{t}\int_{\mathcal{D}}[\f u\otimes \f u] & :\nabla\varphi+\frac{1}{2}[\mathcal{P}(\sigma^{2}\cdot\nabla)]^{2}\varphi\cdot \f u~\de x\de s\\
  +\int_{0}^{t}\int_{\ov{\mathcal{D}}}(\nabla\varphi)_{\sym}:\de \mathfrak{R}(x)\de s&=-\int_{0}^{t}\int_{\mathcal{D}}\sigma^{1}\cdot\varphi-\f u\cdot(\sigma^{2}\cdot\nabla)\varphi \de x\,\de W(s)
\end{align*}
holds for all $\varphi \in \C^2_{\sigma}(\ov\D;\R^d)$, all $t\in[0,T]$, $\P$-almost surely. As $\varphi$ was arbitrary, the assertion holds for all $\varphi \in C^2_{\sigma}(\ov{\D};\R^d)$ $\P$-almost surely. 

By the assumptions on $\sigma^1$ and $\sigma^2$, together with the embedding
$H^2(\mathcal D)\hookrightarrow L^\infty(\mathcal D)$, the stochastic integral
\[
\int_0^t\int_{\mathcal D}\sigma^1-u\cdot(\sigma^2\cdot\nabla)\,\mathrm dx\,\mathrm dW(s)
\]
is well defined with values in $(H^1_\sigma(\mathcal D))^*$, and its dual
pairing with $\varphi\in H^1_\sigma(\mathcal D)$ coincides with the scalar
stochastic integral appearing in \eqref{eq:4.10-3-1}.
Since $C^2_\sigma(\overline{\mathcal D};\mathbb R^d)$, endowed with the
$C^2$-norm, is separable, we may choose a countable dense subset
$\{\varphi_n\}_{n\in\mathbb N}$ and take the union of the corresponding
exceptional null sets, so that \eqref{eq:4.10-3-1} holds for every
$\varphi_n$ and every $t\in[0,T]$ outside a single null set.
For every $\omega$ outside this null set, all terms in
\eqref{eq:4.10-3-1} depend continuously on $\varphi$ with respect to the
$C^2$-norm, uniformly in $t\in[0,T]$, by the pathwise bounds on $\f u$ and
$\mathfrak R$ and the continuity of the stochastic integral in
$(H^1_\sigma(\mathcal D))^*$.
Hence, if $\varphi_n\to\varphi$ in
$C^2(\overline{\mathcal D};\mathbb R^d)$, we may pass to the limit in
\eqref{eq:4.10-3-1}.
Consequently, \eqref{eq:4.10-3-1} holds outside a single null set
simultaneously for all $t\in[0,T]$ and all
$\varphi\in C^2_\sigma(\overline{\mathcal D};\mathbb R^d)$.
\end{proof}

The existence of strong solutions locally-in-time as given in the following definition was shown for the incompressible Euler equations with additive noise in~\cite{strongsol2}, for the incompressible Euler equations with transport noise recently in~\cite{StronEulTransport} and for the incompressible Navier--Stokes equations with additive and transport noise in~\cite{strongNavBoth}. 
 But the latter two results are for the case of periodic boundary conditions on the torus and not for the case of a bounded domain with Navier-slip with friction and impermeability boundary conditions for the Navier--Stokes and the Euler equations, respectively,  as considered in the current article. 
\begin{definition}[Locally-in-time strong  solutions]\label{def:strong}
  
 A random variable $\tilde{\f{u}}$ and a stopping time $\mathfrak{t}$ are called a (local) strong solution to system~\eqref{eq:stochastic_euler_additive} provided
 \begin{enumerate}[label=\roman*)]
     \item the process $t \mapsto \fs{u}(t \wedge \mathfrak{t}, \cdot)$ is $\left(\mathcal{F}_t\right)$-adapted, such that  
    $$ \tilde{\f u} \in L^4(\Omega; \C([0,T];L^2_\sigma(\D))\cap L^2(0,T;\C^1(\D)\cap H^2(\D)))\,.
     $$
\item for all ${\varphi} \in C_{\sigma}^{2}\left(\ov{\D}\right)$ and all $t \geq 0$ there holds $\mathbb{P}$-almost surely
$$
\begin{aligned}
\int_{\D} \fs{u}(t \wedge \mathfrak{t}) \cdot \varphi \mathrm{d} x & =\int_{\D} \fs{u}(0) \cdot \varphi \mathrm{d} x-\int_0^{t \wedge \mathfrak{t}} \int_{\D}(\nabla \fs{u}) \fs{u} \cdot \varphi \mathrm{d} x+ \nu \alpha\int_{\partial \D} \tilde{\f u} \cdot \varphi \de S  \mathrm{d} s \\
&\quad - \int_0^{t\wedge\mathfrak{t}} \int_{\D} \frac{1}{2}\mathcal{P}(\sigma^2\cdot\nabla)\fs u \cdot \mathcal{P}(\sigma^2\cdot\nabla)\varphi +   \nu (\nabla \fs u )_{\sym}: \nabla \varphi \de x  \de s \\
&\quad +\int_0^{t \wedge \mathfrak{t}} \int_{\D} \varphi \cdot \sigma^1 + (\sigma^2\cdot\nabla) \fs u \cdot \varphi  \mathrm{d} x \mathrm{d} W(s);
\end{aligned}
$$
\item we have $\nabla \cdot  \fs{u}(\cdot \wedge \mathfrak{t})=0 $ $\mathbb{P}$-almost surely.
 \end{enumerate} \end{definition}
\begin{theorem}[Weak--strong uniqueness]\label{thm:weakstrong}
     Let $(\f u, E)$ be an energy-variational solution to the stochastic incompressible Euler equations in the sense of Definition~\ref{def:Envar} (resp.~to the stochastic Navier--Stokes equations in the sense of Definition~\ref{def:EnvarNav})  and let $\tu$ be a strong solution in the sense of Definition~\ref{def:strong} with $\nu=0$ (resp.~with $\nu>0$) on the same stochastic basis with the same initial values $ \tilde{\f u} (0)= \f u(0) $ and $E(0)=\frac{1}{2}\|\f u(0)\|_{L^2(\D)}^2$. Additionally, let $\sigma^1 \in L_2(\mathfrak{U}; H^1_{\sigma}(\D))$ and $ \sigma^2 \in L_2(\mathfrak{U}; H^2(\D)\cap H^1_{\sigma}(\D))$. 
     Then, for a.e. $(x,t)\in \D\times (0,T)$ it holds $ \f u(t \wedge \mathfrak{t}) = \tu (t\wedge \mathfrak{t}),  \, \mathbb{P}$-almost surely. 
 \end{theorem}

\begin{proof}[Proof of Theorem~\ref{thm:weakstrong}]
We prove the weak-strong uniqueness of solutions simultaneously for both systems. The only difference lies in the additional dissipative term multiplied by $\nu$, which vanishes for $\nu=0$.

 Let~$\tu$ be a local strong pathwise solution in the sense of Definition~\ref{def:strong}.
     A direct application of Itô's formula (in the Hilbert space version for $L_{\sigma}^2\left(\D \right)$, see \textit{e.g.}~\cite{DaPratoZabczyk}) shows that strong solutions satisfy the energy equality
\begin{equation}\label{eq:strongenin}
\begin{aligned}
     & \int_{\D}\frac{1}{2}|\tu(t)|^2 \mathrm{~d} x + \int_0^t \int_{\D} \nu |( \nabla \tilde{\f{u}} )_{\sym} |^2 \de x+ \nu \alpha \int_{\partial \D} |\tilde{\f u}|^2 \de S \de s \\&=\int_{\D}\frac{1}{2}|\tu(0)|^2 \mathrm{~d} x+ \int_0^t \int_{\D} \tu \cdot \sigma^1 \mathrm{~d} x \mathrm{~d} W(s)+\int_0^t\frac{1}{2}\|\sigma^1\|_{L_2\left(\mathfrak{U}, L^2_{\sigma}(\D)\right)}^2 \mathrm{~d} s 
      \\&\quad + \int_0^t\int_{\D} \frac{1}{2}[\mathcal P(\sigma^2\cdot\nabla)]^2\tilde{\f u} \cdot \tilde{\f u}  -  \frac{1}{2}[\mathcal P(\sigma^2\cdot\nabla)]^2\tilde{\f u} \cdot \tilde{\f u} \de x + \Tr{\int_{\D} \sigma^1 (\sigma^2\cdot\nabla) \tilde{\f u}\de x } \de s 
      \\&=\int_{\D}\frac{1}{2}|\tu(0)|^2 \mathrm{~d} x+ \int_0^t \int_{\D} \tu \cdot \sigma^1 \mathrm{~d} x \mathrm{~d} W(s)+\int_0^t\frac{1}{2}\|\sigma^1\|_{L_2\left(\mathfrak{U}, L^2_{\sigma}(\D)\right)}^2 \mathrm{~d} s 
\end{aligned}
\end{equation}
for all $t \in[0, \mathfrak{t}] $ and $\mathbb{P}$-almost surely. The last equation holds due to the cancellation of the Itô-Stratonovich correction term with  the additional term from the Itô formula. Both terms $ \sigma^1$ and $\sigma^2$ are orthogonal such that $\Tr{\int_{\D} \sigma^1 (\sigma^2\cdot\nabla) \tilde{\f u}\de x }=0$.

We consider the relative energy 
$$
\begin{aligned}
\mathbb{E}\left[\left( E- \frac{1}{2}\| \f u\|_{L^2(\D)}^2 + \frac{1}{2}\| \f u-\tu\|_{L^2(\D)}^2\right)( t )  \right] 
= \mathbb{E}\left[\left( E - \int_{\D} \f u \cdot \tu  + \frac{1}{2}\|\tu\|_{L^2(\D)}^2\right)(t)   \right] 
\end{aligned}
$$
We consider~\eqref{eq:envar} and~\eqref{eq:envarNav} with  $\f\varphi = \tilde{\f u}$ fulfilling $$ \de \tilde{\f u} +  \left( \mathcal P\nabla \cdot  (\tilde{\f u}\otimes \tilde{\f u})  - \frac{\nu}{2} \mathcal P \Delta \tu - \frac{1}{2}[\mathcal P(\sigma^2\cdot\nabla)]^2\tilde{\f u} \right)\de t = \mathcal P\left(\sigma^1 + (\sigma^2 \cdot \nabla) \tilde{\f u} \right) \de W \,,$$ 
such that $ A =    \frac{\nu}{2} \mathcal P \Delta \tu + \frac{1}{2}[\mathcal P(\sigma^2\cdot\nabla)]^2\tilde{\f u} -\mathcal P   \nabla \cdot  (\tilde{\f u}\otimes \tilde{\f u})$ and $ B = \mathcal P\left(\sigma^1 + (\sigma^2 \cdot \nabla) \tilde{\f u} \right)$.
We note that this is allowed due to the assumed regularity in~Definition~\ref{def:strong}.
Indeed, for the choices of $A$ and $B$, we find 
%
\begin{align*}
    &\mathbb{E}\left[\int_0^T\left\|\left[\mathcal P\nabla \cdot  (\tilde{\f u}\otimes \tilde{\f u})  - \frac{\nu}2 \mathcal P \Delta \tu - \frac{1}{2}[\mathcal P(\sigma^2\cdot\nabla)]^2\tilde{\f u}\right](\cdot \wedge \mathfrak{t})\right\|_{L^{2}(\D)}^2\de t \right]\\
    &\leq {} \| \tilde {\f u} \|_{L^4(\Omega; \C([0,T];L^2_\sigma (\D)))}\|\nabla \tilde{\f u}\|_{L^4(\Omega;L^2(0,T;\C^1(\D;\R^{d\times d})))} + \nu \| \tilde{\f u} \|_{L^2(\Omega;L^2(0,T;H^2(\D)))} \\&+ C \| \sigma^2 \|_{L_2(\mathfrak U;H^2(\D)\cap H^1_\sigma(\D))}^2 \| \tilde{\f u} \|_{L^2(\Omega;L^2(0,T;H^2(\D)))}
\intertext{as well as }
&\mathbb{E}\left[\int_0^T\left\|\left[\mathcal P\left(\sigma^1 + (\sigma^2 \cdot \nabla) \tilde{\f u} \right)\right](\cdot \wedge \mathfrak{t})\right\|_{L_2(\mathfrak U;
     H^1_{\sigma}(\D))}^2\de t \right]
\\
&\leq {} \| \sigma^1 \|_{L_2(\mathfrak U ; H^1_\sigma(\D))} + \| \sigma^2\|_{L_2(\mathfrak{U}; H^2(\D)\cap H^1_\sigma(\D))} \| \tilde{\f u} \|_{L^2(\Omega;L^2(0,T;H^2(\D)))}\,.
\end{align*}

From this, we infer that 
\begin{equation*}
            \begin{aligned}
     \left[ E  - \int_{\D} \f{u} \cdot \tu  \, \de x \right ] \Big|_{s-}^{t}&+  \int_s^t 
    \int_{\D} \nu (\nabla \f u)_{\sym} : (\nabla \f u - \nabla \tu) + \left[ \f{u} \otimes \f{u} \right] : \nabla \tu  \dx  \de \tau  \\ &+\int_s^t 2 \| (\nabla \tu)_{\sym,-} \|_{L^\infty(D,\R^{d\times d})}
    \left[ \frac{1}{2}\int_{\D} | \f{u}|^2 \dx - E  \right] 
      \de \tau 
     \\& 
    + \int_s^t\int_{\D}  \frac{1}{2}[\mathcal P(\sigma^2\cdot \nabla)] ^2 \tu \cdot \f u   \de x 
    + \nu \alpha  \int_{\partial \D} \f u \cdot (\f u - \tilde{\f u} ) \de S 
    \de \tau \\ & -\int_s^t  \int_{\D} ( \nabla \cdot  (\tilde{\f u}\otimes \tilde{\f u})  - \frac{\nu}2 \Delta \tu - \frac{1}{2}[\mathcal P(\sigma^2\cdot\nabla)]^2\tilde{\f u} ) \cdot \f{u} \de x  \de 
    \tau\\
       \leq{}& \int_s^t   \int_{\D}  (\f u -  \tu )\cdot \sigma^1 + \f u \cdot \mathcal P(\sigma^2 \cdot \nabla) \tu   - \f{u}\cdot \mathcal P(\sigma^1 + (\sigma^2 \cdot \nabla) \tilde{\f u} )  \de x \de W({\tau})\\ &  + \int_s^t \frac{1}{2}
 \|\sigma^1\|_{L_2(\mathfrak{U}, L^2_\sigma)}^2  -
 \Tr{\int_{\D} \sigma ^1 \cdot\mathcal P (\sigma^1 + (\sigma^2 \cdot \nabla) \tilde{\f u} ) \de x} \de \tau \\
 & + \int_s^t \Tr{ \int_{\D} \f u \cdot \mathcal P(\sigma^2 \cdot \nabla) (\sigma^1 + \mathcal P(\sigma^2 \cdot \nabla) \tilde{\f u} ) \de x }
 \de 
\tau  \,.   \end{aligned}
    \end{equation*}
Adding the energy inequality for $\tu$, \textit{i.e.,}~\eqref{eq:strongenin} implies after some cancellations 
\begin{equation}\label{eq:envarNavthree}
            \begin{aligned}
     \left[ E  - \int_{\D} \f{u} \cdot \tu  \, \de x + \frac{1}{2} \| \tu\|_{L^2(\D)}^2 \right ] \Big|_{s-}^{t}&+  \int_s^t 
    \int_{\D} \nu |(\nabla \f u)_{\sym} - (\nabla \tu)_{\sym}|^2+ \left[ \f{u} \otimes \f{u} \right] : \nabla \tu \, \dx \de \tau \\ &  + \nu \alpha\int_s^t  \| \f u - \tilde{\f u} \|_{L^2(\partial\D)}^2\de \tau  \\ &+\int_s^t 2 \| (\nabla \tu)_{\sym,-} \|_{L^\infty(D,\R^{d\times d})}
    \left[ \frac{1}{2}\int_{\D} | \f{u}|^2\de x  - E  \right] \de \tau \\
    &-\int_s^t  \int_{\D} ( \nabla \cdot  (\tilde{\f u}\otimes \tilde{\f u})   ) \cdot \f{u} \de x  \de 
    \tau 
       \leq{} 0 \,,
\end{aligned}
    \end{equation}
where we used beside the algebraic cancellations that both noise terms are assumed to be orthogonal, \textit{i.e.,} $ \Tr{\int_{\D}\f u \cdot (\sigma^2 \cdot \nabla) \sigma^1\de x } =0 $. 
    
Via the usual formula for the convection terms, we find 
\begin{align*}
    \int_{\D}\left[ \f{u} \otimes \f{u} \right] : \nabla \tu -( \nabla \cdot  (\tilde{\f u}\otimes \tilde{\f u})   ) \cdot \f{u} \de x
    =\int_{\D} [( \f u-\tu) \otimes (\f u-\tu)  ]: \nabla \tu \de x \,.
\end{align*}
Inserting this back into~\eqref{eq:envarNavthree} implies 
\begin{equation*}
            \begin{aligned}
     \left[ E  - \int_{\D} \f{u} \cdot \tu  \, \de x + \frac{1}{2} \| \tu\|_{L^2(\D)}^2 \right ] \Big|_{s-}^{t}&+  \int_s^t 
    \int_{\D} \nu |(\nabla \f u)_{\sym} - (\nabla \tu)_{\sym}|^2\de x  + \nu \alpha\int_s^t  \| \f u - \tilde{\f u} \|_{L^2(\partial\D)}^2\de \tau\\ & +\int_s^t  \int_{\D} [( \f u-\tu) \otimes (\f u-\tu)  ]: \nabla \tu \de x\, \de \tau  \\ &+\int_s^t 2 \| (\nabla \tu)_{\sym,-} \|_{L^\infty(\D,\R^{d\times d})}
    \left[ \frac{1}{2}\int_{\D} | \f{u}|^2 - E  \right]  \de
    \tau \leq 0 \,, 
\end{aligned}
    \end{equation*}
such that after 
 observing that 
$$
\begin{aligned}
\int_{\D}(\f v \otimes \f v ):\nabla \tu \de x  ={}& \int_{\D}(\f v \otimes \f v ):(\nabla \tu)_{\sym,+} \de x + \int_{\D}(\f v \otimes \f v ):(\nabla \tu )_{\sym,-} \de x
\\\geq{}& -2 \| ( \nabla \tu)_{\sym,-} \|_{L^\infty(\R^{d\times d})} \frac{1}{2} \int_{\D} \tr( \f v \otimes \f v ) \de x 
\\={}& -2 \| ( \nabla \tu)_{\sym,-} \|_{L^\infty(\R^{d\times d})} \frac{1}{2} \| \f v \|^2_{L^2(\D)} \,
\end{aligned}$$
with $\f v =\f u-\tu $, we find  
\begin{multline*}
     \left[ E  - \int_{\D} \f{u} \cdot \tu  \, \de x + \frac{1}{2} \| \tu\|_{L^2(\D)}^2 \right ] \Big|_{s-}^{t}\\
    \leq{}  \int_s^t2 \| (\nabla \tu_{\sym,-} \|_{L^\infty(D,\R^{d\times d})}
    \left[ E- \frac{1}{2}\int_{\D} | \f{u}|^2 + \frac{1}{2} \| \f u - \tu \|_{L^2(\D)}^2   \right]  \de
    \tau\,.
    \end{multline*}
Applying Gronwall's inequality pathwise, we infer that 
\begin{multline*}
    \left( E - \int_{\D} \f u \cdot \tu \de x + \frac{1}{2}\|\tu\|_{L^2(\D)}^2\right)(t)   
   \\ \leq \exp\left( 2\int_0^t \| (\nabla \tu)_{\sym,-} \|_{L^\infty(\D;\R^{d\times d})} \de \tau \right)
 \left( E(0) - \int_{\D} \f u(0) \cdot \tu (0) \de x + \frac{1}{2}\|\tu(0)\|_{L^2(\D)}^2\right)  = 0
\end{multline*}
$\mathbb P$-almost surely, which proves the assertion. 
 \end{proof}

\section{Acknowledgments}

 B.G. acknowledges support by the European Union (ERC, FluCo, grant agreement No. 101088488). Views and opinions
expressed are however those of the author(s) only and do not necessarily reflect those of the
European Union or of the European Research Council. Neither the European Union nor the
granting authority can be held responsible for them. 
 R.L.  acknowledges funding by the Deutsche Forschungsgemeinschaft (DFG, German Research Foundation) within SPP 2410 Hyperbolic Balance Laws in Fluid Mechanics: Complexity, Scales, Randomness (CoScaRa), project number 526018747  and the hospitality of the Erwin--Schrödinger institute in Vienna during the thematic semester on free boundary problems. 
 The authors thank Max Sauerbrey for insightful discussion.

\end{document}